\documentclass{amsart}

\usepackage{amssymb}
\usepackage{amsmath}
\usepackage{verbatim}
\usepackage{url}
\usepackage{graphicx}
\usepackage{color}

\usepackage[foot]{amsaddr}

\usepackage{mathrsfs}

\newcommand{\eps}{\varepsilon}

\newcommand{\opnm}{\operatorname}

\newcommand{\Bff}{\mathbf}
\newcommand{\BigO}{\mathcal{O}}

\newcommand{\supp}{\operatorname{supp}}
\newcommand{\shift}{\mathcal{S}}
\newcommand{\ccc}{\rho_1}
\newcommand{\sss}{\rho_2}

\DeclareFontFamily{U}{wncy}{}
\DeclareFontShape{U}{wncy}{m}{n}{<->wncyr10}{}
\DeclareSymbolFont{mcy}{U}{wncy}{m}{n}
\DeclareMathSymbol{\Sh}{\mathord}{mcy}{"58} 
\DeclareMathSymbol{\sha}{\mathord}{mcy}{"78} 
\newcommand{\comb}{\sha}

\newtheorem{theorem}{Theorem}[section]
\newtheorem{proposition}[theorem]{Proposition}
\newtheorem{corollary}[theorem]{Corollary}
\newtheorem{lemma}[theorem]{Lemma}

\theoremstyle{definition}
\newtheorem{definition}[theorem]{Definition}

\theoremstyle{remark}
\newtheorem{remark}[theorem]{Remark}
\newtheorem{example}[theorem]{Example}

\newtheorem*{acknowledgements}{Acknowledgements}
\newcommand{\fnJ}[1]{}

\numberwithin{equation}{section}
\numberwithin{figure}{section}

\newcommand{\ii}{\textnormal{i}}
\newcommand{\dd}{\textnormal{d}}
\newcommand{\ee}{\textnormal{e}}

\title[Dirac brushes]{Dirac brushes (or, the fractional Fourier transform of Dirac combs)}

\author{Joe Viola}

\email{Joseph.Viola@univ-nantes.fr}

\address{Universit\'e de Nantes, CNRS, Laboratoire de Math\'ematiques Jean Leray (LMJL - UMR 6629), F-44000 Nantes, France}

\begin{document}

\begin{abstract}
In analogy with the Poisson summation formula, we identify when the fractional Fourier transform, applied to a Dirac comb in dimension one, gives a discretely supported measure. We describe the resulting series of complex multiples of delta functions, and through either the metaplectic representation of $SL(2,\Bbb{Z})$ or the Bargmann transform, we see that the the identification of these measures is equivalent to the functional equation for the Jacobi theta functions. In tracing the values of the antiderivative in certain small-angle limits, we observe Euler spirals, and on a smaller scale, these spirals are made up of Gauss sums which give the coefficient in the aforementioned functional equation.
\end{abstract}

\maketitle

\setcounter{tocdepth}{1}
\tableofcontents

\section{Introduction}

\subsection{Setting and main results}

For $r > 0$, we consider the Dirac comb
\begin{equation}\label{eq_def_comb}
	\comb_r(x) = \sqrt{r}\sum_{k \in \Bbb{Z}} \delta(x - rk)
\end{equation}
as a distribution in the dual Schwartz space $\mathscr{S}'(\Bbb{R})$, meaning that the Dirac comb is defined via its action $f \mapsto \sqrt{r}\sum_{k\in\Bbb{Z}}f(rk)$ on rapidly decaying smooth functions. We write the Fourier transform as
\begin{equation}\label{eq_def_Fourier}
	\mathcal{F}f(\xi) = \ii^{-1/2}\int \ee^{-2\pi\ii x \xi} f(x)\,\dd x
\end{equation}
(with the convention $\ii^{-1/2} = \sqrt{-\ii} = \ee^{-\pi\ii/4}$); the normalizations for $\mathcal{F}$ and $\comb_r$ are both chosen to respect the metaplectic representation (Section \ref{ss_metapl}). 

The Poisson summation formula is, in the sense of distributions,
\begin{equation}\label{eq_Poisson_original}
	\mathcal{F} \comb_r(x) = \ii^{-1/2} \comb_{1/r}(x).
\end{equation}
Since $\mathcal{F}\comb_{r}(x) = (\ii / r)^{-1/2} \sum_{k\in\Bbb{Z}} \ee^{-2\pi\ii rkx}$, the Poisson summation formula implies remarkable cancellations for the oscillating functions $\ee^{-2\pi\ii rkx}$, in that their sum vanishes on $\Bbb{R}\backslash \frac{1}{r}\Bbb{Z}$ in the distributional sense. We can interpret this cancellation as arising after rotating the Dirac comb by an angle of $-\pi/2$ in phase space, which is the action of the Fourier transform. We can rotate a function in phase space by other angles using the fractional Fourier transform.

\begin{definition}\label{def_FrFT}
For $\alpha \in (-2, 2) \backslash \{0\}$, the fractional Fourier transform has integral kernel
\begin{equation}\label{eq_def_op_FrFT}
	\mathcal{F}^\alpha f(x) = \frac{1}{\sqrt{\ii \sin(\pi\alpha/2)}}\int \ee^{\frac{\pi\ii}{\sin(\pi\alpha/2)}(\cos(\pi\alpha/2)(x^2 + y^2) - 2xy)}f(y)\,\dd y.
\end{equation}
Our convention for the square root is that the real part should be positive. Naturally, $\mathcal{F}^0 f = f$, and one extends the definition to any $\alpha \in \Bbb{R}$ via $\mathcal{F}^2 f(x) = -\ii f(-x)$. Therefore $\mathcal{F}^{\alpha + 2k}f(x) = \ee^{-\frac{\pi\ii}{2}k}\mathcal{F}^\alpha f((-1)^k x)$ for any $k \in \Bbb{Z}$ and $\alpha \in (-2, 2)$.
\end{definition}

This work is devoted to answering the question of whether cancellations occur when rotating a Dirac comb in phase space by other angles. A geometrically natural answer, which happens to be correct, is that cancellations occur if and only if the rotation by $-\frac{\pi \alpha}{2}$ of the lattice $r\Bbb{Z} \times \frac{1}{r}\Bbb{Z}$ gives a discrete set when projected onto the first coordinate, which may be phrased as follows.

\begin{theorem}\label{thm_support}
For $\alpha \in \Bbb{R}$ and $r > 0$, the fractional Fourier transform (Definition \ref{def_FrFT}) of the Dirac comb \eqref{eq_def_comb}, $\mathcal{F}^\alpha \comb_r$, has discrete support if and only if $r \cos \frac{\pi\alpha}{2}$ and $\frac{1}{r}\sin \frac{\pi\alpha}{2}$ are linearly dependent over $\Bbb{Z}$. Otherwise, $\opnm{supp}\mathcal{F}^\alpha \comb_r = \Bbb{R}$.
\end{theorem}

In the case where $r \cos \frac{\pi\alpha}{2}$ and $\frac{1}{r}\sin \frac{\pi\alpha}{2}$ are linearly dependent over $\Bbb{Z}$, one can deduce (Remark \ref{rem_rho_s} below) that there exist $a,b\in\Bbb{Z}$ relatively prime such that
\begin{equation}\label{eq_def_rho_s}
	\ee^{\frac{\pi\ii \alpha}{2}} = \frac{1}{s}\left(\frac{a}{r} + \ii br\right), \quad s = \sqrt{(a/r)^2 + (br)^2}.
\end{equation}

We show that when $\mathcal{F}^\alpha \comb_r$ has discrete support, it is a Dirac comb multipled by an oscillating Gaussian factor $\ee^{-\pi\ii t x^2}$ and with a possible half-integer shift in phase space described in \eqref{eq_def_shift}. When $c,d\in\Bbb{Z}$ are such that $ad - bc = 1$, this Gaussian factor depends on
\begin{equation}\label{eq_def_t}
	t = \frac{ac}{r^2} + bdr^2.
\end{equation}
(The choice of $(c,d)$ is not unique but this does not change the result; see Remark \ref{rem_t_rep}.) Because of this Gaussian factor, the delta-functions which make up $\mathcal{F}^\alpha \comb_r$ point in many directions in the complex plane, so their sum resembles a brush instead of a comb. (See Figure \ref{fig_intro}.)

\begin{theorem}\label{thm_identify}
Let $\alpha \in \Bbb{R}$ and $r > 0$ be such that, as in \eqref{eq_def_rho_s}, $\ee^{\frac{\pi\ii}{2}\alpha} = \frac{1}{s}(a/r + \ii br)$ for $a, b \in \Bbb{Z}$ relatively prime and $s > 0$. Let $c, d \in\Bbb{Z}$ be such that $ad - bc = 1$, and let $t = ac/r^2 + bdr^2$ as in \eqref{eq_def_t}. Recall the definitions of the Dirac comb $\comb_r$ in \eqref{eq_def_comb} and the fractional Fourier transform $\mathcal{F}^\alpha$ in Definition \ref{def_FrFT}.

Then there exists some $\mu_0 \in \Bbb{C} \backslash \{0\}$ such that
\begin{equation}\label{eq_thm_identify}
	\mathcal{F}^\alpha \comb_r(x) = \mu_0 \ee^{-\pi\ii t x^2 - \pi\ii cd s x}\comb_{1/s}(x-\frac{ab}{2s}).
\end{equation}
\end{theorem}

We separately identify the coefficient $\mu_0$ in Section \ref{s_symmetries}, which corresponds to the functional equation for Jacobi theta functions as described in Section \ref{ss_theta}.

\begin{theorem}\label{thm_mu_8}
The constant $\mu_0$ in Theorem \ref{thm_identify} is an eighth root of unity which depends only on $a,b,c,d$ and whether $\alpha \in (-2, 2]$ modulo 8.
\end{theorem}

\begin{remark}
The terms $\pi \ii cd s x$ and $\frac{ab}{2s}$ in \eqref{eq_thm_identify} depend only on the parity of $cd$ and $ab$, since 
\[
	\ee^{2\pi\ii s x}\comb_{1/s}(x) = \comb_{1/s}(x-\frac{2}{2s}) = \comb_{1/s}(x).
\]
Furthermore, at least one of $ab$ and $cd$ must be even because $ad-bc = 1$.
\end{remark}

\begin{remark}\label{rem_sign_evenness}
We discuss several symmetries for $\mu$ in Section \ref{s_symmetries}. We note here that the dependence on whether $\alpha \in ]-2, 2]$ modulo 8 is an ambiguity regarding the sign only, which comes from the fact that $\ee^{\frac{\pi\ii}{2}(\alpha+4)} = \ee^{\frac{\pi\ii}{2}\alpha}$ while $\mathcal{F}^{\alpha + 4} = -\mathcal{F}^\alpha$.

We also note that $\mathcal{F}^\alpha \comb_r$ is an even distribution (regardless of whether the support is discrete or not) because $\comb_r$ is even and the reflection $f \mapsto f(-x)$ can be written as $f(-x) = \ii \mathcal{F}^2f(x)$ which commutes with $\mathcal{F}^\alpha$: 
\[
	\mathcal{F}^\alpha \comb_r(-x) = \ii \mathcal{F}^2\mathcal{F}^\alpha \comb_r(x) = \mathcal{F}^\alpha( \ii \mathcal{F}^2\comb_r)(x) = \mathcal{F}^\alpha \comb_r(x).
\]
\end{remark}

\begin{remark}
We express $\mathcal{F}^\alpha \comb_r(x)$ in Theorem \ref{thm_identify} using multiplication by $\ee^{-\pi \ii t x^2}$, primarily because this multiplication is a metaplectic operator. Another natural way to write these objects is as the sum of Dirac masses localized on (possibly half-)integer multiples of $\frac{1}{s}$, as described in detail in Remark \ref{rem_sum_expression}:
\begin{equation}\label{eq_identify_sum_even_odd}
	\mathcal{F}^\alpha \comb_r(x) = \mu_0 s^{-1/2}\sum_{k \in \Bbb{Z}+\{\frac{ab}{2}\}}\ee^{-\frac{\pi \ii t}{s^2} k^2 - \pi\ii cdk}\delta(x-\frac{k}{s}).
\end{equation}
\end{remark}

\begin{remark}\label{rem_rho_s}
Let us suppose that $r\cos \frac{\pi\alpha}{2}$ and $\frac{1}{r}\sin\frac{\pi\alpha}{2}$ are linearly dependent over $\Bbb{Z}$. To see that \eqref{eq_def_rho_s} holds, note that there exist $a',b'\in\Bbb{Z}$ not both zero (and, without loss of generality, relatively prime) such that 
\begin{equation}\label{eq_ab_cossin}
	b'r\cos \frac{\pi\alpha}{2} - \frac{a'}{r}\sin\frac{\pi\alpha}{2} = 0.
\end{equation}
Consequently, $s' = (a'/r - \ii b' r)\ee^{\frac{\pi\ii\alpha}{2}}$ is nonzero (as the product of a two nonzero complex numbers) and real (since the imaginary part vanishes by choice of $a', b'$). We obtain $\ee^{\frac{\pi\ii \alpha}{2}} = \frac{s'}{(a'/r)^2 + (b'r)^2}(a'/r + \ii b'r)$, allowing us to conclude that \eqref{eq_def_rho_s} holds with $(a,b) = (\operatorname{sgn} s')(a', b')$ and $s = |s'|$. (The fact that $s = \sqrt{(a/r)^2 + (br)^2}$ follows because $|\ee^{\frac{\pi\ii\alpha}{2}}| = 1$.)
\end{remark}

\subsection{Illustrations of some ``Dirac brushes''}\label{ss_brushes}

To give a concrete example of Theorem \ref{thm_identify}, let $a = b = r = 1$ and take $c = 0$, $d = 1$ so that $s = \sqrt{2}$ and $t = 1$. In this case, $\mu_0 = 1$ (see Section \ref{s_symmetries}), and one can compute that
\[
	\mathcal{F}^{1/2}\comb_1(x) = 2^{-1/4}\ee^{-\frac{\pi\ii}{8}}\sum_{k\in\Bbb{Z}} \gamma_k \delta\left(x-\frac{k+1/2}{\sqrt{2}}\right),
\]
where
\[
	\gamma_k = \ee^{-\frac{\pi\ii}{2}(k^2+k)} = \left\{\begin{array}{ll} 1,& k \equiv 0, 3~(\opnm{mod} 4) \\ -1, & k\equiv 1,2~(\opnm{mod} 4).\end{array}\right.
\]
We illustrate this distribution at the top of Figure \ref{fig_intro}, where we illustrate functions $f:\Bbb{R} \to \Bbb{C}$ as curves in $\Bbb{R} \times \Bbb{C} \cong \Bbb{R}^3$.

For comparison, we also illustrate $\mathcal{F}^{\alpha}\comb_1$ for $\cot(\pi\alpha/2) = 20/21$ and $200/201$ (chosen arbitrarily among fractions near $1$), allowing us to see that the ``diameter'' and spacing for these brushes, $s^{-1/2}$ and $s^{-1}$ respectively, tend to zero as $s = \sqrt{a^2 + b^2}$ becomes large.

Let us consider the third example, where $\ee^{\frac{\pi\ii\alpha_3}{2}} = (80401)^{-1/2}(200+201\ii)$. The result of Theorem \ref{thm_identify} with $c = -1$, $d = -1$, and $t = -401$, expressed as in \eqref{eq_identify_sum_even_odd}, is
\[
	\mathcal{F}^{\alpha_3}\comb_1(x) = \frac{\ee^{-\pi\ii/4}}{(80401)^{1/4}} \sum_{k\in\Bbb{Z}} \ee^{\pi\ii\frac{401}{80401}k^2}\delta(x - \frac{k}{\sqrt{80401}}).
\]
To illustrate that this intimidating expression is far from random (despite appearances in Figure \ref{fig_intro}), in Figure \ref{fig_intro_cty} we plot the antiderivative
\[
	\Pi_\alpha(X) = \int_0^X \mathcal{F}^\alpha \comb_1(x)\,\dd x,
\]
described in Definition \ref{def_antiderivative}, for the three examples in Figure \ref{fig_intro} as well as for $\alpha_4$ satisfying $\cot\frac{\pi\alpha_4}{2} = \frac{2000}{2001}$. We see that the graphs of $\Pi_{\alpha_3}$ and $\Pi_{\alpha_4}$ become quite similar to the graph of $\Pi_{\alpha_1}$ for $\alpha_1 = \pi/4$, up to a spiraling error discussed in Section \ref{ss_spirals}.

\begin{figure}
\centering
\includegraphics[width = \textwidth]{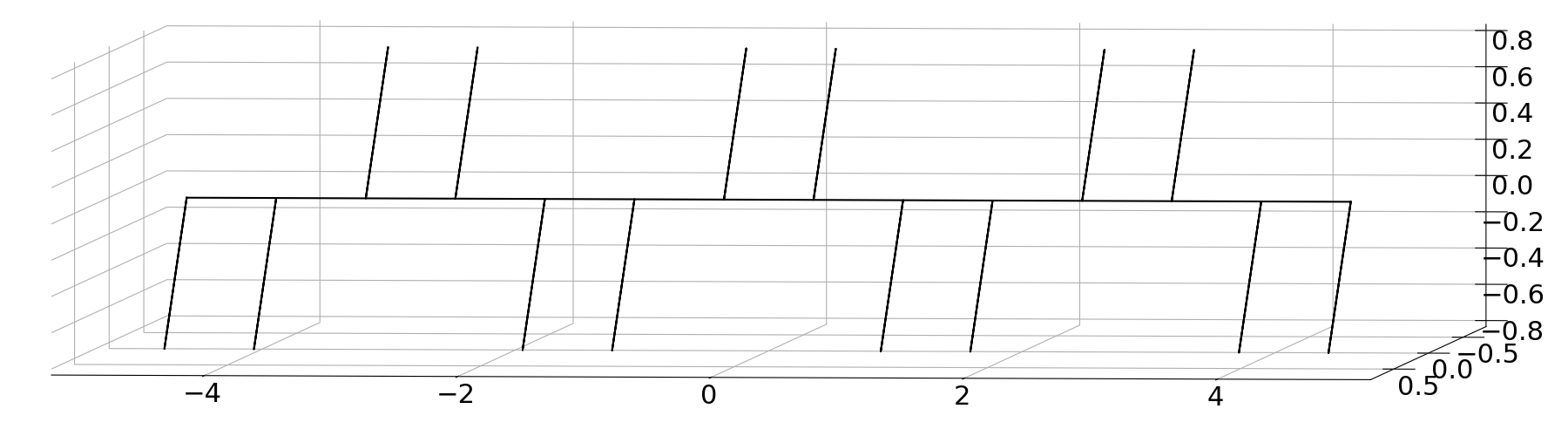}
\includegraphics[width = \textwidth]{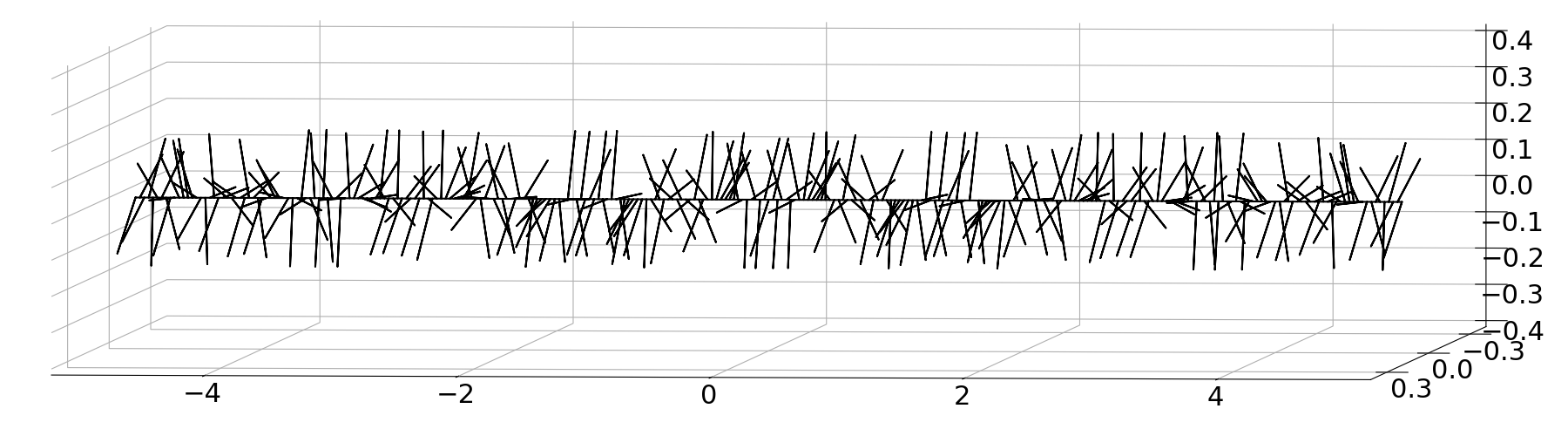}
\includegraphics[width = \textwidth]{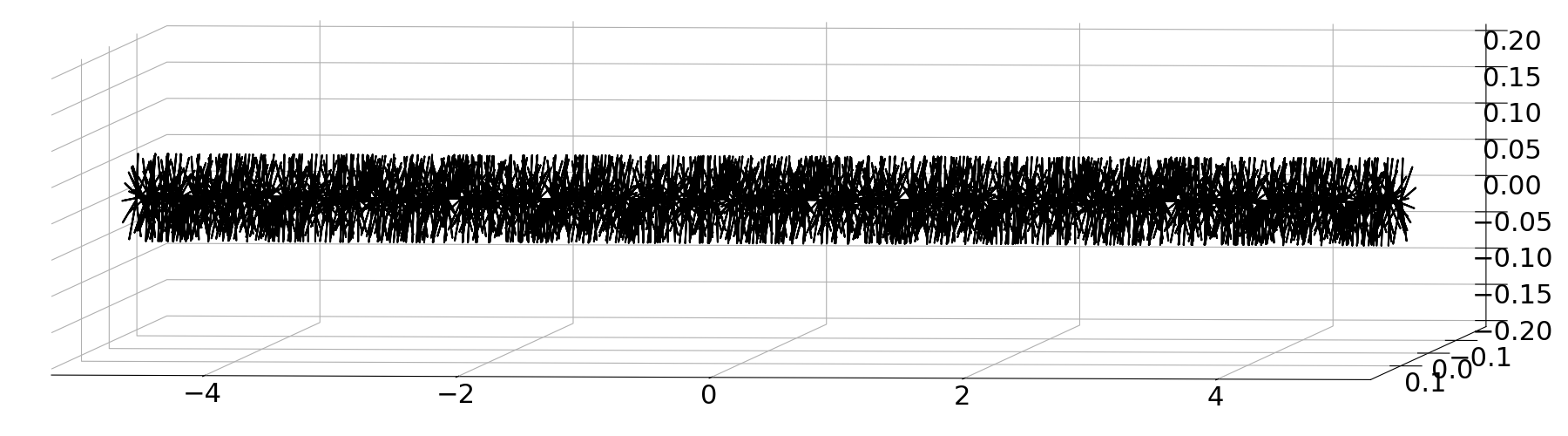}
\caption{Plots, where $f(x)$ is plotted as (horizontal, vertical, depth) = $(x, \Re f(x), \Im f(x))$, of $\mathcal{F}^{\alpha}\comb_1(x)$ where $\cot\frac{\pi\alpha}{2} = \frac{1}{1}, \frac{20}{21}, \frac{200}{201}$. Here, $\gamma \delta(x-x_0)$ is represented by the corresponding line segment from (horizontal, vertical, depth) = $(x_0, 0, 0)$ to $(x_0, \Re \gamma, \Im \gamma)$.\label{fig_intro}}
\end{figure}

\begin{figure}
\centering
\includegraphics[width = \textwidth]{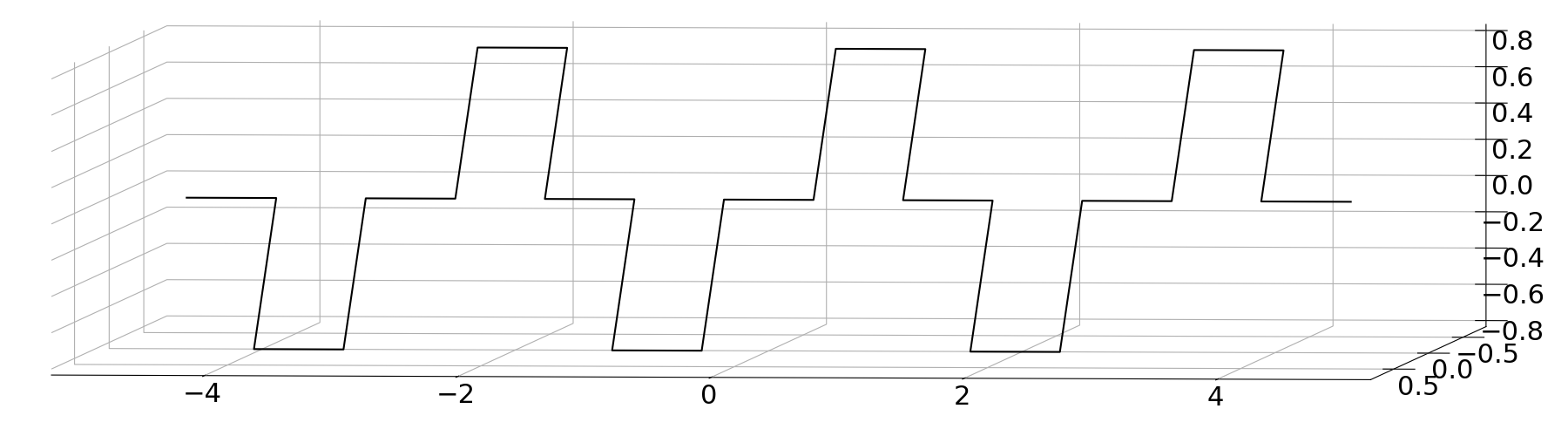}
\includegraphics[width = \textwidth]{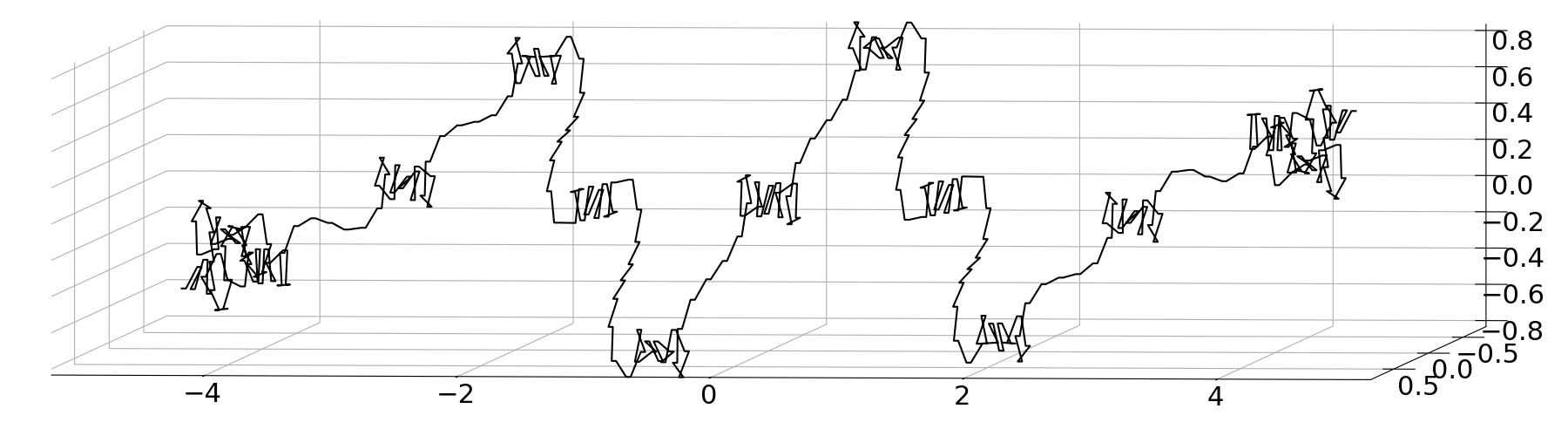}
\includegraphics[width = \textwidth]{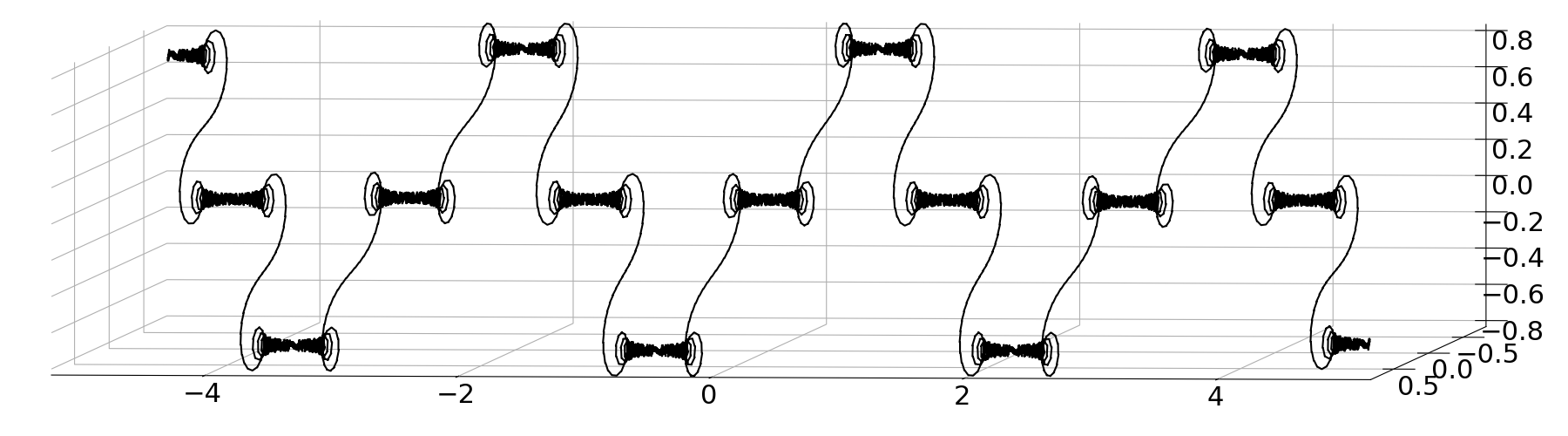}
\includegraphics[width = \textwidth]{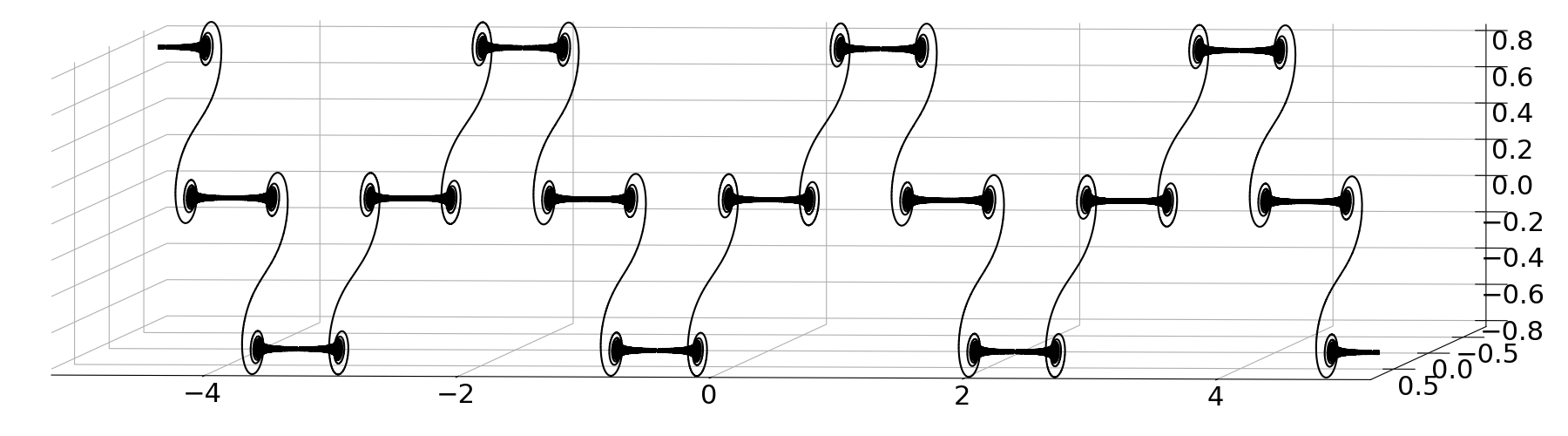}
\caption{Plots, where $f(x)$ is plotted as (horizontal, vertical, depth) = $(x, \Re f(x), \Im f(x))$, of the antiderivatives of the three measures in Figure \ref{fig_intro} plus, at the bottom, where $\cot\frac{\pi\alpha}{2} = \frac{2000}{2001}$.\label{fig_intro_cty}}
\end{figure}

\subsection{Context and plan of paper}

The author has often worked with (complexified) metaplectic operators \cite{Viola_2017} and happened upon this problem performing some numerical approximations. (It suffices to replace $\alpha$ by $\alpha - \ii \eps$ to obtain super-exponential convergence.) One immediately sees surprising and complicated structure behind the natural question of the interaction between a Dirac comb and the fractional Fourier transform, where the latter can be understood as the Schr\"odinger evolution of the quantum harmonic oscillator (Remark \ref{rem_Schro}).

The objects considered are central to the theory of modular forms, about which the author knows very little. The connection with the metaplectic representation is well-established; see for instance \cite{Lion_Vergne_1980}. The author would certainly be interested to find out to what extend this work runs in parallel to (or duplicates) works already established in that domain or elsewhere.

In studying this question, one is naturally led to prioritize the metaplectic representation and its associated (somewhat peculiar) normalizations, already seen in \eqref{eq_def_comb} and \eqref{eq_def_Fourier}. As shown in Section \ref{s_symmetries}, the symmetries giving Theorem \ref{thm_mu_8} (or, equivalently, the functional equation for the Jacobi theta functions) come from a relatively short list of metaplectic identites. Specifically, one commutes half-integer shifts \eqref{eq_shift_composition} and metaplectic operators \eqref{eq_Egorov} with straightforward linear algebra, and one can compose metaplectic operators (Lemma \ref{lem_M_compose}) via their associated linear maps with some occasional sign considerations. The particularities of the Dirac comb appear in the application of integer shifts \eqref{eq_shift_integer_comb}, the Poisson summation formula \eqref{eq_Poisson_original}, and the seemingly elementary fact that $\ee^{\pi\ii x^2 - \pi \ii x}\comb_1 = \comb_1$ because $k^2 \equiv k~(\opnm{mod} 2)$ for $k\in\Bbb{Z}$. Geometrically, this says that a shear in phase space \eqref{eq_def_canon_Wt} acts on the Dirac comb in the same way as a half-integer momentum shift, and this fact alone accounts for all half-integer shifts in the proof of Theorem \ref{thm_mu_8}.

The plan of this paper is as follows. In Section \ref{s_direct} which follows, we introduce shifts in phase space and their interaction with the Dirac comb and the fractional Fourier transform, and then we prove Theorems \ref{thm_support} and \ref{thm_identify}. We introduce the metaplectic point of view in Section \ref{s_metapl}; this allows us to rephrase Theorem \ref{thm_identify} in terms of the Dirac comb being an eigenfunction of certain shifted metaplectic operators. We also obtain the relation between Theorem \ref{thm_identify} and the functional equation for the Jacobi theta functions. In Section \ref{s_symmetries} we use symmetries given by the metaplectic representation to obtain the coefficient $\mu$ following the algorithm in \cite{Mumford_1983}. In Section \ref{s_Bargmann} we show that, with the Bargmann transform, one can deduce Theorem \ref{thm_identify} from the functional equation, and we also can establish a type of weak boundedness used later to show weak continuity of the family $\{\mathcal{F}^\alpha \comb_r\}_{\alpha \in\Bbb{R}}$. In Section \ref{s_further} we consider periodicity and parity of $\mathcal{F}^\alpha\comb_r$. Section \ref{s_continuity} concerns continuity of $\{\mathcal{F}^\alpha \comb_1\}_{\alpha \in \Bbb{R}}$: this family is wildly divergent in absolute value in the sense of complex measures, yet continuous when ``smoothed'' by the harmonic oscillator raised to any power less than $-1/2$. Its antiderivative is not locally uniformly continuous, but we can understand some of its behavior as $\alpha \to 0$ as approaching Fresnel integrals, and within these Fresnel integrals we can analyze repeating patterns which give the coefficient $\mu$ in terms of Gauss sums. The last section is devoted to the question, not answered here, of what happens when Theorem \ref{thm_support} gives $\opnm{supp} \mathcal{F}^\alpha \comb_1 = \Bbb{R}$: numerically, it certainly seems that the antiderivative converges, but to a rough and possibly self-similar function.

\begin{acknowledgements}
This work began with some conversations during the workshop ``Atelier d'Analyse Harmonique 2017'' of the ERC FAnFArE. In particular, the author is grateful to Yves Meyer for discussions and encouragement at the outset of this study. The author is greatly indebted to Francis Nier for numerous discussions and helpful references. The author also acknowledges the support of the R\'egion Pays de la Loire through the project EONE (\'Evolution des Op\'erateurs Non-Elliptiques).

\end{acknowledgements}

\section{Direct proof of Theorems \ref{thm_support} and \ref{thm_identify}}\label{s_direct}

\subsection{Shifts in phase space and the fractional Fourier transform}

This paper is essentially an application of the theory of metaplectic operators applied to shifts in phase space. Behind all these objects, of course, is the language of symplectic linear algebra, but we minimize the use of this language in the hopes of giving a more approchable presentation.

We define the phase-space shift for $(x_0, \xi_0) \in \Bbb{R}^2$ as
\begin{equation}\label{eq_def_shift}
	\shift_{(x_0, \xi_0)}f(x) = \ee^{-\pi \ii x_0 \xi_0 + 2\pi\ii \xi_0 x}f(x-x_0).
\end{equation}
One can readily verify the composition law
\begin{equation}\label{eq_shift_composition}
	\shift_{(x_1, \xi_1)}\shift_{(x_2,\xi_2)} = \ee^{\pi\ii(\xi_1 x_2 - x_2\xi_1)}\mathcal{S}_{(x_1+x_2, \xi_1+\xi_2)}.
\end{equation}

When $(jr, k/r) \in r\Bbb{Z} \times \frac{1}{r}\Bbb{Z}$, the Dirac comb $\comb_r$ in \eqref{eq_def_comb} is an eigenfunction of $\mathcal{S}_{(jr, k/r)}$ with eigenvalue $1$ or $-1$:
\begin{equation}\label{eq_shift_integer_comb}
	\mathcal{S}_{(jr,k/r)}\comb_r = \ee^{-\pi\ii jk}\comb_r, \quad \forall j, k \in \Bbb{Z}.
\end{equation}
Moreover, any distribution $f$ such that $\mathcal{S}_{(r, 0)}f = \mathcal{S}_{(0,1/r)}f = f$ must be a multiple of the Dirac comb $\comb_r$, a classical fact which we recall in Lemma \ref{lem_id_u1} below. In this way, \eqref{eq_shift_integer_comb} defines $\comb_r$ up to constants.

The fractional Fourier transform acts on shifts, via conjugation, by rotating the shift: defining
\begin{equation}\label{eq_def_F_canon}
	\Bff{F}_\alpha = \left(\begin{array}{cc} \cos \frac{\pi\alpha}{2} & \sin \frac{\pi\alpha}{2} \\ -\sin \frac{\pi\alpha}{2} & \cos\frac{\pi\alpha}{2}\end{array}\right),
\end{equation}
one has
\begin{equation}\label{eq_Egorov_F}
	\mathcal{F}^\alpha \mathcal{S}_{(x_0, \xi_0)} = \mathcal{S}_{\Bff{F}_\alpha(x_0, \xi_0)}\mathcal{F}^\alpha, \quad \forall (x_0, \xi_0) \in \Bbb{R}^2.
\end{equation}
This is but one example of the action of a metaplectic operator acting on shifts, discussed in more generality in Section \ref{ss_metapl}, but it is sufficient to prove Theorems \ref{thm_support} and \ref{thm_identify}.

\begin{example}
For the Fourier transform, the standard identity
\[
	\left(\mathcal{F}f(\cdot - x_0)\right)(\xi) = \ee^{-2\pi \ii x_0 \xi}(\mathcal{F}f)(\xi)
\]
can be written in terms of \eqref{eq_Egorov_F} with $\alpha = 1$ as
\[
	\mathcal{F}\mathcal{S}_{(x_0, 0)} = \mathcal{S}_{(0, -x_0)}\mathcal{F}.
\]
The presence of the factor $\ee^{-\pi\ii x_0\xi_0}$ in the definition \eqref{eq_def_shift} significantly simplifies expressions like $\mathcal{F}\shift_{(x_0, \xi_0)} = \shift_{(\xi_0, -x_0)}\mathcal{F}$ because constants arising from changes of variables cancel.
\end{example}

\subsection{Proof of Theorems \ref{thm_support} and \ref{thm_identify}}\label{ss_proofs_elementary}

We now prove Theorems \ref{thm_support} and \ref{thm_identify} using the Egorov relation \eqref{eq_Egorov_F} for the fractional Fourier transform. We also use \eqref{eq_shift_integer_comb} and the idea behind the classical Lemma \ref{lem_id_u1}: when $v \in \mathcal{D}'(\Bbb{R})$ is a distribution and $f \in C^\infty(\Bbb{R})$ is smooth with simple isolated zeros, $fv = 0$ implies that $v$ is a series of delta-functions supported on $\{f = 0\}$.

\begin{proof}[Proof of Theorem \ref{thm_support}]
To simplify notation, let us write
\[
	\rho_1 = \cos \frac{\pi\alpha}{2}, \quad \rho_2 = \sin\frac{\pi\alpha}{2}
\]
so that $\ee^{\frac{\pi\ii\alpha}{2}} = \ccc + \ii \sss.$

Applying \eqref{eq_Egorov_F} to \eqref{eq_shift_integer_comb} gives, for any $j, k\in\Bbb{Z}$,
\begin{equation}\label{eq_thm_support_proof}
	\begin{aligned}
	\mathcal{F}^\alpha \comb_r(x) &= \mathcal{F}^\alpha \ee^{\pi\ii jk}\mathcal{S}_{(jr, k/r)}\comb_r(x)
	\\ &= \ee^{\pi\ii jk}\mathcal{S}_{(jr\ccc + k\sss/r, -jr\sss + k\ccc/r)}\mathcal{F}^\alpha \comb_r (x)
	\end{aligned}
\end{equation}
For any distribution $v$, $\opnm{supp}\shift_{(x_0, \xi_0)} v = \opnm{supp} v + \{x_0\}$. Therefore for every $j,k\in\Bbb{Z}$, 
\[
	\opnm{supp}\mathcal{F}^\alpha \comb_r = \opnm{supp}\mathcal{F}^\alpha \comb_r + \{jr\ccc + k\sss/r\}.
\]
If $r\cos \frac{\pi\alpha}{2}$ and $\frac{1}{r}\sin\frac{\pi\alpha}{2}$ are linearly independent over $\Bbb{Z}$, we can make 
\[
	jr\rho_1 + \frac{k}{r}\rho_2 = jr\cos \frac{\pi\alpha}{2} + \frac{k}{r}\sin\frac{\pi\alpha}{2}
\]
arbitrarily small by Dirichlet's Approximation Theorem (for example, \cite[Sect.\ 1.2]{Meyer_1972}. Furthermore, $\{jr\rho_1 + k\rho_2/r\}_{(j,k)\in\Bbb{Z}}$ is a subgroup of $\Bbb{R}$ that now contains arbitrarily small elements, so it and $\opnm{supp} \mathcal{F}^\alpha \comb_r$ are dense. Since $\opnm{supp} \mathcal{F}^\alpha \comb_r$ is closed, we have shown that $\opnm{supp}\mathcal{F}^\alpha \comb_r = \Bbb{R}$.

On the other hand, if $r\ccc = r\cos \frac{\pi\alpha}{2}$ and $\sss/r = \frac{1}{r}\sin\frac{\pi\alpha}{2}$ are linearly dependent over $\Bbb{Z}$, as in Remark \ref{rem_rho_s} we may choose $a,b\in\Bbb{Z}$ relatively prime and $s > 0$ such that $\ee^{\frac{\pi\ii\alpha}{2}} = \frac{1}{s}(a/r + \ii b r)$ as in \eqref{eq_def_rho_s}. Replacing $\ccc$ with $a/(rs)$ and $\sss$ with $br/s$ and setting $(j,k) = (b,a)$ in \eqref{eq_thm_support_proof} gives $jr\ccc + k\sss/r = 0$ and
\[
	-jr\sss + \frac{k\ccc}{r} = \frac{1}{s}(b^2r^2 + \frac{a^2}{r^2}) = s,
\]
so
\[
	\mathcal{F}^\alpha \comb_r = \ee^{\pi \ii ab}\mathcal{S}_{(0, -s)}\mathcal{F}^\alpha \comb_r = \ee^{-2\pi\ii(sx-\frac{ab}{2})}\mathcal{F}^\alpha \comb_r.
\]
Consequently,
\[
	\opnm{supp} \mathcal{F}^\alpha \comb_r \subseteq \{x \::\: \ee^{-2\pi\ii (sx - \frac{ab}{2})} = 1\} = \frac{1}{s}(\Bbb{Z} + \{\frac{ab}{2}\}),
\]
which of course implies that $\opnm{supp}\mathcal{F}^\alpha \comb_r$ is discrete. In fact, because the zeros of $\ee^{2\pi \ii (sx - \frac{ab}{2})} - 1$ are simple, we can conclude from \cite[Thm.~3.1.16]{Hormander_ALPDO1} that $\mathcal{F}^\alpha \comb_r(x)$ is a delta-function when restricted to a sufficiently small neighborhood of any $\frac{1}{s}(k+\frac{ab}{2})$; see the proof of Lemma \ref{lem_id_u1}, which is taken from \cite[Sec.~7.2]{Hormander_ALPDO1}. Therefore for some sequence $\{\gamma_k\}_{k\in\Bbb{Z}}$ of complex numbers,
\begin{equation}\label{eq_series_gammak}
	\mathcal{F}^\alpha \comb_r(x) = \sum_{k\in\Bbb{Z}} \gamma_k \delta\left(x-\frac{1}{s}\left(k + \frac{ab}{2}\right)\right).
\end{equation}
\end{proof}

\begin{proof}[Proof of Theorem \ref{thm_identify}]
We continue our analysis using \eqref{eq_thm_support_proof} and \eqref{eq_series_gammak}. When $\ee^{\frac{\pi\ii\alpha}{2}} = \frac{1}{s}(a/r + \ii br)$ for $a,b\in\Bbb{Z}$ relatively prime and $s > 0$, we minimize the shift in space in \eqref{eq_thm_support_proof} by setting $j = d$ and $k = -c$ chosen such that $ad - bc = 1$. Consequently, $jr\ccc + k\sss/r = 1/s$. Then, with $t = ac/r^2 + bdr^2$ as in \eqref{eq_def_t} and using the expression \eqref{eq_series_gammak},
\[
	\begin{aligned}
	\mathcal{F}^\alpha \comb_r(x) &= \ee^{-\pi\ii cd}\mathcal{S}_{(1/s, -t/s)}\mathcal{F}^\alpha \comb_r(x)
\\ &= \ee^{-\pi\ii cd + \pi\ii \frac{t}{s^2} - 2\pi\ii \frac{t}{s}x}\mathcal{F}^\alpha \comb_r(x-1/s)
	\\ &= \ee^{-\pi\ii cd + \pi\ii \frac{t}{s^2} - 2\pi\ii \frac{t}{s}x} \sum_{k\in\Bbb{Z}}\gamma_k\delta(x-\frac{k+1+ab/2}{s})
	\\ &= \sum_{k\in\Bbb{Z}}\ee^{-\pi\ii cd + \pi\ii \frac{t}{s^2} - 2\pi\ii \frac{t}{s}x}\gamma_{k-1}\delta(x-\frac{k+ab/2}{s})
	\\ &= \sum_{k\in\Bbb{Z}}\ee^{-\pi\ii cd + \pi\ii \frac{t}{s^2} - 2\pi\ii \frac{t}{s^2}(k+ab/2)}\gamma_{k-1}\delta(x-\frac{k+ab/2}{s}).
	\end{aligned}
\]
We conclude that, in \eqref{eq_series_gammak}, the coefficients $\{\gamma_k\}_{k\in\Bbb{Z}}$ obey
\[
	\begin{aligned}
	\gamma_k &= \ee^{-\pi\ii cd + \pi\ii \frac{t}{s^2} - 2\pi\ii \frac{t}{s^2}(k+ab/2)}\gamma_{k-1}
	\\ &= \ee^{\pi\ii S_k}\gamma_0
	\end{aligned}
\]
when
\[
	\begin{aligned}
	S_k &= \sum_{j=1}^k \left(-cd + \frac{t}{s^2}(1 - 2j - ab)\right)
	\\ &= -cd k + \frac{t}{s^2}(k - k(k+1)-abk)
	\\ &= -cds \frac{k+ab/2}{s} - t\left(\frac{k+ab/2}{s}\right)^2 + cds\frac{ab}{2s} + t\left(\frac{ab}{2s}\right)^2.
	\end{aligned}
\]

We recognize that
\[
	\gamma_k = \ee^{-\pi\ii cds x_k - \pi\ii t x_k^2 + \pi \ii cds x_0 + \pi\ii t x_0^2}\gamma_0, \quad x_k = \frac{k+ab/2}{s}.
\]
Plugging back into \eqref{eq_series_gammak}, we obtain
\[
	\begin{aligned}
	\mathcal{F}^\alpha \comb_r(x) &= \ee^{\pi\ii cds x_0 + \pi\ii t x_0^2}\gamma_0\sum_{k\in\Bbb{Z}} \ee^{-\pi\ii cds x_k - \pi\ii t x_k^2} \delta(x-\frac{k+ab/2}{s})
	\\ &= \mu_0 \ee^{-\pi\ii cds x - \pi\ii t x^2}\comb_{1/s}(x-\frac{ab}{2s})
	\end{aligned}
\]
with $\mu_0 = s^{1/2}\ee^{\pi\ii cds x_0 + \pi\ii t x_0^2}\gamma_0$.\fnJ{check the computation one last time.} This proves Theorem \ref{thm_identify}.
\end{proof}

\section{Using the metaplectic representation of $SL(2,\Bbb{R})$}\label{s_metapl}

\subsection{Introduction to metaplectic operators}\label{ss_metapl}
The fractional Fourier transform is but one of many metaplectic operators (we restrict our discussion to dimension one except for Section \ref{ss_dim_high}), a closed subgroup of unitary operators on $L^2(\Bbb{R})$ which are automorphisms of $\mathscr{S}(\Bbb{R})$ and $\mathscr{S}'(\Bbb{R})$ which preserve linear forms in $(x,D_x) = (x, \frac{1}{2\pi\ii} \frac{\dd }{\dd x})$. (See, for instance, \cite{Leray_1981}.) They are in two-to-one correspondence with matrices with real entries and determinant one,
\begin{equation}\label{eq_def_M_Bff}
	\Bff{M} = \left(\begin{array}{cc} a & b \\ c & d \end{array}\right) \in SL(2, \Bbb{R}),
\end{equation}
and when $b \neq 0$ (we postpone the case $b = 0$ for Remark \ref{rem_bzero}), we may write the integral kernel for $\mathcal{M} = \mathcal{M}(\Bff{M})$ as
\begin{equation}\label{eq_M_kernel}
	\mathcal{M}f(x) = \frac{1}{\sqrt{\ii b}}\int_{\Bbb{R}} \ee^{\frac{\pi\ii}{b}(dx^2 - 2xy + ay^2)}f(y)\,\dd y.
\end{equation}
We adopt the convention that the sign of the square root is such that $\Re \sqrt{\ii b} > 0$, but $-\mathcal{M}$ is also a metaplectic operator, and $\mathcal{M}$ and $-\mathcal{M}$ are the only metaplectic operators associated with $\Bff{M}$.

Our principal tool linking metaplectic operators and shifts in phase space is the (exact) Egorov relation
\begin{equation}\label{eq_Egorov}
	\mathcal{M}\mathcal{S}_{(x_0, \xi_0)} = \mathcal{S}_{\Bff{M}(x_0, \xi_0)}\mathcal{M}.
\end{equation}
It is elementary to confirm this from the integral kernel \eqref{eq_M_kernel}; we omit the moderately lengthy computation.

The metaplectic group is generated by $\mathcal{F}$, scaling, and multiplication by Gaussians with imaginary exponents; we define these two last families now.

For $r > 0$,
\begin{equation}\label{eq_def_op_Vr}
	\mathcal{V}_r f(x) = \sqrt{r} f(rx), \quad r > 0
\end{equation}
is clearly unitary on $L^2(\Bbb{R})$. It also induces, in the sense of \eqref{eq_Egorov}, the transformation
\begin{equation}\label{eq_def_canon_Vr}
	\Bff{V}_r = \left(\begin{array}{cc} 1/r & 0 \\ 0 & r\end{array}\right).
\end{equation}
It is, in fact, through the scaling $\mathcal{V}_r$ that we have defined the comb $\comb_r$ in \eqref{eq_def_comb}: because a $\delta$-function is homogeneous of degree $-1$,
\begin{equation}\label{eq_def_comb_metapl}
	\comb_r(x) := \mathcal{V}_{1/r}\comb_1(x) = r^{-1/2}\sum_{k\in\Bbb{Z}}\delta(x/r - k) = r^{1/2}\sum_{k\in\Bbb{Z}}\delta(x-rk).
\end{equation}

For $t \in \Bbb{R}$, we also have
\begin{equation}\label{eq_def_op_Wt}
	\mathcal{W}_t f(x) = \ee^{\pi\ii t x^2}f(x),
\end{equation}
which induces the transformation
\begin{equation}\label{eq_def_canon_Wt}
	\Bff{W}_t = \left(\begin{array}{cc} 1 & 0 \\ t & 1\end{array}\right).
\end{equation}

\begin{remark}\label{rem_bzero} We can now describe $\mathcal{M}(\Bff{M})$ when the upper-right entry $b$ of $\Bff{M}$ is zero. In this case we must have, for some choice of sign making $a > 0$,
\[
	\Bff{M} = \pm \left(\begin{array}{cc} a & 0 \\ c & 1/a\end{array}\right).
\]
When the sign is positive,
\[
	\mathcal{M}(\Bff{M}) = \mathcal{V}_{1/a} \mathcal{W}_{ac}:f(x)\mapsto a^{-1/2}\ee^{\pi\ii\frac{c}{a}x^2}f(x/a).
\]
When the sign is negative, since $\mathcal{F}^2f(x) = -\ii f(-x)$ we can write
\[
	\mathcal{M}(\Bff{M}) = \mathcal{F}^2 \mathcal{V}_{1/a} \mathcal{W}_{ac}:f(x)\mapsto -\ii a^{-1/2}\ee^{\pi \ii \frac{c}{a}x^2}f(-x/a).
\]
The choice to use $\mathcal{F}^2$ instead of $\mathcal{F}^{-2} = -\mathcal{F}^2$ is completely arbitrary and has no deeper significance.
\end{remark}

\begin{remark}\label{rem_unitary}
If $K_{\Bff{M}}(x,y)$ is the integral kernel of $\mathcal{M}(\Bff{M})$, notice that $\overline{K_{\Bff{M}}(x,y)} = K_{\Bff{M}^{-1}}(y,x)$. For this reason, when $\langle \varphi, f\rangle = \varphi(f)$ is the sesquilinear dual bracket between $\varphi\in\mathscr{S}'(\Bbb{R})$ and $f\in \mathscr{S}(\Bbb{R})$, we still have
\[
	\langle \mathcal{M}(\Bff{M})\varphi, f\rangle = \langle \varphi, \mathcal{M}(\Bff{M})^{-1}f\rangle.
\]
Therefore (so long as both members are in $L^2(\Bbb{R})$ or one is in $\mathscr{S}'(\Bbb{R})$ and one is in $\mathscr{S}(\Bbb{R})$) we treat all inner products like $L^2(\Bbb{R})$ across which the ``unitary'' operators $\mathcal{M}(\Bff{M})$ and $\shift_{(x_0, \xi_0)}$ may be passed.
\end{remark}

\begin{remark}\label{rem_Schro}
The generators of the metaplectic representation are the Schr\"odigner evolutions of certain degree-two polynomials in 
\[
	(x, D_x) = (x, \frac{1}{2\pi\ii}\frac{\dd}{\dd x}).
\]
The infinitesimal generator of the family of fractional Fourier transforms $\{\mathcal{F}^\alpha\}_{\alpha \in \Bbb{R}}$ is the quantum harmonic oscillator
\begin{equation}\label{eq_def_Q0}
	Q_0 = \pi(x^2 + D_x^2) = \pi\left( x^2 - \frac{1}{4\pi}\frac{\dd^2}{\dd x^2}\right),
\end{equation}
meaning that
\begin{equation}\label{eq_def_Schro_FrFT}
	\mathcal{F}^\alpha = \exp\left(-\frac{\pi\ii}{2}\alpha Q_0\right)
\end{equation}
in the sense that
\[
	\partial_\alpha \mathcal{F}^\alpha f = -\frac{\pi\ii}{2}Q_0(\mathcal{F}^\alpha f)
\]
for all $f \in \mathscr{S}(\Bbb{R})$.

It is quite elementary to see that the infinitesimal generator of $\mathcal{W}_t$ is $x^2$ in the same sense, since
\[
	\partial_t \ee^{\pi\ii t x^2}f(x) = 2\pi\ii x^2(\ee^{\pi\ii t x^2}f(x)).
\]
The scaling $\mathcal{V}_r$ can be written using 
\[
	R_0 = xD_x + \frac{1}{4\pi\ii},
\]
where the term $\frac{1}{4\pi\ii}$ comes from the Weyl quantization. Specifically,
\[
	\mathcal{V}_r = \exp(2\pi\ii(\log r)R_0),
\]
and one could therefore make the case that is more natural to use $\tilde{\mathcal{V}}_r f(x) = \ee^{r/2}f(\ee^r x)$ as appears elsewhere in the literature. The author feels, in the end, that the current definition of $\mathcal{V}_r$ is perhaps more familiar.

Similarly, shifts in phase space are also Schr\"odinger evolutions, in the sense that
\[
	\shift_{(x_0, \xi_0)} = \exp\left(2\pi\ii (\xi_0 x - x_0 D_x)\right).
\]
This may be checked directly or found in essentially any discussion of the Weyl quantization.
\end{remark}

\subsection{Restatement and proof of Theorem \ref{thm_identify}}

We rephrase Theorem \ref{thm_identify} in terms of metaplectic operators since this simplifies the proof of Theorem \ref{thm_identify} significantly and lays the foundation for identifying the coefficient $\mu_0$. On the other hand, the factor $t = \frac{ac}{r^2} + bdr^2$ appears as a \emph{deus ex machina}.

Using the metaplectic definition $\comb_r = \mathcal{V}_{1/r} \comb_1$ presented in \eqref{eq_def_comb_metapl}, $\mathcal{W}_t$ from \eqref{eq_def_op_Wt}, and shifts from \eqref{eq_def_shift} which satisfy \eqref{eq_shift_integer_comb}, Theorem \ref{thm_identify} becomes
\[
	\mathcal{F}^\alpha \mathcal{V}_{1/r} \comb_1 = \mu_0 \ee^{\frac{\pi\ii}{4}abcd}\mathcal{W}_{-t} \mathcal{V}_s\shift_{(\frac{ab}{2}, -\frac{cd}{2})}\comb_1.
\]

We can solve for $\comb_1$ on the right, obtaining
\[
	\mathcal{S}_{(-\frac{ab}{2}, \frac{cd}{2})}\mathcal{V}_{1/s}\mathcal{W}_t\mathcal{F}^\alpha \mathcal{V}_{1/r}\comb_1 = \mu \comb_1, \quad \mu = \ee^{\frac{\pi\ii}{4}abcd}\mu_0.
\]
A straightforward computation, presented in Lemma \ref{lem_M_Z}, shows that under the hypotheses of Theorem \ref{thm_identify} and when $\alpha \in (-2, 2]$,
\begin{equation}\label{eq_FrFT_to_M}
	\mathcal{V}_{1/s}\mathcal{W}_t\mathcal{F}^\alpha \mathcal{V}_{1/r} = \mathcal{M}(\Bff{M}), \quad \Bff{M} = \left(\begin{array}{cc} a & b \\ c & d\end{array}\right) \in SL(2,\Bbb{Z}).
\end{equation}
Conversely, allowing for any $\alpha\in\Bbb{R}$ and varying choices of $t$ (Remark \ref{rem_t_rep}), every $\pm \mathcal{M}(\Bff{M})$ for $\Bff{M} \in SL(2,\Bbb{Z})$ can be obtained in this manner.

Of course, $\ee^{\frac{\pi\ii}{4}abcd}$ is an eighth root of unity. We may also, remaining within the set of eighth roots of unity, replace $(ab, -cd)$ by any other $(q,p)\in\Bbb{Z}^2$ with the same parity, since by \eqref{eq_shift_composition} and \eqref{eq_shift_integer_comb}, whenever $(q, p) = (q', p')~(\opnm{mod} 2)$,
\begin{equation}\label{eq_shift_change_rep}
	\begin{aligned}
	\shift_{(\frac{q}{2}, \frac{p}{2})}\comb_1 &= \shift_{(\frac{q'}{2}, \frac{p'}{2})}\mathcal{S}_{(-\frac{q'}{2}, -\frac{p'}{2})}\mathcal{S}_{(\frac{q}{2}, \frac{p}{2})}\comb_1
	\\ &= \ee^{\frac{\pi\ii}{4}(-p'q + pq')}\shift_{(\frac{q'}{2}, \frac{p'}{2})}\shift_{(\frac{q-q'}{2}, \frac{p-p'}{2})} \comb_1
	\\ &= \ee^{\frac{\pi\ii}{4}(pq'-p'q) - \frac{\pi\ii}{4}(q-q')(p-p')}\shift_{(\frac{q'}{2}, \frac{p'}{2})}\comb_1
	\\ &= \ee^{\frac{\pi\ii}{4}(p(q'-q) - (p'-p)q')}\shift_{(\frac{q'}{2}, \frac{p'}{2})}\comb_1.
	\end{aligned}
\end{equation}

We therefore rephrase, and prove, Theorem \ref{thm_identify} in the following equivalent way.

\begin{theorem}\label{thm_id_metapl}
Let $\Bff{M} = \begin{pmatrix} a & b \\ c & d\end{pmatrix}\in SL(2,\Bbb{R})$, and let $q, p\in\Bbb{R}$. Recall the definitions of the Dirac comb $\comb_1$ in \eqref{eq_def_comb}, the shift $\shift_{(\frac{q}{2}, \frac{p}{2})}$ in \eqref{eq_def_shift}, and the metaplectic operator $\mathcal{M}(\Bff{M})$ in \eqref{eq_M_kernel} or in Remark \ref{rem_bzero}.

Then there exists $\mu = \mu(\Bff{M}; (q,p)) \in \Bbb{C}\backslash \{0\}$ for which
\begin{equation}\label{eq_id_metapl}
	\shift_{(\frac{q}{2}, \frac{p}{2})}\mathcal{M}(\Bff{M})\comb_1 = \mu \comb_1
\end{equation}
if and only if $\Bff{M} \in SL(2,\Bbb{Z})$ and if $(q, p) \equiv (ab, cd)~(\opnm{mod} 2)$.
\end{theorem}

\begin{remark}
One could equivalently say that $\comb_1$ is an eigenfunction of $\mathcal{S}_{(\frac{q}{2}, \frac{p}{2})}\mathcal{M}(\Bff{M})$ as an operator on $\mathscr{S}'(\Bbb{R})$ if and only if $\Bff{M} \in SL(2,\Bbb{Z})$ and $(q, p) \equiv (ab, cd)~(\opnm{mod} 2)$. Section \ref{s_symmetries} is devoted to proving that $\mu(\Bff{M}; (q,p))$ is an eighth root of unity, which is equivalent to Theorem \ref{thm_mu_8}.
\end{remark}

\begin{proof}
Using \eqref{eq_shift_integer_comb}, the Egorov relation \eqref{eq_Egorov}, and \eqref{eq_shift_composition}, for any $j,k\in\Bbb{Z}$
\begin{align*}
	\shift_{(\frac{q}{2}, \frac{p}{2})}\mathcal{M}(\Bff{M})\comb_1 &= \ee^{\pi\ii jk}\shift_{(\frac{q}{2}, \frac{p}{2})} \mathcal{M}(\Bff{M}) \shift_{(j,k)} \comb_1
	\\ &= \ee^{\pi\ii jk}\shift_{(\frac{q}{2}, \frac{p}{2})}  \shift_{\Bff{M}(j,k)} \mathcal{M}(\Bff{M})\comb_1
	\\ &= \ee^{\pi\ii c_0}\shift_{\Bff{M}(j,k)}\shift_{(\frac{q}{2},\frac{p}{2})}\mathcal{M}(\Bff{M})\comb_1
\end{align*}
where $c_0 = jk + p(aj+bk) - (cj+dk)q$ is not used here. Therefore if \eqref{eq_id_metapl} holds, then for any $j,k \in \Bbb{Z}$
\[
	\comb_1 = \ee^{\pi\ii c_0}\shift_{\Bff{M}(j,k)}\comb_1,
\]
which can only hold if $\Bff{M}(j,k) \in \Bbb{Z}^2$ for all $j,k\in\Bbb{Z}$. (If not, $\shift_{\Bff{M}(j,k)}\comb_1$ will not be supported in $\Bbb{Z}$ or will not be $1$-periodic.) Therefore $\Bff{M} \in SL(2, \Bbb{Z})$ is a necessary condition for \eqref{eq_id_metapl}.

Let us suppose then that $\Bff{M} \in SL(2, \Bbb{Z})$. By the classical Lemma \ref{lem_id_u1}, it is enough to show that
\[
	\mathcal{S}_{(\frac{q}{2}, \frac{p}{2})}\mathcal{M}(\Bff{M})\comb_1
\]
is invariant under $\mathcal{S}_{(1,0)}$ and $\mathcal{S}_{(0,1)}$. Using \eqref{eq_shift_composition}, \eqref{eq_Egorov} where we apply $\Bff{M}^{-1}$, and \eqref{eq_shift_integer_comb},
\begin{align*}
	\shift_{(1,0)}\shift_{(\frac{q}{2}, \frac{p}{2})}\mathcal{M}(\Bff{M})\comb_1 &= \ee^{-\pi\ii p}\mathcal{S}_{(\frac{q}{2}, \frac{p}{2})}\mathcal{S}_{(1,0)}\mathcal{M}(\Bff{M})\comb_1 
	\\ &= \ee^{-\pi\ii p} \mathcal{S}_{(\frac{q}{2}, \frac{p}{2})}\mathcal{M}(\Bff{M})\mathcal{S}_{(d, -c)}\comb_1 
	\\ &= \ee^{-\pi\ii p + \pi\ii cd}\mathcal{S}_{(\frac{q}{2}, \frac{p}{2})}\mathcal{M}(\Bff{M})\comb_1.
\end{align*}
Similarly,
\begin{align*}
	\mathcal{S}_{(0,1)}\mathcal{S}_{(\frac{q}{2}, \frac{p}{2})}\mathcal{M}(\Bff{M})\comb_1 &= \ee^{\pi\ii q}\mathcal{S}_{(\frac{q}{2}, \frac{p}{2})}\mathcal{S}_{(0,1)}\mathcal{M}(\Bff{M})\comb_1 
	\\ &= \ee^{\pi\ii q} \mathcal{S}_{(\frac{q}{2}, \frac{p}{2})}\mathcal{M}(\Bff{M})\mathcal{S}_{(-b, a)}\comb_1 
	\\ &= \ee^{\pi\ii q + \pi\ii ab}\mathcal{S}_{(\frac{q}{2}, \frac{p}{2})}\mathcal{M}(\Bff{M})\comb_1.
\end{align*}
We see that $\shift_{(\frac{q}{2}, \frac{p}{2})}\mathcal{M}(\Bff{M})\comb_1$ is invariant under $\shift_{(1,0)}$ and $\shift_{(0,1)}$ if and only if $\ee^{-\pi\ii p + \pi\ii cd} = \ee^{\pi\ii q + \pi\ii ab} = 1$ if and only if $(q, p) \equiv (ab, cd)~(\opnm{mod} 2)$; therefore, when this holds, it is a multiple of the Dirac comb $\comb_1$. Since the multiple cannot be zero since $\shift_{(\frac{q}{2}, \frac{p}{2})}\mathcal{M}(\Bff{M})$ is invertible on $\mathscr{S}'(\Bbb{R})$, this proves the theorem.
\end{proof}

\begin{remark}\label{rem_sum_expression}
Having defined $\mu(\Bff{M}; (q,p))$ in Theorem \ref{thm_id_metapl}, we record the equivalent formula for $\mathcal{F}^\alpha \comb_r$ in terms of a sum of $\delta$-functions. Recall that we assume that $\ee^{\frac{\pi\ii}{2}\alpha} = \frac{1}{s}(\frac{a}{r} + \ii br)$ for $a, b \in \Bbb{Z}$ relatively prime and $s = \sqrt{a^2 + b^2}$, that $c, d \in \Bbb{Z}$ are such that $ad - bc = 1$, and that $t = \frac{ac}{r^2} + bdr^2$. Fix $q, p \in \Bbb{Z}$ such that $(q,p) \equiv (ab, cd)~(\opnm{mod} 2)$, and let $\mu = \mu\left(\begin{pmatrix} a & b \\ c & d\end{pmatrix}; (q,p)\right)$. To handle $\alpha \notin (-2, 2]$, we simply let $\eps = \eps(\alpha) = 1$ if $\alpha \in (-2, 2]~(\opnm{mod} 8)$ and $-1$ otherwise. We then solve for $\mathcal{F}^\alpha \comb_r$ in
\[
	\shift_{(\frac{q}{2}, \frac{p}{2})}\mathcal{V}_{\frac{1}{s}} \mathcal{W}_t \mathcal{F}^\alpha \comb_r = \eps\mu(\Bff{M}; (q,p))\comb_1
\]
to obtain
\begin{align*}
	\mathcal{F}^\alpha \comb_r(x) &= \eps\mu\mathcal{W}_{-t}\mathcal{V}_s \shift_{(-\frac{q}{2}, -\frac{p}{2})}\comb_1(x)
	\\ &= \eps \mu \mathcal{W}_{-t}\mathcal{V}_s \left(\sum_{k \in\Bbb{Z}}\ee^{-\frac{\pi\ii}{4}qp - \pi\ii px}\delta(x + \frac{q}{2} - k)\right)
	\\ &= \eps \mu \mathcal{W}_{-t}\mathcal{V}_s\left(\sum_{k \in \Bbb{Z} - \{\frac{q}{2}\}}\ee^{-\frac{\pi\ii}{4}qp - \pi\ii px}\delta(x-k)\right)
	\\ &= \eps \mu s^{-\frac{1}{2}}\sum_{k \in \Bbb{Z} - \{\frac{q}{2}\}}\ee^{-\frac{\pi\ii}{4}qp - \pi\ii psx-\pi\ii t x^2}\delta(x-\frac{k}{s})
	\\ &= \eps \mu s^{-\frac{1}{2}}\sum_{k \in \Bbb{Z} - \{\frac{q}{2}\}}\ee^{-\frac{\pi\ii}{4}qp - \pi\ii pk - \pi\ii \frac{t}{s^2} k^2}\delta(x-\frac{k}{s}).
\end{align*}
\end{remark}

\begin{definition}\label{def_antiderivative}
Taking $r = 1$ for simplicity, when $\cos\frac{\pi\alpha}{2}$ and $\sin\frac{\pi\alpha}{2}$ are linearly dependent over $\Bbb{Z}$, we define the antiderivative for $X > 0$ as
\[
	\Pi_\alpha(X) = \int_0^X \mathcal{F}^\alpha \comb_1(x)\,\dd x.
\]
More precisely, since $\delta$-functions are involved,
\[
	\Pi_\alpha(X) = \frac{1}{2}\left(\int_{[0, X]} \mathcal{F}^\alpha\comb_1(x)\,\dd x + \int_{(0, X)}\mathcal{F}^\alpha\comb_1(x)\,\dd x\right).
\]
We extend $\Pi_\alpha(X)$ to be an odd function, so in particular $\Pi_\alpha(0) = 0$.

For example, with this definition, if $X \in (k, k+1)$ for $k\in\Bbb{Z}$, then $\Pi_0(X) = k+\frac{1}{2}$, and $\Pi_0(k) = k$ for all $k \in \Bbb{Z}$.

The advantages of this definition are that $\Pi_\alpha$ should be an odd function, since $\mathcal{F}^\alpha \comb_1(x)$ is even (Remark \ref{rem_sign_evenness}) and that $\Pi_\alpha(X_1+X_2) = \Pi_\alpha(X_1) + \Pi_\alpha(X_2)$. 
\end{definition}

\subsection{And the functional equation for theta functions with half-integer characteristics}\label{ss_theta}

The Jacobi theta functions can be obtained by taking the duality product of Gaussians and (shifted) Dirac combs. Let
\[
	g_{(0, 0), \tau}(x) = \ee^{\pi\ii\tau x^2}
\]
and
\begin{equation}\label{eq_def_gaussian}
	\begin{aligned}
	g_{(x_0, \xi_0), \tau}(x) &= \shift_{(x_0, \xi_0)}g_{(0, 0), \tau}(x)
	\\ &= \ee^{-\pi\ii x_0\xi_0 + 2\pi\ii \xi_0 x + \pi\ii \tau(x-x_0)^2}.
	\end{aligned}
\end{equation}
When $\Im \tau > 0$ and $z \in \Bbb{C}$, we write
\begin{equation}\label{eq_def_theta_qp}
	\theta_{qp}(z,\tau) = \langle \shift_{(\frac{q}{2}, \frac{p}{2})}g_{(0, z), \tau}, \comb_1\rangle.
\end{equation}
This notation differs slightly from the notation used in \cite{Mumford_1983} but agrees for $\theta_{00}, \theta_{10}$, and $\theta_{01}$ and agrees more generally up to eighth roots of unity.

Testing Theorem \ref{thm_identify} against Gaussians gives the classical functional equation \cite[Thm.~7.1]{Mumford_1983} for Jacobi theta functions, which describes the transformation of these theta functions under the action of $SL(2,\Bbb{Z})$. Or, if the reader prefers, the identification of $\mu(\Bff{M}; (p,q))$ as an eighth root of unity is a consequence of this classical functional equation.

\begin{theorem}\label{thm_theta}[Theorem~7.1, \cite{Mumford_1983}]
Let $\Bff{M} = \begin{pmatrix} a & b \\ c & d\end{pmatrix}\in SL(2,\Bbb{Z})$ and let $q,p\in\Bbb{Z}$ be such that $(q,p) \equiv (ab, cd)~(\opnm{mod} 2)$. Then for any $z, \tau \in\Bbb{C}$ with $\Im \tau > 0$, let
\[
	z' = \frac{z}{a+b\tau}, \quad \tau' = \frac{c+d\tau}{a+b\tau},
\]
and $\mu = \mu(\Bff{M}; (q,p))$ from Theorem \ref{thm_id_metapl}. Then
\begin{equation}\label{eq_theta}
	\theta_{00}(z,\tau) = \bar{\mu}\ee^{-\frac{\pi\ii b}{a+b\tau}z^2}(a+b\tau)^{-1/2} \theta_{qp}(z', \tau').
\end{equation}
(The square root is taken with positive real part.)
\end{theorem}

\begin{proof}
We use $\langle \cdot, \cdot\rangle$ to indicate the sesquilinear duality product between $\mathscr{S}(\Bbb{R})$ and $\mathscr{S}'(\Bbb{R})$. (See Remark \ref{rem_unitary}.) To simplify notation, let $\mathcal{M} = \mathcal{M}(\Bff{M})$ and $\mu = \mu(\Bff{M}; (q,p))$. Using Lemma \ref{lem_M_on_gaussian}, \eqref{eq_Egorov}, and \eqref{eq_gaussian_shift_project},
\[
	\mathcal{M}g_{(0,z), \tau} = (a+b\tau)^{-1/2}g_{(bz, dz), \tau'} = \ee^{-\pi\ii b(d-\tau'b)z^2}g_{(0, (d-\tau'b)z), \tau'}.
\]
One computes that
\[
	d-\tau'b = d- b\left(\frac{c+d\tau}{a+b\tau}\right) = \frac{ad - bc}{a+b\tau} = \frac{1}{a + b\tau}
\]
since $ad - bc = 1$. Therefore
\[
	\mathcal{M}g_{(0, z), \tau} = (a+b\tau)^{-1/2}\ee^{-\frac{\pi\ii b}{a+b\tau}z^2}g_{(0, z'),\tau'},
\]
and
\begin{align*}
	\theta_{00}(z, \tau) &= \langle g_{(0, z), \tau}, \comb_1\rangle
	\\ &= \langle \shift_{(\frac{q}{2}, \frac{p}{2})}\mathcal{M} g_{(0, z), \tau}, \shift_{(\frac{q}{2}, \frac{p}{2})}\mathcal{M}\comb_1\rangle
	\\ &= \ee^{-\frac{\pi\ii b}{a+b\tau}z^2}(a+b\tau)^{-1/2}\langle \shift_{(\frac{q}{2}, \frac{p}{2})} g_{(0, z'), \tau'}, \mu \comb_1\rangle
	\\ &= \bar{\mu}\ee^{-\frac{\pi\ii b}{a+b\tau}z^2}(a+b\tau)^{-1/2} \theta_{qp}(z', \tau'),
\end{align*}
which proves the theorem.
\end{proof}

\begin{remark}
This presentation of the functional equation for the Jacobi theta functions allows us to identify certain separate ``moving parts.'' Most importantly, the eighth root of unity belongs to the Dirac comb and its transformation under metaplectic operators, and it has nothing to do with $z$ or $\tau$. The linear fractional transformation giving $\tau'$ from $\tau$ and the factor $(a + b\tau)^{-1/2}$ are due to the transformation of a (centered) Gaussian under a metaplectic operator. 

Finally, the formula for $z'$ and the factor $\ee^{-\frac{\pi\ii b}{a+b\tau}z^2}$ comes first from the metaplectic operator, which transforms $(0, z)$ into $(bz, dz)$, and then from our insistence that the shift is of the form $(0, z')$, in momentum only. This (somewhat unnatural) requirement leads us to project $(bz, dz)$ onto $\{(0, \xi)\}_{\xi \in \Bbb{C}}$ along $\Lambda_\tau = \{(x, \tau x)\}_{x\in\Bbb{C}}$, which explains the dependence of $z'$ on $\tau$ as well as the factor $\ee^{-\frac{\pi\ii b}{a+b\tau}z^2}$. (See Lemma \ref{lem_equiv_shifts}.) One could naturally simplify this dependence in defining for $(x_0, \xi_0) \in \Bbb{R}^2$ and $\Im \tau > 0$ the function
\[
	\Theta_{qp}((x_0, \xi_0), \tau) = \langle \shift_{(\frac{q}{2}, \frac{p}{2})}g_{(x_0, \xi_0), \tau}, \comb_1\rangle
\]
for which
\[
	\Theta_{00}((x_0, \xi_0), \tau) = \bar{\mu}(a+b\tau)^{-1/2}\Theta_{qp}(\Bff{M}(x_0, \xi_0), \tau').
\]
But, of course, one could eliminate $\tau'$ and $(a+b\tau)^{-1/2}$ as well by writing Theorem \ref{thm_id_metapl}.
\end{remark}

\subsection{Lemmas on the metaplectic operator $\mathcal{M}$}

In this section, we collect a number of elementary lemmas on metaplectic operators and their effect on Gaussians. The results of principal interest are \eqref{eq_FrFT_to_M} relating Theorem \ref{thm_identify} to Theorem \ref{thm_id_metapl} and the effect of a metaplectic operator on a Gaussian.

\begin{lemma}\label{lem_M_Z}
Let $r > 0$ and let $\alpha \in (-2, 2]$ be such that $\ee^{\frac{\pi\ii}{2}\alpha} = \frac{1}{s}(a/r+\ii br)$ for $s > 0$ and $a,b\in\Bbb{Z}$ relatively prime, as in \eqref{eq_def_rho_s}. Let $t = ac/r^2 + bdr^2$ as in \eqref{eq_def_t} and let $\Bff{M} = \begin{pmatrix} a & b \\ c & d \end{pmatrix}\in SL(2,\Bbb{Z})$. Recall the metaplectic operators $\mathcal{F}^\alpha$, $\mathcal{V}_{\rho}$, and $\mathcal{W}_{\tau}$ from Definition \ref{def_FrFT}, \eqref{eq_def_op_Vr}, and \eqref{eq_def_op_Wt}. Then
\[
	\mathcal{V}_{1/s}\mathcal{W}_t\mathcal{F}^\alpha \mathcal{V}_{1/r} = \mathcal{M}(\Bff{M})
\]
defined in \eqref{eq_M_kernel}.
\end{lemma}

\begin{proof}
We use the Mehler formula \eqref{eq_def_op_FrFT}; the case $\alpha = 2$, when $\sin \frac{\pi\alpha}{2} = 0$, is trivial because $s = 1$ and $\mathcal{F}^\alpha f(x) = -f(x)$. Along with the change of variables $y' = y/r$, we obtain
\[
	\begin{aligned}
	\mathcal{V}_{1/s}\mathcal{W}_t\mathcal{F}^\alpha \mathcal{V}_{1/r} f(x) &= \mathcal{V}_{1/s} \ee^{\pi\ii tx^2}\frac{1}{\sqrt{\ii br/s}}\int \ee^{\frac{\pi\ii}{br/s}(\frac{a}{rs}x^2 - 2xy +\frac{a}{rs}y^2)}\frac{1}{\sqrt{r}}f(y/r)\,\dd y
	\\ &= \frac{1}{\sqrt{s}}\ee^{\frac{\pi\ii t}{s^2}x^2}\frac{1}{\sqrt{\ii b r/s}} \int \ee^{\frac{\pi\ii}{b}(\frac{a}{r^2s^2}x^2 - 2xy' + a(y')^2)}\sqrt{r}f(y')\,\dd y'
	\\ &= \frac{1}{\sqrt{\ii b}}\int \ee^{\frac{\pi\ii}{b}(\frac{1}{s^2}(\frac{a}{r^2} + bt)x^2 - 2xy' + a(y')^2)}f(y')\,\dd y'.
	\end{aligned}
\]
It suffices to show that the coefficient of $x^2$ in the exponent is in fact $d$. Using $bc+1 = ad$,
\[
	\frac{a}{r^2} + bt = \frac{a}{r^2} + \frac{bac}{r^2} + b^2dr^2 = \frac{a}{r^2}(bc+1) + (br)^2 d = ((a/r)^2 + (br)^2)d = s^2d.
\]
This confirms that the coefficient of $x^2$ is $d$ and completes the proof of the lemma.
\end{proof}

In order to compute the effect of a metaplectic operator on a Gaussian
\[
	g_{(x_0, \xi_0), \tau} = \shift_{(x_0, \xi_0)}g_{0, \tau}, \quad g_{0, \tau}(x) = \ee^{\pi\ii x^2},
\]
we begin with the centered Gaussian $g_{(0,0), \tau}$.

\begin{lemma}\label{lem_M_on_gaussian}
Let $g_{(0,0), \tau}(x) = \ee^{\pi\ii\tau x^2}$ and suppose that $\Im \tau > 0$. Let $\Bff{M} = \begin{pmatrix} a & b \\ c & d \end{pmatrix} \in SL(2,\Bbb{R})$. Define
\[
	\tau' = \frac{c+d\tau}{a + b\tau}.
\]
Then, with the metaplectic operator $\mathcal{M}(\Bff{M})$ defined as in \eqref{eq_M_kernel},
\[
	\mathcal{M}(\Bff{M})g_{(0,0), \tau} = (a+b\tau)^{-1/2}g_{(0,0), \tau'}.
\]
where the square root is chosen such that $\Re((a+b\tau)^{-1/2}) > 0$.
\end{lemma}

\begin{proof}
This is a direct computation:
\begin{align*}
	\mathcal{M}(\Bff{M})g_{(0,0), \tau}(x) &= \frac{1}{\sqrt{\ii b}} \int \ee^{\frac{\pi\ii}{b}(dx^2 - dxy + ay^2)+\pi\ii\tau y^2}\,\dd y
	\\ &= \frac{1}{\sqrt{\ii b}}\ee^{\frac{\pi\ii b}{d}x^2}\int \ee^{\pi\ii(\tau + \frac{a}{b})y^2 - \frac{2\pi\ii}{b}xy}\,\dd y
	\\ &= \frac{1}{\sqrt{\ii b}\sqrt{-\ii (a/b+\tau)}} \ee^{\frac{\pi\ii b}{d}x^2 - \frac{\pi\ii}{b(a+b\tau)}x^2}.
\end{align*}
We check the sign of the coefficient: in order to integrate the Gaussian, we have used that 
\[
	\pi\ii (\tau + \frac{a}{b}) = -\pi(-\ii(\tau + \frac{a}{b}))
\]
where $\Re (-\ii(\tau + \frac{a}{b})) > 0$. The square root is therefore chosen such that 
\[
|\arg(\sqrt{-\ii (a/b + \tau)})| < \pi/4.
\]
Because $\mathcal{M}$ is defined with $\Re \sqrt{\ii b} > 0$, meaning that $\arg\sqrt{\ii b} = \pm \frac{\pi\ii}{4}$, we have that
\[
	\left|\arg\left(\sqrt{\ii b}\sqrt{-\ii (a/b+\tau)}\right)\right| < \frac{\pi}{2}.
\]
Therefore
\[
	\frac{1}{\sqrt{\ii b}\sqrt{-\ii (a/b+\tau)}} = (a+b\tau)^{-1/2}
\]
with positive real part.

The coefficient of $\pi\ii x^2$ in the exponent is
\[
	\frac{1}{b}\left(d-\frac{1}{a + b\tau}\right) = \frac{1}{b}\left(\frac{bd + ad-1}{a + b\tau}\right) = \tau'
\]
because $ad-1 = bc$. This completes the proof of the lemma.
\end{proof}

\begin{remark}\label{rem_gaussians_suffice}
For any $\tau$ with $\Im \tau > 0$ fixed, the set of Gaussians $\{g_{\Bff{v}, \tau}\::\: \Bff{v} \in\Bbb{R}^2\}$ has dense span in $L^2(\Bbb{R})$; indeed, if $\langle f, g_{\Bff{v}, \tau}\rangle = 0$ for all $\Bff{v}\in\Bbb{R}^2$, a modified Bargmann transform of $f$, like the one used in Section \ref{s_Bargmann}, vanishes identically. The relation
\[
	\mathcal{M}(\Bff{M})g_{\Bff{v}, \tau} = (a+b\tau)^{-1/2}g_{\Bff{Mv}, \tau},
\]
coming from the Egorov relation \eqref{eq_Egorov} and Lemma \ref{lem_M_on_gaussian}, therefore suffices as a definition of $\mathcal{M}(\Bff{M})$. Note that this definition does not require $b \neq 0$, but the choice of sign when $a = -1$ and $b = 0$ is still somewhat arbitrary.

However, there is some ambiguity in that many different $\Bff{v}$ give, up to constant multiples, the same $g_{\Bff{v}, \tau}$. These $\Bff{v}$ are equal modulo the plane $\{(x, \tau x)\}_{x\in\Bbb{C}}$ associated with the Gaussian $g_{0, \tau}$; see \cite[Sect.~5]{Hormander_1995}.
\end{remark}

\begin{lemma}\label{lem_equiv_shifts}
	Let $(x_0, \xi_0), (y_0, \eta_0) \in \Bbb{C}^2$ and let $\tau \in \Bbb{C}$. Then there exists $c_0 \in \Bbb{C}$ such that
	\[
		g_{(x_0, \xi_0), \tau} = c_0 g_{(y_0, \eta_0), \tau}
	\]
	if and only if $\xi_0 - \eta_0 = \tau (x_0 - y_0)$. In this case, $c_0 = \ee^{\pi\ii(\xi_0 y_0 - \eta_0 x_0)}$.
\end{lemma}

\begin{proof}
	It is instructive to begin with the case $(y_0, \eta_0) = (0,0)$, where it is evident that
	\begin{align*}
		g_{(x_0, \xi_0), \tau}(x) &= \ee^{-\pi\ii x_0\xi_0 + 2\pi\ii \xi_0 x + \pi\ii\tau(x-x_0)^2}
		\\ &= \ee^{-\pi\ii x_0(\xi_0 - \tau x_0) + 2\pi\ii (\xi_0 - \tau x_0)x + \pi\ii \tau x^2}
		\\ &= g_{(0,0), \tau}(x)
	\end{align*}
	if and only if $\xi_0 = \tau x_0$ (and in this case the constant is $c_0 = 1$).
	
	The general result follows from writing the equivalent statement
	\begin{align*}
		g_{0, \tau} &= c_0\shift_{-(x_0, \xi_0)}\shift_{(y_0, \eta_0)} g_{0, \tau}
		\\ &= c_0\ee^{-\pi\ii(\xi_0 y_0 - \eta_0 x_0)}g_{(y_0-x_0, \eta_0-\xi_0), \tau}
	\end{align*}
	and applying the previous special case.
\end{proof}

A consequence, which may be checked directly, is that, when $\tau \neq 0$ and if $y_0 = x_0 - \xi_0/\tau$ and $\eta_0 = \xi_0 - \tau x_0$,
\begin{equation}\label{eq_gaussian_shift_project}
	g_{(x_0, \xi_0), \tau} = \ee^{-\pi\ii x_0 \eta_0} g_{(0, \eta_0), \tau} = \ee^{\pi\ii \xi_0 y_0}g_{(y_0, 0), \tau}.
\end{equation}

As an application of the effect of a metaplectic operator on a Gaussian, we finish this subsection by proving some well-known formulas for compositions of metaplectic operators.

\begin{lemma}\label{lem_M_compose}
Let $\Bff{M}_1\Bff{M}_2 = \Bff{M}_3$ for $\Bff{M}_j \in SL(2,\Bbb{R})$ with
\[
	\Bff{M}_j = 
		\begin{pmatrix}
			a_j & b_j \\
			c_j & d_j
		\end{pmatrix},
		\quad j = 1,2,3.
\]
Then there exists a $\sigma \in \{\pm 1\}$ such that
\[
	\mathcal{M}(\Bff{M}_1)\mathcal{M}(\Bff{M}_2) = \sigma \mathcal{M}(\Bff{M}_3).
\]
For any $\tau \in \Bbb{C}$ with $\Im \tau > 0$,
\[
	\sigma = (a_2 + b_2\tau)^{-1/2}
	\left(\frac{a_3 + b_3\tau}{a_2+b_2\tau}\right)^{-1/2}
	(a_3 + b_3\tau)^{1/2},
\]
where each of the three terms has positive real part.
\end{lemma}

\begin{corollary} The set $\{\pm \mathcal{M}(\Bff{M}) \::\: \Bff{M} \in G\}$ is closed for any subgroup $G$ of $SL(2,\Bbb{R})$.
\end{corollary}

\begin{proof}
In view of the density of $\{g_{\Bff{v}, \tau}\::\: \Bff{v} \in \Bbb{R}^2\}$ (see Remark \ref{rem_gaussians_suffice}), it suffices to show that $\mathcal{M}(\Bff{M}_1)\mathcal{M}(\Bff{M}_2)$ and $\sigma\mathcal{M}(\Bff{M}_3)$ agree on any Gaussian $g_{\Bff{v}, \tau}$. When
\[
	\tau_1 = \frac{c_2+d_2\tau}{a_2 + b_2\tau}, \quad \tau_2 = \frac{c_1 + d_1\tau_1}{a_1+b_1\tau_1},
\]
\[
	\mathcal{M}(\Bff{M}_1)\mathcal{M}(\Bff{M}_2)g_{\Bff{v}, \tau} = (a_2+b_2\tau)^{-1/2}(a_1+b_1\tau_1)^{-1/2}g_{\Bff{M}_1\Bff{M}_2 \Bff{v}, \tau_2},
\]
and when $\tau_3 = \frac{c_3+d_3\tau}{a_3+b_3\tau}$,
\[
	\sigma\mathcal{M}(\Bff{M}_3)g_{\Bff{v}, \tau} = \sigma(a_3 + b_3\tau)^{-1/2}g_{\Bff{M}_3\Bff{v},\tau_3}.
\]
Since $\Bff{M}_1\Bff{M}_2 = \Bff{M}_3$, it suffices to show that $\tau_2 = \tau_3$ and that we have the correct formula for $\sigma$.

Expanding
\[
	\Bff{M}_3 = \begin{pmatrix}
					a_3 & b_3 \\ c_3 & d_3
				\end{pmatrix}
			=	\begin{pmatrix}
					a_1 a_2 + b_1c_2 & a_1b_2 + b_1 d_2 \\
					c_1a_2 + d_1c_2 & c_1b_2 + d_1d_2
				\end{pmatrix}
			= \Bff{M}_1\Bff{M}_2,
\]
we compute
\[
	\tau_2 = \frac{c_1(a_2+b_2\tau) + d_1(c_2 + d_2\tau)}{a_1(a_2 + b_2\tau) + b_1(c_2 + d_2\tau)} = \frac{c_3 + d_3\tau}{a_3 + b_3\tau}
\]
and
\[
	a_1 + b_1\tau_1 = \frac{a_3 + b_3\tau}{a_2 + b_2\tau}.
\]
Therefore $\tau_2 = \tau_3$ and
\begin{align*}
	\sigma &= (a_3 + b_3\tau)^{1/2}(a_2+b_2\tau)^{-1/2}(a_1 + b_1\tau_1)^{-1/2}
	\\ &= (a_2 + b_2\tau)^{-1/2}
	\left(\frac{a_3 + b_3\tau}{a_2+b_2\tau}\right)^{-1/2}
	(a_3 + b_3\tau)^{1/2}
\end{align*}
as promised. It is clear that $\sigma^2 = 1$ and that $\sigma$ is a continuous function of $\tau$ to $\{\pm 1\}$, so $\sigma$ is independent of $\tau$ and the lemma is proved.
\end{proof}

In the case of scaling or multiplication by a Gaussian, $\sigma = 1$, which may also be easily verified using the integral kernel.

\begin{corollary}\label{cor_M_compose_VW}
If $\Bff{M}_1 = \Bff{V}_r$ for $r > 0$ or if $\Bff{M}_1 = \Bff{W}_t$ for $t \in \Bbb{R}$ as in \eqref{eq_def_canon_Vr} or \eqref{eq_def_canon_Wt}, then for any $\Bff{M}_2 \in SL(2,\Bbb{R})$
\[
	\mathcal{M}(\Bff{M}_1)\mathcal{M}(\Bff{M}_2) = \mathcal{M}(\Bff{M}_1\Bff{M}_2).
\]
\end{corollary}

\begin{proof}
In the former case, $\Bff{M}_1 = \Bff{V}_r$, we have $(a_1, b_1) = (1/r, 0)$ and $(a_3, b_3) = (a_2/r, b_2/r)$, and in the latter case, $\Bff{M}_1 = \Bff{W}_t$, we have $(a_1, b_1) = (1, 0)$ and $(a_3, b_3) = (a_2, b_2)$. It is obvious, in either case, that $\sigma = 1$ in Lemma \ref{lem_M_compose}.
\end{proof}

\begin{remark}
Having established composition rules for metaplectic operators, we note that Lemma \ref{lem_M_Z} follows upon checking that $\Bff{V}_{1/s}\Bff{W}_t\Bff{F}_\alpha \Bff{V}_{1/r} = \Bff{M}$ which is an elementary matrix multiplication.
\end{remark}

\subsection{Identification of $\comb_r$ via invariance under integer shifts}

We record the classical fact that the Dirac comb $\comb_r$ from \eqref{eq_def_comb} is the only distribution (up to constant multiples) invariant under both $\mathcal{S}_{(r, 0)}$ and $\mathcal{S}_{(0, 1/r)}$. The proof is adapted from \cite[Sec.~7.2]{Hormander_ALPDO1}.

\begin{lemma}\label{lem_id_u1}
A distribution $v \in \mathscr{D}'(\Bbb{R})$ is a constant multiple of the Dirac comb $\comb_r$ in \eqref{eq_def_comb} for $r > 0$ if and only if
\begin{equation}\label{eq_lem_id_u1}
	\mathcal{S}_{(r,0)}v = \mathcal{S}_{(0, 1/r)}v = v.
\end{equation}
\end{lemma}

\begin{proof}
Writing $\mathcal{S}_{(r, 0)}\comb_r = \mathcal{S}_{(0, 1/r)}\comb_r = \comb_r$ using the definition \eqref{eq_def_comb} gives the equivalent statement
\[
	\sqrt{r}\sum_{k\in\Bbb{Z}}\delta(x-rk) = \sqrt{r}\sum_{k\in\Bbb{Z}}\delta(x-r-rk) = \sqrt{r}\sum_{k\in\Bbb{Z}} \ee^{2\pi\ii x/r}\delta(x-rk).
\]
This is true by a change of variables and by the observation that $\ee^{2\pi\ii x/r} = 1$ when $x = rk$ for any $k\in\Bbb{Z}$.

Suppose, conversely, that the two equalities in \eqref{eq_lem_id_u1} hold. When $\mathcal{S}_{(0, 1/r)}v = v$, $(\ee^{2\pi\ii x/r}-1)v(x) = 0$. Therefore $\supp v \subseteq r\Bbb{Z}$ and, since the zeros of $(\ee^{2\pi\ii x/r} - 1)$ are nondegenerate, in a neighborhood of any $rk \in r\Bbb{Z}$,
\[
	(x-rk)v(x) = \frac{x-rk}{\ee^{2\pi\ii x/r} - 1}(\ee^{2\pi\ii x/r}-1)v(x)= 0.
\]
Therefore, by \cite[Thm.~3.1.16]{Hormander_ALPDO1}, $v(x)$ coincides with a multiple of a $\delta$ function in a neighborhood of any $rk \in r\Bbb{Z}$, or
\[
	v(x) = \sum_{k \in \Bbb{Z}} \gamma_k \delta(x-rk)
\]
for some sequence $\{\gamma_k\}_{k\in\Bbb{Z}}$ of complex numbers. Finally, $\mathcal{S}_{(r,0)}v = v$, so $v$ is $r$-periodic. Therefore $\{\gamma_k\}$ is a constant sequence and $v = \gamma_0 \comb_r$.
\end{proof}

\section{Evaluation of the coefficient $\mu$}\label{s_symmetries}

We follow \cite{Mumford_1983} in showing that the coefficient $\mu(\Bff{M};(q, p))$ in Theorem \ref{thm_id_metapl} can be obtained from certain symmetries. We remark that, when $\Bff{I}$ is the identity matrix and $(q, p) \equiv (0,0)~(\opnm{mod} 2)$, then $\mathcal{M}(\Bff{I})$ is the identity operator, and by \eqref{eq_shift_integer_comb},
\[
	\mu(\Bff{I}; (q, p)) = \ee^{-\frac{\pi\ii}{4}pq} \in \{\pm 1\}.
\]
We also note that changing representations of $(q, p)$ modulo 2 follows readily from \eqref{eq_shift_change_rep}: by Theorem \ref{thm_id_metapl}, when $(q, p) \equiv (q', p') \equiv (ab, cd)~(\opnm{mod} 2)$,
\[
	\mu(\Bff{M};(q,p))\mathcal{S}_{(-\frac{q}{2}, -\frac{p}{2})} \comb_1 
	= \mathcal{M}(\Bff{M})\comb_1 
	= \mu(\Bff{M};(q',p'))\mathcal{S}_{(-\frac{q'}{2}, -\frac{p'}{2})} \comb_1.
\]
By \eqref{eq_shift_change_rep},
\[
	\mathcal{S}_{(-\frac{q}{2}, -\frac{p}{2})}\comb_1 = \ee^{\frac{\pi\ii}{4}((p'-p)q' - p(q'-q))}\mathcal{S}_{(-\frac{q'}{2}, -\frac{p'}{2})}\comb_1,
\]
so, still supposing that $(q,p) \equiv (q', p') \equiv (ab, cd)~(\opnm{mod} 2)$,
\begin{equation}\label{eq_mu_change_rep}
	\mu(\Bff{M}; (q, p)) = \ee^{\frac{\pi\ii}{4}((p'-p)q' - p(q'-q))}\mu(\Bff{M};(q', p')).
\end{equation}

We therefore search for symmetries paying little regard to $(p,q)$. Following the proof of \cite[Thm.~7.1]{Mumford_1983}, it is enough to understand how $\mu$ changes under the transformations
\[
	\begin{aligned}
	\Bff{M} = \left(\begin{array}{cc} a & b \\ c & d\end{array}\right) &\mapsto \left(\begin{array}{cc} -b & a \\ -d & c\end{array}\right) = \Bff{M}\Bff{F}_1,
	\\ \Bff{M} &\mapsto \left(\begin{array}{cc} a & b \\ c+ja & d+jb\end{array}\right) = \Bff{W}_j\Bff{M},
	\\ \Bff{M} &\mapsto \left(\begin{array}{cc} a+jb & b \\ c+jd & d\end{array}\right) = \Bff{M}\Bff{W}_j.
	\end{aligned}
\]
Here, the matrices $\Bff{F}_1$ and $\Bff{W}_j$ (where $j \in \Bbb{Z}$) are defined in \eqref{eq_def_F_canon} and \eqref{eq_def_canon_Wt}. The effect of these transformations on $\mu$ follow from the Poisson summation formula and the effect of multiplication by Gaussians on $\mathcal{M}$.

\begin{proposition}\label{prop_mu_transformations}
For $\Bff{M} \in SL(2, \Bbb{Z})$ as in \eqref{eq_def_M_Bff} and for $(q,p) \equiv (ab, cd)~(\opnm{mod} 2)$ where $a,b,c,d$ are the entries of $\Bff{M}$, let $\mu(\Bff{M}; (q,p))$ be as in Theorem \ref{thm_id_metapl}. Let $\sigma = -1$ if $a < 0$ and $b \geq 0$ and $\sigma = +1$ otherwise, and let $j \in \Bbb{Z}$ be arbitrary. Then, with $\Bff{F}_1$ and $\Bff{W}_j$ from \eqref{eq_def_F_canon} and \eqref{eq_def_canon_Wt},
\[
	\begin{aligned}
	\mu(\Bff{M}; (q,p)) &= \sigma \ee^{\frac{\pi\ii}{4}}\mu(\Bff{MF}_1; (q,p)),
	\\&= \ee^{-\frac{\pi\ii}{4}jp}\mu(\Bff{W}_j\Bff{M}; (q, p+j(q-1)))
	\\&= \ee^{\frac{\pi\ii}{4}j(dq-pb)}\mu(\Bff{M}\Bff{W}_j; (q-jb, p-jd)).
	\end{aligned}
\]
\end{proposition}

\begin{proof}
For the first, we revert to the expression corresponding to Theorem \ref{thm_identify} via Lemma \ref{lem_M_Z}. Specifically, with $r = 1$ and when $\ee^{\frac{\pi\ii}{2}\alpha} = \frac{1}{s}(a+\ii b)$ for $\alpha \in (-2, 2]$,
\begin{equation}\label{eq_prop_symmetries_1}
	\mu(\Bff{M};(q,p))\comb_1 
	= \mathcal{S}_{(-\frac{q}{2}, -\frac{p}{2})}\mathcal{V}_{1/s}\mathcal{W}_t \mathcal{F}^\alpha \comb_1 
	= \mathcal{S}_{(-\frac{q}{2}, -\frac{p}{2})}\mathcal{V}_{1/s}\mathcal{W}_t \mathcal{F}^{\alpha+1}\mathcal{F}^{-1} \comb_1.
\end{equation}
By the Poisson summation formula \eqref{eq_Poisson_original}, $\mathcal{F}^{-1} \comb_1 = \ee^{\frac{\pi\ii}{4}}\comb_1$. Note that $\ee^{\frac{\pi\ii}{2}(\alpha + 1)} = \frac{1}{s}(-b + \ii a)$ and $t = ac + bd = (-b)(-d) + ac$ is unchanged under replacing $\Bff{M}$ with $\Bff{M}\Bff{F}_1$. However, if $b \geq 0$ and $a < 0$, then $\alpha + 1$ is no longer in $(-2, 2]$. To correct for this, we replace $\mathcal{F}^{\alpha + 1}$ by $\mathcal{F}^{\alpha - 3} = -\mathcal{F}^{\alpha+1}$. Therefore
\[
	\mathcal{V}_{1/s}\mathcal{W}_t \mathcal{F}^{\alpha+1} = \sigma\mathcal{M}(\Bff{M}\Bff{F}_1)
\]
and \eqref{eq_prop_symmetries_1} becomes
\[
	\mu(\Bff{M};(q,p))\comb_1 = \sigma\ee^{\frac{\pi\ii}{4}}\mathcal{S}_{(-\frac{q}{2}, -\frac{p}{2})}\mathcal{M}(\Bff{M}\Bff{F}_1) \comb_1 = \sigma\ee^{\frac{\pi\ii}{4}}\mu(\Bff{M}\Bff{F}_1; (q, p))\comb_1,
\]
proving the first equality in the proposition.

For the second and third equalities, we use the observation that
\begin{equation}\label{eq_M_Wj}
	\mathcal{M}(\Bff{M}) = \mathcal{W}_{-j}\mathcal{M}(\Bff{W}_j \Bff{M}) = \mathcal{M}(\Bff{M}\Bff{W}_j)\mathcal{W}_{-j}
\end{equation}
with the multiplication operator $\mathcal{W}_{-j} f(x) = \ee^{-\pi\ii j x^2}f(x)$ defined in \eqref{eq_def_op_Wt} and $\Bff{W}_j$ from \eqref{eq_def_canon_Wt}. Notice that, unlike composition with the Fourier transform, there is no change of sign; see Corollary \ref{cor_M_compose_VW}.

By the definition of $\mu(\Bff{M};(q,p))$ in Theorem \ref{thm_id_metapl}, the first equality in \eqref{eq_M_Wj}, and the Egorov relation \eqref{eq_Egorov} for $\mathcal{W}_j$ and $\Bff{W}_j$,
\begin{align*}
	\mu(\Bff{M}; (q,p))\comb_1 &= \mathcal{S}_{(\frac{q}{2}, \frac{p}{2})}\mathcal{M}(\Bff{M})\comb_1 
	\\ & = \mathcal{S}_{(\frac{q}{2}, \frac{p}{2})}\mathcal{W}_{-j}\mathcal{M}(\Bff{W}_j\Bff{M})\comb_1 
	\\ & = \mathcal{W}_{-j}\mathcal{S}_{(\frac{q}{2}, \frac{p+jq}{2})}\mathcal{M}(\Bff{W}_j\Bff{M})\comb_1
\end{align*}
In order to cancel $\mathcal{W}_{-j}$, we observe that
\begin{equation}\label{eq_Wj_shift_comb}
	\mathcal{W}_j \comb_1 = \mathcal{S}_{(0, \frac{j}{2})}\comb_1, \quad j \in \Bbb{Z},
\end{equation}
because $\ee^{\pi\ii j x^2} = \ee^{\pi\ii j x}$ for $x \in \opnm{supp} \comb_1 = \Bbb{Z}$. Therefore $\comb_1 = \mathcal{S}_{(0, -\frac{j}{2})}\mathcal{W}_j \comb_1$, and our computation becomes
\begin{align*}
	\mu(\Bff{M}; (q,p))\comb_1 
	& = \mathcal{S}_{(0, -\frac{j}{2})}\mathcal{W}_j \mu(\Bff{M}; (q,p)) \comb_1 
	\\ & = \mathcal{S}_{(0, -\frac{j}{2})}\mathcal{W}_j \mathcal{W}_{-j}\mathcal{S}_{(\frac{q}{2}, \frac{p+jq}{2})}\mathcal{M}(\Bff{W}_j\Bff{M})\comb_1 
	\\ & = \mathcal{S}_{(0, -\frac{j}{2})}\mathcal{S}_{(\frac{q}{2}, \frac{p+jq}{2})}\mathcal{M}(\Bff{W}_j\Bff{M})\comb_1 
	\\ & = \ee^{-\frac{\pi\ii}{4} jq}\mathcal{S}_{(\frac{q}{2}, \frac{p+jq}{2})}\mathcal{M}(\Bff{W}_j\Bff{M})\comb_1
\end{align*}
by the composition law \eqref{eq_shift_composition}. By the definition of $\mu$ in Theorem \ref{thm_id_metapl}, this proves
\[
	 \mu(\Bff{M}; (q,p)) = \ee^{-\frac{\pi\ii}{4} jq} \mu(\Bff{M}; (q, p+j(q-1)))
\]
and therefore the second equality in the proposition.

The final equality in the proposition follows directly from the definition of $\mu$ in Theorem \ref{thm_id_metapl}, the second equality in \eqref{eq_M_Wj}, the observation \eqref{eq_Wj_shift_comb}, the Egorov theorem for $\mathcal{M}(\Bff{M}\Bff{W}_j)$, and the composition law \eqref{eq_shift_composition}:
\begin{align*}
	\mu(\Bff{M}; (q,p))\comb_1 
	& = \mathcal{S}_{(\frac{q}{2}, \frac{p}{2})}\mathcal{M}(\Bff{M})\comb_1 
	\\ &= \mathcal{S}_{(\frac{q}{2}, \frac{p}{2})}\mathcal{M}(\Bff{M}\Bff{W}_j)\mathcal{W}_{-j}\comb_1 
	\\ & = \mathcal{S}_{(\frac{q}{2}, \frac{p}{2})}\mathcal{M}(\Bff{M}\Bff{W}_j)\mathcal{S}_{(0, -\frac{j}{2})}\comb_1 
	\\ & = \mathcal{S}_{(\frac{q}{2}, \frac{p}{2})}\mathcal{S}_{-\frac{j}{2}(b, d)}\mathcal{M}(\Bff{M}\Bff{W}_j)\comb_1 
	\\ & = \ee^{-\frac{\pi\ii}{4}j(pb-dq)}\mathcal{S}_{\frac{1}{2}((q,p) - (jb, jd))}\mathcal{M}(\Bff{M}\Bff{W}_j)\comb_1 
	\\ & = \ee^{\frac{\pi\ii}{4}j(dq-pb)}\mu(\Bff{M}\Bff{W}_j; (q,p) - (jb,jd))\comb_1.
\end{align*}
This completes the proof of the proposition.
\end{proof}

\begin{remark}\label{rem_t_rep}
Another view of the second equality in Proposition \ref{prop_mu_transformations} is that multiplying $\mathcal{M}(\Bff{M})$ by $\mathcal{W}_j$ is equivalent to changing the choice of $t$ in Theorem \ref{thm_identify}. Indeed, 
\[
	\mathcal{W}_j \mathcal{V}_{1/s}\mathcal{W}_t = \mathcal{V}_{1/s}\mathcal{W}_{t+js^2},
\]
and replacing $t$ by $t' = ac' + bd'$ in Theorem \ref{thm_identify} (taking $r = 1$ for simplicity) is possible if and only if $ad' - bc' = ad - bc$, meaning that $(c' - c, d' - d) = j(b, a)$ for some $j \in \Bbb{Z}$ because $a$ and $b$ are relatively prime. In this case, $t' = t + js^2$, in correspondance with the second equality in Proposition \ref{prop_mu_transformations}.
\end{remark}

\begin{remark}\label{rem_conj}
There are many symmetries which we do not use in proving the functional equation. One worth remarking on is
\[
	\mathcal{F}^{-\alpha}\comb_1 = \overline{\mathcal{F}^{\alpha}\comb_1},
\]
which follows immediately from the fact that whenever $f \in \mathscr{S}(\Bbb{R})$ is real-valued, $\overline{\mathcal{F}^\alpha f} = \mathcal{F}^{-\alpha}f$ and $\comb_1 f \in \Bbb{R}$ as well. This allows us to compute for any $f \in \mathscr{S}(\Bbb{R})$ real-valued (see Remark \ref{rem_unitary})
\[
	\overline{\langle f, \mathcal{F}^\alpha\rangle} = \langle \overline{\mathcal{F}^{-\alpha} f}, \comb_1\rangle = \langle \mathcal{F}^{\alpha} f, \comb_1\rangle = \langle f, \mathcal{F}^{-\alpha} \comb_1\rangle.
\]
This corresponds to the symmetry
\[
	\mu\left(\begin{pmatrix} a & b \\ c & d \end{pmatrix}; (q, p)\right) = \mu\left(\begin{pmatrix} a & -b \\ -c & d \end{pmatrix}; (q, -p)\right)^{-1}.
\]
\end{remark}

\subsection{Algorithm to identify $\mu$}

We recall the method in the proof of \cite[Theorem~7.1]{Mumford_1983} to identify the coefficient in the functional equation for the Jacobi theta function, adapted for our setting.

Let $a,b,c,d,q,p \in \Bbb{Z}$ be such that $ad-bc = 1$ and $(ab, cd) = (q,p)~(\opnm{mod} 2)$. We can choose $j$ such that $|a+jb| \leq |b|/2$; we apply the third equality in Proposition \ref{prop_mu_transformations} to replace $(a,b,c,d)$ with $(a+jb, b, c+jd, d)$. The first equality (the Poisson summation formula) allows us to replace $(a,b)$ with $(-b, a)$, and we continue until $(a,b) = (\pm 1, 0)$. In this case $ab = 0$, so we can replace $(p,q)$ with $(0, q)$ by multiplying by $\ee^{-\frac{\pi\ii}{4}pq}$ in view of \eqref{eq_mu_change_rep}. Finally, the observation \eqref{eq_Wj_shift_comb} and $\shift_{(0,j)}\comb_1 = \comb_1$ for all $j \in \Bbb{Z}$ gives $\mu(1, 0, c, 1; 0, q) = 1$ for all $c \in \Bbb{Z}$ and $q \equiv c (\opnm{mod} 2)$. Applying the first equality of Proposition \ref{prop_mu_transformations} twice gives that $\mu(-1, 0, c, -1; 0, q) = -\ii$. This gives us the value of $\mu$.

To make this procedure concrete, we present it in a few lines of code in Python, which returns $K \in \Bbb{Z}$ such that $\mu = \ee^{\frac{\pi\ii}{4}K}$.

\begin{verbatim}
def logmu(a,b,c,d,q,p):
    K = 0
    while b != 0:
        j = -a//b
        K = K + j*(d*q - p*b) 	# Third equality
        a, c, p, q = a+j*b, c+j*d, q-j*b, p-j*d 	# Third equality
        if a < 0 and b >= 0:
            K = K - 3 	# First equality, sigma = -1
        else:
            K = K + 1 	# First equality, sigma = +1
        a,b,c,d = -b,a,-d,c 	# First equality
    K = K+p*q 	# Reduction from (q, p) to to (0, q)
    if a == -1:
        K = K - 2	# If a = d = -1
    return(K%8)
\end{verbatim}

\subsection{A remark on higher dimensions}\label{ss_dim_high}

We have focused here on dimension one because there is a natural continuous one-parameter group of transformations that the author wishes to study. When rephrased in terms of the metaplectic representation of $Sp(2n, \Bbb{Z})$, very little changes in higher dimension. We describe the analogous result rapidly without attempting any deeper study. 

In this section, $n \geq 1$ is the dimension and
\[
	\comb_1(x) = \sum_{k \in \Bbb{Z}^n}\delta(x-k).
\]

We recall that $Sp(2n, \Bbb{R})$ is the set of linear transformations of $\Bbb{R}^{2n}$ leaving invariant the symplectic form
\[
	\sigma((x,\xi), (y, \eta)) = \xi\cdot y - \eta \cdot x, \quad (x,\xi), (y,\eta) \in \Bbb{C}^{2n}
\]
(so $\Bff{M} \in Sp(2n, \Bbb{R})$ if $\sigma(\Bff{M}\Bff{v}, \Bff{M}\Bff{w}) = \sigma(\Bff{v}, \Bff{w})$ for all $\Bff{v}, \Bff{w} \in \Bbb{C}^{2n}$). Shifts are defined identically,
\[
	\shift_{(x_0, \xi_0)}f(x) = \ee^{-\pi\ii x_0\cdot \xi_0 + 2\pi\ii \xi_0 x}f(x-x_0), \quad (x_0, \xi_0) \in\Bbb{R}^{2n}.
\]
One has the composition law $\shift_{\Bff{v}}\shift_{\Bff{w}} = \ee^{\pi\ii\sigma(\Bff{v}, \Bff{w})}\shift_{\Bff{v}+\Bff{w}} = \ee^{2\pi\ii\sigma(\Bff{v}, \Bff{w})}\shift_{\Bff{w}}\shift_{\Bff{v}}$ and that $\shift_{(j, k)}\comb_1 = \ee^{-\pi\ii j\cdot k}\comb_1$ whenever $j, k \in \Bbb{Z}^n$.

When
\[
	\Bff{M} = \begin{pmatrix} A & B \\ C & D \end{pmatrix} \in Sp(2n, \Bbb{R})
\]
for $A, B, C, D \in \Bff{M}_{n\times n}(\Bbb{R})$, and when $\det B \neq 0$, we may similarly define the metaplectic operator
\[
	\mathcal{M}(\Bff{M})f(x) = (\det (-\ii B))^{-1/2}\int \ee^{\pi\ii(DB^{-1}x\cdot x -2B^{-1}x\cdot y + B^{-1}Ay\cdot y)}f(y)\,\dd y
\]
(where the author has no preferred choice of sign for the square root). When $\det B = 0$ one may concoct a composition of multiplication by Gaussians with imaginary phase and changes of variables. One may alternatively view the metaplectic representation through its generators: the Fourier transform in the first variable, changes of variables, and multiplication by Gaussians with imaginary phase (see for instance \cite[Lem.~18.5.9]{Hormander_ALPDO3}). As a final approach, one may define a metaplectic operator through its effect on Gaussians as in Remark \ref{rem_gaussians_suffice}.

With any approach, one has a family of operators $\mathcal{M}(\Bff{M})$ corresponding two-to-one with $Sp(2n, \Bbb{R})$; these operators are unitary on $L^2(\Bbb{R}^n)$, isomorphisms of $\mathscr{S}(\Bbb{R}^n)$ and $\mathscr{S}'(\Bbb{R}^n)$, and the Egorov relation
\[
	\mathcal{M}(\Bff{M})\shift_{\Bff{v}} = \shift_{\Bff{Mv}}\mathcal{M}(\Bff{M}), \quad \forall \Bff{v}\in \Bbb{R}^{2n}.
\]

We simply record that an obvious analogue of Theorem \ref{thm_id_metapl} holds in any dimension.

\begin{theorem}
For $n \geq 1$, let $\Bff{M} = \begin{pmatrix} A & B \\ C & D\end{pmatrix} \in Sp(2n, \Bbb{Z})$. Define $q_j^0$ as the scalar product of the $j$-th rows of $A$ and of $B$ (equivalently, the $j$-th diagonal entry of $AB^\top$) and let $p_j^0$ be the scalar product of the $j$-th rows of $C$ and of $D$. Let $q^0 = (q_j^0)_{j=1}^n$ and let $p^0 = (p_j^0)_{j=1}^n$. Then, when $(q, p) \in \Bbb{R}^{2n}$, when $\shift_{(\frac{q}{2}, \frac{p}{2})}$ is the corresponding shift operator, when $\mathcal{M}(\Bff{M})$ is one of the two metaplectic operators corresponding to $\Bff{M}$, and when $\comb_1$ is the Dirac comb, there exists $\mu$ such that
\begin{equation}\label{eq_multid}
	\shift_{(\frac{q}{2}, \frac{p}{2})}\mathcal{M}(\Bff{M}) \comb_1 = \mu \comb_1
\end{equation}
if and only if $(q, p) \equiv (q^0, p^0)~(\opnm{mod} 2)$, and in this case $\mu^8 = 1$.
\end{theorem}

\begin{proof}
Let $e_j \in \Bbb{R}^n$ be the basis vector which has zeros for entries except for a one in the $j$-th position. Just as in the proof of Theorem \ref{thm_id_metapl}, following \cite[Sec.~7.2]{Hormander_ALPDO1}, existence of $\mu \in \Bbb{C}$ follows from showing that
\[
	\shift_{(e_j, 0)}\shift_{(\frac{q}{2}, \frac{p}{2})}\mathcal{M}(\Bff{M}) \comb_1 = \shift_{(0, e_j)}\shift_{(\frac{q}{2}, \frac{p}{2})}\mathcal{M}(\Bff{M}) \comb_1 = \shift_{(\frac{q}{2}, \frac{p}{2})}\mathcal{M}(\Bff{M}) \comb_1
\]
for $j = 1, \dots, g$.

We compute using the Egorov relation that
\[
	\shift_{(e_j, 0)}\shift_{(\frac{q}{2}, \frac{p}{2})}\mathcal{M}(\Bff{M}) \comb_1 = \ee^{-\pi\ii p_j}\shift_{(\frac{q}{2}, \frac{p}{2})}\mathcal{M}(\Bff{M})\shift_{\Bff{v}_j} \comb_1
\]
where $p = (p_j)_{j=1}^n$ and $\Bff{v}_j$ is the $j$-th column of $\Bff{M}^{-1}$. Since $\Bff{M} \in Sp(2n, \Bbb{Z})$, we can write \cite[Sect.~1.1.9]{Lion_Vergne_1980}
\[
	\Bff{M}^{-1} = \begin{pmatrix} D^\top & -B^\top \\ -C^\top & A^\top\end{pmatrix}.
\]
When $c_{k\ell}$ and $d_{k\ell}$ are the entries of $C$ and $D$, 
\[
	\Bff{v}_j = (d_{j1}, \dots, d_{jn}, -c_{j1}, \dots, -c_{jn}, )
\]
and, since $p^0_j$ was defined to be $\sum_{j=1}^n c_{j1}d_{j1}$,
\[
	\shift_{(e_j, 0)}\shift_{(\frac{q}{2}, \frac{p}{2})}\mathcal{M}(\Bff{M}) \comb_1 = \ee^{\pi\ii (p^0_j - p_j)}\shift_{(\frac{q}{2}, \frac{p}{2})}\mathcal{M}(\Bff{M})\shift_{\Bff{v}_j} \comb_1.
\]

Similarly,
\[
	\shift_{(0, e_j)}\shift_{(\frac{q}{2}, \frac{p}{2})}\mathcal{M}(\Bff{M}) \comb_1 = \ee^{-\pi\ii (q^0_j + q_j)}\shift_{(\frac{q}{2}, \frac{p}{2})}\mathcal{M}(\Bff{M})\shift_{\Bff{v}_j} \comb_1.
\]
Therefore there exists $\mu \in \Bbb{C}$ satisfying \ref{eq_multid}.

The fact that $\mu^8 = 1$ follows from the fact \cite[Prop.\ A 5]{Mumford_1983} that $Sp(2n, \Bbb{Z})$ is generated by
\[
	\begin{pmatrix} 0 & I \\ -I & 0\end{pmatrix}, \quad \begin{pmatrix} A^{-1} & 0 \\ 0 & A^\top\end{pmatrix}, \quad \begin{pmatrix} I & 0 \\ B & I\end{pmatrix},
\]
where $A \in GL(n, \Bbb{Z})$ and $B \in \Bbb{M}_{n\times n}(\Bbb{Z})$ is symmetric. These correspond to the Fourier transform, changes of variables $\mathcal{V}_{A}f(x) = (\det A)^{1/2}f(Ax)$, and multiplication by Gaussians with imaginary phase $\mathcal{W}_{B}f(x) = \ee^{\pi\ii x\cdot Bx}f(x)$. The Fourier transform gives an eighth root of unity $\ee^{-\frac{\pi\ii}{4}n}$ by the Poisson summation formula, and the changes of variables give fourth roots of unity (depending on the choice of signs for square roots) because $\det A = \pm 1$. 

Multiplication by Gaussians corresponds to a half-integer shift: when $b_\Delta = (b_{jj})_{j=1}^n$ is the diagonal of $B = (b_{jk})_{j,k=1}^n$,
\[
	\mathcal{W}_B \comb_1 = \shift_{(0, \frac{b_\Delta}{2})}\comb_1.
\]
This is because $B$ is symmetric and, for $x \in \Bbb{Z}^n$,
\[
	x\cdot Bx = \sum_{j=1}^n b_{jj}x_j^2 + 2\sum_{1 \leq j < k \leq n} b_{jk}x_jx_k \equiv \sum_{j=1}^g b_{jj}x_j^2 \quad (\opnm{mod} 2),
\]
so
\[
	\ee^{\pi\ii (Bx \cdot x - b_\Delta \cdot x)}\comb_1 = \comb_1.
\]
As in dimension one, commutators and compositions of half-integer shifts produce eighth roots of unity (and only eighth roots of unity).

Having shown that generators of the metaplectic representation, when composed with appropriate half-integer shifts, act on the Dirac comb by multiplication by eighth roots of unity, we have proved the theorem.
\end{proof}

\section{The Bargmann transform}\label{s_Bargmann}

The Bargmann transform \cite{Bargmann_1961}
\begin{equation}\label{eq_def_Bargmann}
	\mathfrak{B}f(z) = \int \ee^{-\pi x^2 + 2\pi z x - \frac{\pi}{2}z^2}f(x)\,\dd x
\end{equation}
plays a central role in the analysis of the quantum harmonic oscillator $Q_0$ defined in \eqref{eq_def_Schro_FrFT}. (A good reference is \cite[Chap.\ 1.6]{Folland_1989}.) The Bargmann transform, which can be formally viewed as the metaplectic operator associated with
\begin{equation}\label{eq_canon_Bargmann}
	\Bff{B} = \begin{pmatrix} 1 & -\ii \\ -\frac{\ii}{2} & \frac{1}{2} \end{pmatrix},
\end{equation}
is a unitary map from $L^2(\Bbb{R})$ to the space $\mathfrak{F}$ of holomorphic functions $u:\Bbb{C} \to \Bbb{C}$ for which
\[
	\|u\|_{\mathfrak{F}}^2 = \int_{\Bbb{C}} |u(z)|^2\,2^{1/2}\ee^{-\pi|z|^2}\,\dd \Re z \,\dd \Im z
\]
is finite. In this section, we show that the Bargmann transform allows us to prove Theorems \ref{thm_identify} and \ref{thm_mu_8} as a consequence of the classical functional equation in Theorem \ref{thm_theta}. As in Section \ref{ss_proofs_elementary}, we write
\begin{equation}\label{eq_Bargmann_rho}
	\rho := \ee^{\frac{\pi\ii}{2}\alpha} = \rho_1 + \ii \rho_2, \quad \rho_1, \rho_2\in \Bbb{R}.
\end{equation}
Supposing that $\cos\frac{\pi\alpha}{2}$ and $\sin \frac{\pi\alpha}{2}$ are linearly independent over $\Bbb{Z}$, we write $\rho = \frac{1}{s}(a+\ii b)$ for $a,b\in\Bbb{Z}$ linearly independent and $s = \sqrt{a^2 + b^2} > 0$ (Remark \ref{rem_rho_s}).

Our goal being to present the proof as rapidly as possible, we restrict ourselves to the case where $r = 1$ and where $ab$ and $cd$ are both even. In this case, Theorems \ref{thm_identify} and \ref{thm_mu_8} reduce to the existence of some $\mu_0\in\Bbb{C}$ for which $\mu_0^8 = 1$ and
\begin{equation}\label{eq_Bargmann_identify}
	\mathcal{F}^\alpha \comb_1 = \mu_0\ee^{-\pi \ii t x^2}\comb_{1/s}(x).
\end{equation}
We leave it to the interested reader to extend the computation to general $r > 0$ (which is quite straightforward) and general $a, b, c, d$ such that $ad-bc = 1$, which can be done using symmetries of theta-functions as in \cite[Table 0, p.~19]{Mumford_1983}. We also rely on direct computation; applications of the metaplectic theory, while clearly interesting to the author, are presented elsewhere in this work.

The Bargmann transform of the Dirac comb $\comb_1$ is
\begin{equation}\label{eq_Bargmann_comb}
	\mathfrak{B}\comb_1(z) = \sum_{k\in\Bbb{Z}} \ee^{-\frac{\pi}{2}z^2 + 2\pi z k - \pi k^2} = \ee^{-\frac{\pi}{2}z^2}\theta_{00}(-\ii z, \ii)
\end{equation}
for $\theta_{00}(z, \tau) = \sum_{k\in\Bbb{Z}}\ee^{2\pi z k + \pi\ii\tau k^2}$ as in Section \ref{ss_theta}, and
\begin{equation}\label{eq_Bargmann_rhs}
	\mathfrak{B}(\ee^{-\ii t (\cdot)^2}\comb_{1/s})(z) = s^{-1/2}\ee^{-\frac{\pi}{2}z^2}\theta_{00}\left(-\frac{\ii}{s}z, \frac{\ii - t}{s^2}\right).
\end{equation}

Conjugation by $\mathfrak{B}$ simplifies the harmonic oscillator because
\begin{equation}\label{eq_Bargmann_HO}
	\mathfrak{B}Q_0\mathfrak{B}^* = z\cdot \frac{\dd}{\dd z} + \frac{1}{2},
\end{equation}
and therefore, recalling \eqref{eq_Bargmann_rho} and using the shorthand $\rho^{-1/2}=\ee^{-\frac{\pi\ii}{4}\alpha}$,
\begin{equation}\label{eq_Bargmann_FrFT}
	\mathfrak{B}\mathcal{F}^\alpha\mathfrak{B}^*v(z) = \mathfrak{B}\ee^{-\frac{\pi\ii}{2}\alpha Q_0}\mathfrak{B}^*v(z) = \rho^{-1/2}v(z/\rho).
\end{equation}
The former equality follows from direct computation or an Egorov theorem for Weyl quantizations; the latter can be seen because the problem $\left(\ii \partial_t + z\frac{\dd}{\dd z}\right)U(t,z) = 0$ is easy to solve. (See for instance \cite{Aleman_Viola_2014a, Aleman_Viola_2018}.)

In Figure \ref{fig_Bargmann} we draw contours of the imaginary part of the normalized Bargmann transform $\Im(\mathfrak{B}\comb_1(z)\ee^{-\frac{\pi}{2}|z|^2})$; this is normalized in the sense that 
\[
	\|\mathfrak{B}\comb_1\|_{\mathfrak{F}} = \|\ee^{-\frac{\pi}{2}|z|^2}\mathfrak{B}\comb_1(z)\|_{L^2(\Bbb{C})}.
\] 
Note that this function is bounded by Proposition \ref{prop_Bargmann_mass_periodic}. By \eqref{eq_Bargmann_FrFT}, we are, in effect, studying symmetries under rotation of this function.

\begin{figure}
\centering
\includegraphics[width = \textwidth]{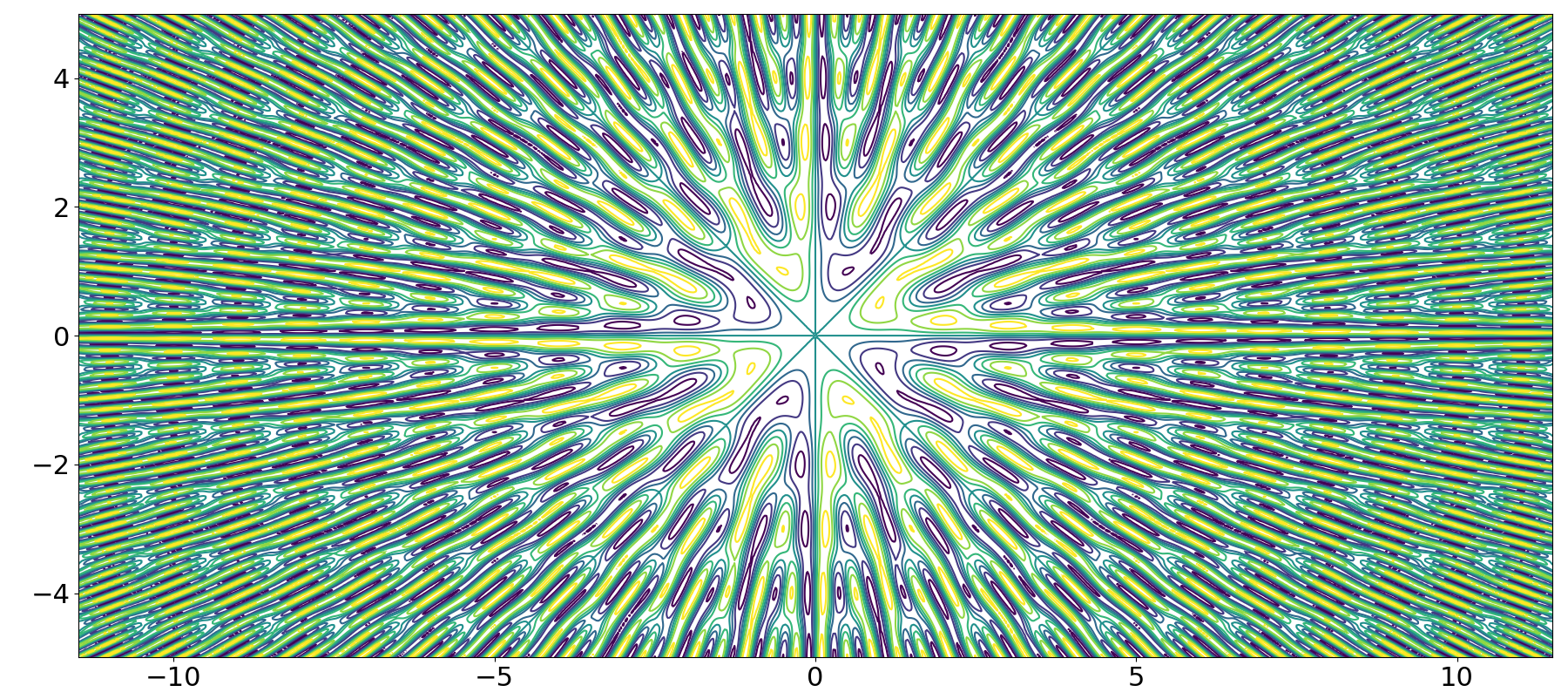}
\caption{Contours of the imaginary part of the normalized Bargmann transform, $\Im(\ee^{-\frac{\pi}{2}|z|^2}\mathfrak{B}\comb_1(z))$ as a function of $z$ in the complex plane.\label{fig_Bargmann}}
\end{figure}

\begin{example}
The Poisson summation formula \eqref{eq_Poisson_original}, for $r = 1$ and on the Bargmann side, reads that
\[
	\ee^{-\frac{\pi\ii}{4}}\mathfrak{B}\comb_1(z/\ii) = \ee^{-\frac{\pi\ii}{4}}\mathfrak{B}\comb_{1}(z).
\]
This corresponds to the symmetry under rotation by $\pi/2$ of the Bargmann transform of $\comb_1$, illustrated in Figure \ref{fig_Bargmann}.
\end{example}

We apply the (unitary) Bargmann transform to \eqref{eq_Bargmann_identify}. Applying \eqref{eq_Bargmann_FrFT} to \eqref{eq_Bargmann_comb} for the left-hand side and using \eqref{eq_Bargmann_rhs} as the right-hand side, we see that Theorem \ref{thm_identify} (for $r = 1$ and $ab, cd$ even) is equivalent to
\[
	\rho^{-1/2}\ee^{-\frac{\pi}{2}(z/\rho)^2}\theta_{00}\left(-\frac{\ii}{\rho} z, \ii\right) = \mu_0 s^{-1/2}\ee^{-\frac{\pi}{2}z^2}\theta_{00}\left(-\frac{\ii}{s}z, \frac{\ii - t}{s^2}\right).
\]
Setting $z_0 = -\ii z/\rho$, we rewrite this equality as
\begin{equation}\label{eq_Bargmann_identity_2}
	\theta_{00}\left(\frac{\rho}{s}z_0, \frac{\ii - t}{s^2}\right) = \frac{1}{\mu_0}\left(\frac{\rho}{s}\right)^{-1/2} \ee^{\frac{\pi}{2}(1-\rho^2)z_0^2} \theta_{00}(z_0, \ii).
\end{equation}

To see that this is a consequence of the classical functional equation presented in Theorem \ref{thm_theta}, we analyze $\rho/s$, $1-\rho^2$, and $(\ii-t)/s^2$. Recalling that $\rho^{-1} = \bar{\rho} = \frac{1}{s}(a-\ii b)$,
\[
	\frac{\rho}{s} = \frac{1}{a-\ii b}.
\]
Next,
\[
	1-\rho^2 = 1 - \frac{(a+\ii b)^2}{a^2 + b^2} = \frac{2b^2 - 2\ii a b}{a^2 + b^2} = -\frac{2\ii b(a+\ii b)}{a^2 + b^2} = -\frac{2\ii b}{a-\ii b}.
\]
Finally, using that $ad-bc = 1$,
\[
	\frac{\ii - t}{s^2} = \frac{\ii(ad-bc) - (ac + bd)}{a^2 + b^2} = \frac{(-c+\ii d)(a+\ii b)}{a^2 + b^2} = \frac{-c + \ii d}{a - \ii b}.
\]
These three computations, when inserted into \eqref{eq_Bargmann_identity_2}, give that Theorem \ref{thm_identify} for $r = 1$ and $ab, cd$ even is equivalent to
\[
	\theta_{00}\left(\frac{z_0}{a-\ii b}, \frac{-c + \ii d}{a-\ii b}\right) = \frac{1}{\mu_0}(a-\ii b)^{1/2} \ee^{\frac{\pi\ii(-b)}{a-\ii b}z_0^2}\theta_{00}(z_0, \ii).
\]
This, in turn, is a special case of the well-known Theorem \ref{thm_theta}.

An observation, useful in controlling the regularity of $\comb_r$ (Proposition \ref{prop_weak_continuity}), is that the Bargmann mass $\ee^{-\pi|z|^2}|\mathfrak{B}\comb_r|^2$ is bounded. This follows from the fact that it is a smooth periodic function.

\begin{proposition}\label{prop_Bargmann_mass_periodic}
For $\mathfrak{B}\comb_r$ the Bargmann transform \eqref{eq_def_Bargmann} of a Dirac comb \eqref{eq_def_comb}, for all $z \in \Bbb{C}$ and $w \in r\Bbb{Z} + \ii\frac{1}{r}\Bbb{Z}$,
\[
	\ee^{-\pi|z-w|^2}|\mathfrak{B}\comb_r(z-w)|^2 = \ee^{-\pi|z|^2}|\mathfrak{B}\comb_r(z)|^2.
\]
\end{proposition}

\begin{proof}
We translate \eqref{eq_shift_integer_comb} onto the Bargmann side using the canonical transformation \eqref{eq_canon_Bargmann}:
\[
	\shift_{(rj-\ii k/r, \frac{1}{2}(-\ii rj + k/r))}\mathfrak{B}\comb_r = \ee^{-\pi\ii jk} \mathfrak{B}\comb_r, \quad \forall j,k\in\Bbb{Z}.
\]
If we set $w = rj + \ii k/r \in \Bbb{Z} + \ii \Bbb{Z}$ an arbitrary Gaussian integer, we obtain
\[
	\shift_{(\bar{w}, \frac{1}{2\ii}w)}\mathfrak{B}\comb_r = \ee^{-\frac{\pi\ii}{2} \Im(w^2)}\mathfrak{B}\comb_r
\]
or
\[
	\mathfrak{B}\comb_r(z) = \ee^{-\frac{\pi}{2} |w|^2 + \pi\bar{w}z + \frac{\pi\ii}{2}\Im(w^2)}\mathfrak{B}\comb_r(z-w).
\]

We notice that 
\[
	|z-w|^2 = |z|^2 - 2\Re(\bar{w} z) + |w|^2,
\]
which allows us to write
\[
	\ee^{-\frac{\pi}{2}|z|^2}\mathfrak{B}\comb_r(z) = \ee^{-\frac{\pi}{2}|z-w|^2 + \pi \ii \Im\left(\bar{w}z + \frac{1}{2}w^2\right)}\mathfrak{B}\comb_r(z-w).
\]
The result follows by taking the absolute value squared.
\end{proof}

\section{Periodicity and parity}\label{s_further}

In Section \ref{s_irrational} we observe (without proof) that approximations to $\mathcal{F}^{\alpha}\comb_1$ when $\cot\frac{\pi\alpha}{2} \in \Bbb{R}\backslash \Bbb{Q}$ seem to be at many points nearly periodic, nearly odd, or nearly even. In this section, we consider when $\mathcal{F}^\alpha\comb_1$ is exactly periodic, exactly odd about a point, or exactly even about a point, which is significantly more restrictive.

\subsection{Periodicity and crystalline measures}

To better understand $\mathcal{F}^\alpha \comb_r$, a natural question is whether this function is periodic, as for instance $\mathcal{F}^{1/2}\comb_1$ plainly is (Figure \ref{fig_intro}). Another natural question is whether $\mathcal{F}^\alpha \comb_r$ is a crystalline measure, meaning that the support of it and its Fourier transform are both discrete. We group these two questions together because the answers are the same.

\begin{proposition}
Let $\alpha \in \Bbb{R}$ and $r > 0$. Let the Dirac comb $\comb_r$ be as in \eqref{eq_def_comb}, and let the fractional Fourier transform $\mathcal{F}^\alpha$ be as in \eqref{eq_def_Schro_FrFT}. Then $\mathcal{F}^\alpha \comb_r$ is periodic if and only if the support of $\mathcal{F}(\mathcal{F}^\alpha \comb_r)$ is discrete. Moreover, if $\alpha \notin \Bbb{Z}$ and if the support of $\mathcal{F}^\alpha \comb_r$ is discrete (as in Theorem \ref{thm_support}), then the support of $\mathcal{F}(\mathcal{F}^\alpha \comb_r)$ is also discrete if and only if $r^4 \in \Bbb{Q}$.
\end{proposition}

\begin{proof}
The distribution $\mathcal{F}^\alpha \comb_r$ is periodic if and only if there exists some $x_0 \in \Bbb{R}\backslash \{0\}$ such that
\[
	\mathcal{S}_{(x_0, 0)} \mathcal{F}^\alpha \comb_r = \mathcal{F}^\alpha \comb_r.
\]
By the Egorov relation \eqref{eq_Egorov} for $\mathcal{F}^\alpha$ given by \eqref{eq_def_F_canon}, this is equivalent to
\[
	\mathcal{F}^\alpha \mathcal{S}_{x_0(\cos\frac{\pi\alpha}{2}, \sin\frac{\pi\alpha}{2})}\comb_r = \mathcal{F}^\alpha \comb_r.
\]
Since we may eliminate the operator $\mathcal{F}^\alpha$ from both sides, this is possible if and only if 
\begin{equation}\label{eq_periodicity_suff}
	x_0(\cos\frac{\pi\alpha}{2}, \sin\frac{\pi\alpha}{2}) = (rm, n/r)
\end{equation}
for $m, n\in\Bbb{Z}$ for which $mn$ is even; see \eqref{eq_shift_integer_comb}. Because 
\[
	\left(\cos(\theta+\frac{\pi}{2}), \sin(\theta+\frac{\pi}{2})\right) = (-\sin\theta, \cos \theta),
\]
\[
	m r\cos\frac{\pi(\alpha+1)}{2} + n\frac{1}{r}\sin\frac{\pi(\alpha+1)}{2} = \frac{1}{x_0}(-mn + nm) = 0,
\]
so $\mathcal{F}(\mathcal{F}^\alpha \comb_r) = \mathcal{F}^{\alpha+1}\comb_r$ has discrete support by Theorem \ref{thm_identify}.

Conversely, if $\mathcal{F}(\mathcal{F}^\alpha \comb_r)$ has discrete support, then by Theorem \ref{thm_support} and Remark \ref{rem_rho_s}, there exists $m,n\in\Bbb{Z}$ such that
\[
	\ee^{\frac{\pi}{2}(\alpha+1)} = \frac{1}{s}(\frac{n}{r}+\ii mr), \quad s = \sqrt{(n/r)^2 + (mr)^2}.
\]
By the preceding discussion, in particular \eqref{eq_periodicity_suff}, this means that $\mathcal{F}^\alpha \comb_r$ is periodic with period $s = \sqrt{(n/r)^2 + (mr)^2}$.

Finally, if $\alpha \notin \Bbb{Z}$, then $\mathcal{F}^\alpha \comb_r$ has discrete support if and only if $\frac{1}{r^2}\cot \frac{\pi\alpha}{2} = q_1 \in \Bbb{Q}$, by Theorem \ref{thm_support}. In this case, $\mathcal{F}(\mathcal{F}^\alpha \comb_r)$ is discrete if and only if
\[
	\frac{1}{r^2}\cot\frac{\pi(\alpha+1)}{2} = -\frac{1}{r^2}\tan \frac{\pi\alpha}{2} = - \frac{1}{r^4 q_1} \in \Bbb{Q},
\]
which holds if and only if $r^4 \in \Bbb{Q}$.
\end{proof}

\subsection{Parity}

A function $f(x)$ is even/odd at $x_0 \in \Bbb{R}$ if $y\mapsto f(y+x_0)$ is even/odd, meaning that $f(y+x_0) = \pm f(-y + x_0)$. Recalling that $\ii \mathcal{F}^2 f(x) = f(-x)$, we say that a Schwartz distribution $f \in \mathscr{S}'(\Bbb{R})$ is even/odd at $x_0$ if
\[
	\ii\mathcal{F}^2 \shift_{(-x_0, 0)} f = \pm \shift_{(-x_0, 0)}f.
\]

Take $f = \mathcal{F}^\alpha \comb_r$. As usual, write $\ee^{\frac{\pi\ii}{2}\alpha} = \rho_1 + \ii \rho_2$, and recall the Egorov relation $\shift_{(1, 0)}\mathcal{F}^\alpha = \mathcal{F}^\alpha \shift_{(\rho_1, \rho_2)}$. We also use that $\comb_r$ is even in computing
\begin{align*}
	\ii \mathcal{F}^2 \shift_{(-x_0, 0)}\mathcal{F}^\alpha \comb_r &= \pm \shift_{(-x_0, 0)}\mathcal{F}^\alpha \comb_r
	\\ \iff \shift_{(x_0, 0)}\mathcal{F}^{\alpha}\ii\mathcal{F}^2 \comb_r &= \pm \shift_{(-x_0, 0)}\mathcal{F}^\alpha \comb_r
	\\  \iff \mathcal{F}^\alpha\shift_{x_0(\rho_1, \rho_2)}\comb_r &= \pm \mathcal{F}^\alpha\shift_{-x_0(\rho_1, \rho_2)}\comb_r
	\\ \iff \shift_{2x_0(\rho_1, \rho_2)} \comb_r &= \pm \comb_r.
\end{align*}
Recalling \eqref{eq_shift_integer_comb}, we obtain the following proposition.

\begin{proposition}\label{prop_parity}
Let $\alpha\in\Bbb{R}$ and write $\ee^{\frac{\pi\ii}{2}\alpha} = \rho_1 + \ii \rho_2$ for $\rho_1, \rho_2\in\Bbb{R}$. Let $\mathcal{F}^\alpha \comb_1$ be the fractional Fourier transform of the Dirac comb. Then $\mathcal{F}^\alpha \comb_r$ is even or odd around $x_0\in\Bbb{R}$ if and only if
\[
	2x_0(\rho_1, \rho_2) \in r\Bbb{Z}\times \frac{1}{r}\Bbb{Z},
\]
and in this case $\mathcal{F}^\alpha \comb_1$ is even if $4x_0^2\rho_1\rho_2$ is even and odd if $4x_0^2\rho_1\rho_2$ is odd.
\end{proposition}

\begin{example}
If $\alpha = \frac{1}{2}$ then $2(\rho_1, \rho_2) = (\sqrt{2}, \sqrt{2})$. Then $\mathcal{F}^{\frac{1}{2}}\comb_1$ is even at $x_0$ if and only if $x_0 = \frac{2k}{\sqrt{2}}$ and odd if and only if $x_0 = \frac{2k+1}{\sqrt{2}}$. This agrees with the formula identified and illustrated in Section \ref{ss_brushes}.
\end{example}

\section{Questions of continuity in $\alpha$}\label{s_continuity}

The most interesting aspect (to the author) of the study of the fractional Fourier transform of Dirac combs is the question of how the result varies as $\alpha$ varies.

As we will discuss, it is obvious that the function $\alpha \mapsto \mathcal{F}^\alpha \comb_1$ is far from being absolutely continuous from $\Bbb{R}$ to the set of measures, and it is obvious that it is continuous from $\Bbb{R}$ to $\mathscr{S}'(\Bbb{R})$.

We show that the function $\alpha \mapsto \mathcal{F}^\alpha\comb_1$ is continuous from $\Bbb{R}$ to $Q_0^{-p}L^2(\Bbb{R})$ for any $p > 1/2$, the image being essentially a Sobolev space defined by the quantum harmonic oscillator instead of $D_x$.

We also show that the antiderivative $\Pi_{\alpha}(x)$ (Definition \ref{def_antiderivative}) is not locally uniformly continuous in $\alpha$. It has, however, a particular behavior as $\alpha \to 0^+$ in certain regimes: the antiderivative tends to a rescaled Fresnel integral, leading to the appearance of Euler spirals when the values of $\Pi_\alpha(x)$ are traced in the complex plane.

Within these Euler spirals one observes repeating motifs depending on $b$ if $\tan \frac{\pi\alpha}{2} = \frac{b}{a+jb}$ and $j \to \infty$. We show that these repeating motifs are the graphs of Gauss sums, and one can use this observation to express the eighth root of unity $\mu$ in Theorem \ref{thm_mu_8} via a Gauss sum (which is classical in the study of theta functions).

\subsection{Lack of absolute continuity}

It is obvious that the absolute value of the $\mathcal{F}^\alpha \comb_1$ (when we can describe this measure via Theorems \ref{thm_support} and \ref{thm_identify}) diverges wildly. It suffices to observe that when $\ee^{\frac{\pi\ii}{2}\alpha} = \frac{1}{s}(a+\ii b)$ as in \eqref{eq_def_rho_s}, the absolute value (in the sense of measures) of $\mathcal{F}^\alpha \comb_1$ is
\[
	|\mathcal{F}^\alpha \comb_1(x)| = s^{-1/2}\sum_{k \in \Bbb{Z} + \{\frac{ab}{2}\}} \delta(x - \frac{k}{s}),
\]
which charges an interval of length $L$ with a mass of $Ls^{1/2}$ (up to the obvious rounding error depending on exactly where the endpoints of the interval fall). Because $\{a,b\in\Bbb{Z} \::\: L(a^2+b^2)^{1/4} < C\}$ is finite for any $L, C > 0$, we have the following proposition on divergence of the absolute value of $\mathcal{F}^\alpha \comb_1$ (which applies equally well to any $\mathcal{F}^\alpha \comb_r$ for $r > 0$ fixed).

\begin{proposition}
For any open set $U \subset \Bbb{R}$ and for any $C > 0$, the number of $\alpha \in (-2, 2]$ such that $\cos \frac{\pi\alpha}{2}$ and $\sin\frac{\pi\alpha}{2}$ are linearly dependent over $\Bbb{Z}$ and for which
\[
	\int_U |\mathcal{F}^\alpha \comb_1(x)|\,\dd x \leq C
\]
is finite.
\end{proposition}

\subsection{Weak continuity in $\alpha$}

When $Q_0$ the quantum harmonic oscillator defined in \eqref{eq_def_Q0}, we recall that the Schwartz space $\mathscr{S}(\Bbb{R})$ is simply 
\[
	\mathscr{S}(\Bbb{R}) = \{f\in L^2(\Bbb{R}) \::\: Q_0^N f \in L^2(\Bbb{R}),~~\forall N \in \Bbb{N}\},
\]
with the associated family of seminorms. Since $\mathcal{F}^\alpha$ is a function of $Q_0$, it is obvious that $\mathcal{F}^\alpha$ preserves the seminorms of $\mathscr{S}(\Bbb{R})$. Because $\comb_1\in\mathscr{S}'(\Bbb{R})$, by duality it is clear that $\alpha \mapsto \mathcal{F}^\alpha \comb_1$ is continuous from $\Bbb{R}$ to $\mathscr{S}'(\Bbb{R})$.

We would like to have more precise information on how much regularity (and decay) is required to make $\mathcal{F}^\alpha \comb_1$ continuous in $\alpha \in \Bbb{R}$.

\begin{proposition}\label{prop_weak_continuity}
Let the Dirac comb $\comb_r$ be as defined in \eqref{eq_def_comb} and let the harmonic oscillator $Q_0$ be as in \eqref{eq_def_Schro_FrFT}. For any $r > 0$,
\[
	Q_0^{-p}\comb_r \in L^2(\Bbb{R}), \quad \forall p > \frac{1}{2}.
\]
\end{proposition}

\begin{remark} For brevity, we will use the Bargmann transform to control $\ee^{-tQ_0}\comb_1$, using methods similar to those in \cite{Aleman_Viola_2014a, Aleman_Viola_2018}. An alternate proof, omitted here, is to explicitly compute $\ee^{-tQ_0}\comb_1$ using the Mehler kernel. For small $t > 0$, one obtains an almost orthogonal family of Gaussians and the sum of the norms gives the somewhat sharper estimate
\[
	\|\ee^{-tQ_0}\comb_1\|^2 = \frac{1}{2\sinh t}(1+\BigO(\ee^{-c/t}))
\]
from which one can deduce the result.

To quickly check the plausibility of the result of the proposition, heuristically, on the Bargmann side, $Q_0$ acts like $1+|z|^2$ (see, for instance, \cite[Prop.~B.2]{Viola_2012a}). Therefore we expect
\[
	\|Q_0^{-p}\comb_1\|^2 \approx \int_{\Bbb{C}} (1+|z|^2)^{-2p}|\mathfrak{B}\comb_r(z)|^2\ee^{-\pi|z|^2}\,\dd \Re z\, \dd \Im z
\]
which converges whenever $p > \frac{1}{2}$ because $\ee^{-\pi |z|^2}|\mathfrak{B}\comb_r(z)|^2 \in L^\infty(\Bbb{C})$ by Proposition \ref{prop_Bargmann_mass_periodic}.
\end{remark}

\begin{proof}
Because $\mathscr{S}(\Bbb{R}) = \{f \in L^2(\Bbb{R})\::\: Q_0^k f \in L^2(\Bbb{R})\textnormal{ for all }k \in \Bbb{N}\}$, we have that $Q_0^\beta$ acts continuously on $\mathscr{S}(\Bbb{R})$ and therefore on $\mathscr{S}'(\Bbb{R})$ which contains $\comb_r$. If $p > 0$ is such that the integral converges absolutely,
\begin{equation}\label{eq_continuity_Laplace}
	Q_0^{-p}f = \Gamma(p) \int_0^\infty t^{p - 1}\ee^{-tQ_0}f\,\dd t.
\end{equation}
If $Q_0 f = \lambda f$, this comes from
\[
	\int_0^\infty t^{p-1}\ee^{-t\lambda}\,\dd t = \lambda^{-p}\int_0^\infty s^{p-1}\ee^{-s}\,\dd s = \lambda^{-p}\Gamma(p),
\]
and for general functions this comes from the decomposition of $f$ into (Hermite) eigenfunctions of $Q_0$.

It therefore is enough to control $\ee^{-tQ_0}\comb_r$ which is in $\mathscr{S}(\Bbb{R})$ when $t > 0$.

Conjugating by the unitary Bargmann transform, we obtain via \eqref{eq_Bargmann_HO} and \eqref{eq_Bargmann_FrFT} that
\begin{align*}
	\|\ee^{-tQ_0}\comb_r\|^2 &= \|\ee^{-t(z\cdot \frac{\dd}{\dd z} + \frac{1}{2})}\mathfrak{B}\comb_r\|_{\mathfrak{F}}^2
	\\ &= \|\ee^{-t/2}\mathfrak{B}\comb_r(\ee^{-t}z)\|_{\mathfrak{F}}^2
	\\ &= \int \ee^{-t}|\mathfrak{B}\comb_r(\ee^{-t}z)|^2\ee^{-\pi|z|^2}\,\dd \Re z \, \dd \Im z
	\\ &= \ee^{t}\int |\mathfrak{B}\comb_r(z)|^2\ee^{-\pi|\ee^{2t}z|^2}\,\dd \Re z \, \dd \Im z
	\\ &= \ee^{t}\int \ee^{-\pi(\ee^{2t}-1)|z|^2}|\mathfrak{B}\comb_r(z)|^2 \ee^{-\pi|z|^2}\,\dd \Re z \, \dd \Im z.
\end{align*}
Because $|\mathfrak{B}\comb_r(z)|^2 \ee^{-\pi|z|^2} \in L^\infty(\Bbb{C})$ (Proposition \ref{prop_Bargmann_mass_periodic}) and because 
\[
	\ee^t\int \ee^{-\pi(\ee^{2t}-1)|z|^2}\,\dd \Re z \,\dd \Im z = \frac{\ee^{t}}{\ee^{2t}-1} = \frac{1}{2\sinh t},
\]
we see that there exists some constant $C > 0$ such that
\[
	\|\ee^{-tQ_0}\comb_r\|^2 \leq \frac{C}{\sinh t}, \quad \forall t > 0.
\]

Therefore
\[
	\|Q_0^{-p}\comb_1\| \leq \left\|\int_0^{\infty} t^{p-1}\ee^{-tQ_0}\comb_r\,\dd t\right\| \leq \sqrt{C}\int_0^{\infty} \frac{t^{p-1}}{\sqrt{\sinh t}}\,\dd t.
\]
Since $\sinh t \sim t$ when $t$ is small and $\sinh t \sim \ee^t$ when $t$ is large, the last integral is finite if and only if $p > \frac{1}{2}$. This proves the proposition.
\end{proof}

\begin{corollary}
Let $f \in \mathscr{S}'(\Bbb{R})$ and define, where possible,
\[
	g(\alpha) = \langle \mathcal{F}^\alpha \comb_r, f\rangle.
\]
If there exists some $\eps > 0$ such that $Q_0^{\frac{1}{2}+\eps} f \in L^2(\Bbb{R})$, then $g$ is continuous on $\Bbb{R}$. If there exists some $\eps > 0$ such that $Q_0^{\frac{3}{2}+\eps} f \in L^2(\Bbb{R})$, then $g$ is $C^1$ on $\Bbb{R}$. If there exists some $t > 0$ such that $\ee^{tQ_0}f\in L^2(\Bbb{R})$, then $g$ is analytic on $\{\frac{\pi}{2}\Im \alpha < t\}$. 
\end{corollary}

\subsection{The limit $\alpha \to 0$, Fresnel integrals, and Gauss sums}\label{ss_spirals}

Numerical computation quickly reveals that the antiderivative
\[
	\Pi_\alpha(X) = \int_0^X \mathcal{F}^\alpha \comb_1(x)\,\dd x
\]
defined in Definition \ref{def_antiderivative} does not, as $\alpha \to 0$ within $\{\alpha \::\: \tan \frac{\pi\alpha}{2} \in \Bbb{Q}\}$, tend locally uniformly to
\[
	\Pi_0(X) = \frac{1}{2}(\lfloor X \rfloor + \lceil X\rceil) = \begin{cases}X, & X \in \Bbb{Z}
	\\ \lfloor X \rfloor + \frac{1}{2}, & X \in \Bbb{R}\backslash \Bbb{Z}. \end{cases}
\]

Some further analysis reveals that the rescaled $\Pi_\alpha(\alpha^{1/2}X)$ resembles the Fresnel integral
\begin{equation}\label{eq_def_Fresnel}
	\ee^{-\frac{\pi\ii}{4}}S(X) = \ee^{-\frac{\pi\ii}{4}}\int_0^X \ee^{\pi\ii x^2}\,\dd x.
\end{equation}
Morever, when $\tan\frac{\pi\alpha}{2} = \frac{b}{a}$ with $a \gg b$ (relatively prime integers) and $b > 0$, the spiral traced by the values of $S(X)$ in the complex plane is made up of small repeating blocks which depend only on $b$ and on $a~(\opnm{mod} b)$. (See Figure \ref{fig_Euler_spiral}.) These blocks correspond to certain Gauss sums, and analyzing these motifs allows us to deduce an expression for the eighth root of unity $\mu(\Bff{M}; (q,p))$ in terms of these Gauss sums. We note that the case $b < 0$ is the complex conjugate of the case $b > 0$; see Remark \ref{rem_conj}.

Our goal is to establish this expression in the following form.

\begin{figure}
\centering
\includegraphics[width = \textwidth]{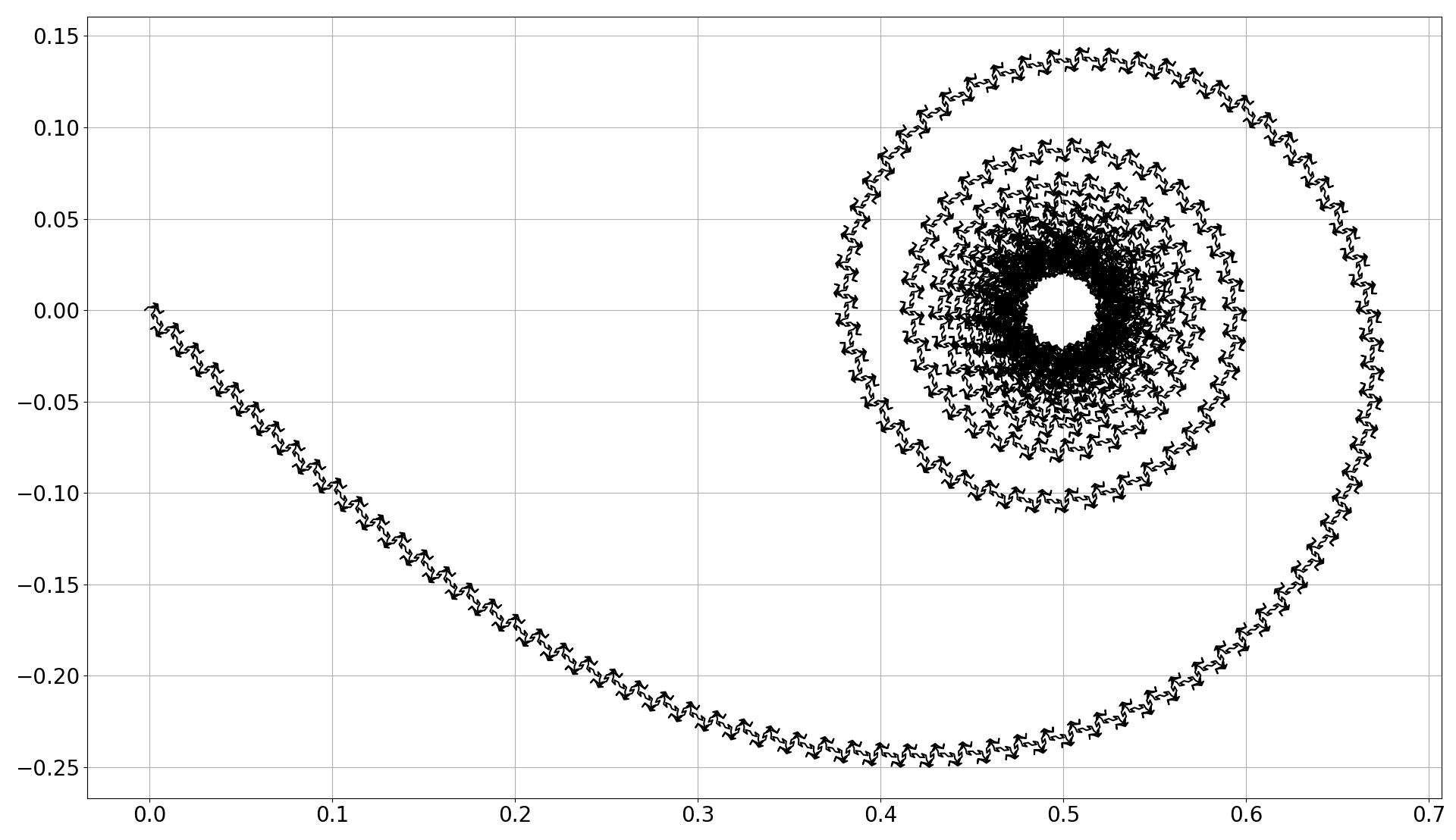}
\caption{Graph of values of $\Pi_{\alpha}(X)$ in the complex plane when $\tan \frac{\pi\alpha}{2} = \frac{38}{160001}$ and $X \in [0, 0.1]$.\label{fig_Euler_spiral}}
\end{figure}

\begin{theorem}\label{thm_Gauss_Sum}
Let $\Bff{M} = \begin{pmatrix} a & b \\ c & d \end{pmatrix} \in SL(2, \Bbb{Z})$ with $b > 0$, let $(q, p) \equiv (ab, cd)~(\opnm{mod} 2)$, and let $\mu(\Bff{M}; (q,p))$ be as in Theorem \ref{thm_id_metapl}. Then, when
\[
	E(n) = -\frac{d}{2b}(n-\frac{q}{2})^2 - \frac{1}{2}pn + \frac{1}{8}qp,
\]
we have the formula
\[
	\mu(\Bff{M}; (q,p)) = \frac{\ee^{-\frac{\pi\ii}{4}}}{\sqrt{b}}\sum_{n=0}^{b-1}\ee^{-2\pi\ii E(n)}.
\]
\end{theorem}

\subsubsection{Decomposition of $\mathcal{F}^{\alpha_j}\comb_1$}

We study $\mathcal{F}^{\alpha}\comb_1$, when $\tan\frac{\pi\alpha}{2} = \frac{b}{a}$ with $a,b\in\Bbb{Z}$ are relatively prime, by studying $\mathcal{F}^{\alpha_j}\comb_1$ when $\tan\frac{\pi\alpha}{2} = \frac{b}{a+jb}$ and $j\to \infty$ in $\Bbb{N}$.

We therefore fix $c, d\in\Bbb{Z}$ with $ad-bc = 1$ and $(q, p) \equiv (ab, cd)~(\opnm{mod 2})$. We write $\mu = \mu\left(\begin{pmatrix} a & b \\ c & d\end{pmatrix}; (q, p)\right)$ as in Theorem \ref{thm_id_metapl}. Following Proposition \ref{prop_mu_transformations}, we let
\[
	\Bff{M}_j = \begin{pmatrix} a + jb & b \\ c + jd & d\end{pmatrix}, \quad (q_j, p_j) = (q-jb, p-jd)
\]
and we have
\[
	\mu_j := \mu(\Bff{M}_j; (q_j, p_j)) = \ee^{\frac{\pi\ii}{4}j(pb-dq)}\mu.
\]

The statement of Theorem \ref{thm_identify}, rephrased as a sum in Remark \ref{rem_sum_expression}, depends on 
\[
	t_j = (a+jb)(c+jd) + bd = j^2 bd + j(ad + bc) + t_0
\]
and $s_j > 0$ defined by
\[
	s_j^2 = (a+jb)^2 + b^2 = j^2b^2 + 2jab + s_0^2.
\]

The sum in Remark \ref{rem_sum_expression}, using Proposition \ref{prop_mu_transformations}, is
\begin{align*}
	\mathcal{F}^{\alpha_j} \comb_1(x) &= \mu_j s_j^{-1/2}\sum_{k\in\Bbb{Z}-\{\frac{q_j}{2}\}} \ee^{-\pi\ii \frac{t_j}{s_j^2}k^2 - \pi\ii p_j k - \frac{\pi\ii}{4} p_jq_j}\delta(x-\frac{k}{s_j})
	\\ &= \mu s_j^{-1/2}\sum_{k\in\Bbb{Z} - \{\frac{q_j}{2}\}} \ee^{-\pi\ii \frac{t_j}{s_j^2} k^2 - \pi\ii p_j k - \frac{\pi\ii}{4}q_j p_j + \frac{\pi\ii}{4}j(pb-dq)}\delta(x-\frac{k}{s_j})
	\\ &= \mu s_j^{-1/2}\sum_{k\in\Bbb{Z} - \{\frac{q_j}{2}\}} \ee^{\pi\ii(\Delta_0 k)^2 + 2\pi\ii E_0(k, j)}\delta(x-\frac{k}{s_j})
\end{align*}
when $\Delta_0 > 0$ is defined by removing the principal part, $-d/b$, of $-t_j/s_j^2$,
\[
	\Delta_0^2 = -\frac{t_j}{s_j^2} + \frac{d}{b},
\]
and $2\pi\ii E_0(k,j)$ is what remains in the exponent,
\begin{equation}\label{eq_def_E0}
	E_0(k,j) = -\frac{d}{2b}k^2 - \frac{1}{2}p_j k - \frac{1}{8}p_jq_j + \frac{1}{8}j(pb - dq).
\end{equation}
We compute that
\[
	\Delta_0^2 = \frac{ds_j^2 - bt_j}{bs_j^2} = \frac{2abdj + ds_0^2 - ((abd + b^2c)j + bt_0)}{bs_j^2} = \frac{bj + ds_0^2 - bt_0}{bs_j^2}.
\]

We will show in Lemma \ref{lem_E0_props} below that, when $k = n - \frac{q_j}{2}$ is a generic element of $\Bbb{Z} - \{\frac{q_j}{2}\}$, the function $E_0(n - \frac{q_j}{2},j)$ is independent of $j$ in $\Bbb{Z}$, is $b$-periodic in $n$ modulo $1$, and can be expressed as
\[
	E(n) = E_0(n-\frac{q_j}{2}, j) = -\frac{d}{2b}(n-\frac{q}{2})^2 - \frac{1}{2}pn + \frac{1}{8}qp.
\]
Therefore, $\ee^{2\pi\ii E(n+bm)} = \ee^{2\pi\ii E(n)}$ for all $n, m \in \Bbb{Z}$. In particular, we can express $k \in \Bbb{Z} - \{\frac{q_j}{2}\}$ in a unique way as
\[
	k = n + bm - \frac{q_j}{2}, \quad m \in \Bbb{Z}, n \in \{0, 1, \dots, b-1\},
\]
and when $m_0 \in [0, 1)$ is defined by $k \equiv b(m+m_0)~(\opnm{mod} b)$ or
\[
	m_0 = m_0(n, b, q, j) \equiv \frac{1}{b}(n - \frac{q_j}{2}) \quad (\opnm{mod} 1),
\]
we obtain
\[
	\mathcal{F}^{\alpha_j} \comb_1(x) = \mu s_j^{-1/2}\sum_{n = 0}^{b-1}\ee^{2\pi \ii E(n)}\sum_{m \in \Bbb{Z}} \ee^{\pi \ii (\Delta_0 b(m+m_0))^2}\delta(x - \frac{b}{s_j}(m+m_0)).
\]

It is convenient to introduce
\begin{equation}\label{eq_def_rD1}
	\begin{aligned}
	r &= \frac{b}{s_j} = \frac{b}{\sqrt{(a+jb)^2 + b^2}},
	\\ \Delta &= b\Delta_0 = \sqrt{\frac{b^2j + bds_0^2 - b^2 t_0}{b^2j^2 + 2abj + s_0^2}},
	\end{aligned}
\end{equation}
which are functions of $a, b, c, d,$ and $j$. We record the asymptotics that
\begin{align*}
	r &= \frac{1}{j}(1+\BigO(j^{-1})),
	\\ \Delta^2 & = \frac{1}{j}(1+\BigO(j^{-1})).
\end{align*}
We also introduce notation for the inner sum above, which as we will describe is essentially a discrete approximation to a Fresnel integral:
\begin{equation}\label{eq_def_Fres_Riemm}
	f(x; r, \Delta, m_0) = \sqrt{r}\sum_{m \in \Bbb{Z}}\ee^{\pi\ii(\Delta(m+m_0))^2}\delta(x-r(m+m_0)).
\end{equation}
We remark that we may write $b$ using shifts and metaplectic transformations as in \eqref{eq_def_shift} and Section \ref{ss_metapl}:
\begin{equation}\label{eq_def_Fres_Riemm_metapl}
	f(x; r, \Delta, m_0) = \mathcal{V}_r \mathcal{W}_{\Delta^2}\shift_{(-m_0, 0)}\comb_1(x).
\end{equation}

\begin{proposition}\label{prop_asmall_decomp}
Let $a, b, c, d \in \Bbb{Z}$ be such that $b > 0$ and $ad-bc = 1$. Also let $(q,p) \equiv (ab, cd)~(\opnm{mod} 2)$ and let $\mu = \mu\left(\begin{pmatrix} a & b \\ c & d\end{pmatrix}; (q,p)\right)$ as in Theorem \ref{thm_id_metapl}. For $j \in \Bbb{N}$, let $\alpha_j \in (0, \pi)$ be such that $\tan \frac{\pi\alpha_j}{2} = \frac{b}{a+jb}$. Let 
\[
	E(n) = -\frac{d}{2b}(n-\frac{q}{2})^2 - \frac{1}{2}pn + \frac{1}{8}qp,
\]
and let 
\[
	f(x; r, \Delta, m_0) = \sqrt{r}\sum_{m \in \Bbb{Z}}\ee^{\pi\ii(\Delta(m+m_0))^2}\delta(x-r(m+m_0)).
\]

Then, with $r = j^{-1}(1+\BigO(j^{-1}))$ and $\Delta = j^{-1/2}(1+\BigO(j^{-1}))$ defined as in \eqref{eq_def_rD1} and $m_0 \equiv \frac{1}{b}(n-\frac{1}{2}(q-jb))~(\opnm{mod} 1)$,
\[
	\mathcal{F}^{\alpha_j}\comb_1(x) = \mu\frac{1}{\sqrt{b}}\sum_{n=0}^{b-1}\ee^{2\pi\ii E(n)} f(x; r, \Delta, m_0).
\]
\end{proposition}

\begin{lemma}\label{lem_E0_props}
The function
\[
	(n, j) \mapsto E_0(n - \frac{q-jb}{2}, j)
\]
with $E_0(k,j)$ defined in \eqref{eq_def_E0} is independent of $j$ and $b$-periodic as a function from $\Bbb{Z}$ to $\Bbb{R}/\Bbb{Z}$; it is given by the formula
\[
	E_0(n - \frac{q-jb}{2}, j) = -\frac{d}{2b}(n-\frac{q}{2})^2 - \frac{1}{2}pn + \frac{1}{8}qp.
\]
\end{lemma}

\begin{proof}
We expand $E_0$ to obtain
\begin{align*}
	E_0(n - \frac{q_j}{2}, j) &= - \frac{d}{2b}(n-\frac{q_j}{2})^2 - \frac{1}{2} p_j(n-\frac{q_j}{2}) - \frac{1}{8}q_jp_j + \frac{1}{8}j(pb-dq)
	\\ &= -\frac{d}{2b}n^2 - \frac{1}{2}(p_j - \frac{d}{b}q_j)n + \frac{1}{8}q_j(p_j - \frac{d}{b}q_j) + \frac{1}{8}j(pb-dq).
\end{align*}
Since $(q_j, p_j) = (q-jb, p-jd)$,
\[
	p_j - \frac{d}{b}q_j = p - \frac{d}{b}q
\]
is $j$-invariant. (One may also recognize the form of a shift, in this case by $(\frac{q}{2}, -\frac{p}{2})$, being projected onto $\{(0, \zeta)\}_{\zeta \in \Bbb{C}}$ along the Lagrangian associated with $\ee^{-\pi\ii\frac{d}{b}x^2}$, as in Lemma \ref{lem_equiv_shifts}.) Therefore
\begin{align*}
	E_0(n - \frac{q_j}{2}, j) &= -\frac{d}{2b}n^2 - \frac{1}{2}(p - \frac{d}{b}q)n + \frac{1}{8}(q-jb)(p - \frac{d}{b}q) + \frac{1}{8}j(pb-dq)
	\\ &= -\frac{d}{2b}n^2 - \frac{1}{2}(p - \frac{d}{b}q)n + \frac{1}{8}q(p-\frac{d}{b}q).
\end{align*}
Recalling the Gaussian $g_{\Bff{v}, \tau}$ defined in \eqref{eq_def_gaussian}, we obtain, through Lemma \ref{lem_equiv_shifts} or through direct computation,
\begin{align*}
	\ee^{2\pi\ii E_0(n-\frac{q_j}{2}, j)} &= \ee^{-\pi\ii\frac{d}{b}n^2 - \pi\ii(p - \frac{d}{b}q)n + \frac{\pi\ii}{4}q(p-\frac{d}{b}q)}
	\\ &= \ee^{\frac{\pi\ii}{4}q(p-\frac{d}{b}q)}g_{(0, -\frac{1}{2}(p-\frac{d}{b}q)), -\frac{d}{b}}(n)
	\\ &= g_{(\frac{q}{2}, -\frac{p}{2}), -\frac{d}{b}}(n)
	\\ &= \ee^{-\pi\ii\frac{d}{b}(n-\frac{q}{2})^2 - \pi\ii p n + \frac{\pi\ii}{4}qp}.
\end{align*}

As for periodicity, we compute
\begin{align*}
	E(n+b)-E(n) &= -\frac{d}{2b}\left( (n+b-\frac{q}{2})^2 - (n-\frac{q}{2})^2\right) - \frac{1}{2}pb
	\\ &= -dn + \frac{1}{2}(dq-bd-pb).
\end{align*}
To show that $E(n+b)-E(n) \equiv 0~(\opnm{mod} 1)$ for $n\in\Bbb{Z}$, it suffices to show that $dq - db - pb$ is even. We recall that at least one of $q$ and $p$ is even, since $ad - bc = 1$ and $(q, p) \equiv (ab, cd)~(\opnm{mod} 2)$. If both $q$ and $p$ are even, then $bd$ cannot be odd because then both $a$ and $c$ are even, contradicting $ad - bc = 1$. If $q$ is odd then $p$ is even and $b$ is odd, so $d(q-b)$ is even. The corresponding reasoning --- if $p$ is odd then $q$ is even and $d$ is odd, so $b(d+p)$ is even --- completes the study of different cases. We conclude that $E(n)$ is $b$-periodic when viewed as a function from $\Bbb{Z}$ to $\Bbb{R}/\Bbb{Z}$.
\end{proof}

\subsubsection{Proof of Theorem \ref{thm_Gauss_Sum}}

We can analyze the ``Riemann sums'' $f(x; r, \Delta, m_0)$ either using standard results of numerical integration (to justify images like the one in Figure \ref{fig_Euler_spiral}) or by testing against a Gaussian. The latter approach turns out to be significantly simpler, which is to be expected because it fits naturally with the metaplectic representation. In both approaches, however, we obtain the same scale: the Fresnel integrals live on a scale $X \sim j^{-1/2}$, so to capture enough information we need to take $X \gg j^{-1/2}$. On the other hand, we surely cannot expect to take $X \gg 1$ or even $X \geq 1$, because the antiderivative of $\mathcal{F}^{\alpha_j}\comb_1$ weakly approaches $\Pi_0(X)$, the antiderivative of $\comb_1$, and $\Pi_0(X)$ jumps from $1/2$ on the interval $(0,1)$ to $3/2$ on the interval $(1, 2)$.

In terms of testing against a Gaussian $g_{(0,0),\tau}(x) = \ee^{\pi\ii \tau x^2}$, this forces us to choose $1 \ll \Im \tau$ and $|\tau| \ll j$, which we express in terms of $r = \Delta^2(1+o(1)) \to 0^+$ as follows.

\begin{proposition}\label{prop_fres_riemm_on_Gaussian}
Let $r, \Delta \in (0, \infty)$ be subject to the restriction $r = \Delta^2(1+o(1))$ as $\Delta \to 0^+$. Let $f$ be as in \eqref{eq_def_Fres_Riemm} and recall the Gaussian $g_{(0,0), \tau}(x) = \ee^{\pi\ii \tau x^2}$ from \eqref{eq_def_gaussian}. Then, for $\tau = \tau_1 + \ii \tau_2 \in \Bbb{C}$ obeying $|\tau| \ll \Delta^{-2}$ and $\tau_2 \to +\infty$ as $\Delta \to 0^+$, independently of $m_0 \in \Bbb{R}$,
\[
	\lim_{\Delta \to 0^+} \langle g_{(0,0), \tau}(x), f(x; r, \Delta, m_0)\rangle = \ee^{-\frac{\pi\ii}{4}}.
\]
\end{proposition}

\begin{proof}
We use the metaplectic presentation in \eqref{eq_def_Fres_Riemm_metapl}, as well as unitarity of metaplectic transformations (Remark \ref{rem_unitary}) and the Poisson summation formula. Letting 
\[
	\zeta = r^2 \tau - \Delta^2,
\]
we compute
\begin{align*}
	\langle g_{(0,0), \tau}(x), f(x; r, \Delta, m_0)\rangle &= \langle g_{(0,0), \tau}, \mathcal{V}_{1/r} \mathcal{W}_{\Delta^2} \shift_{(-m_0, 0)} \comb_1\rangle
	\\ &= \langle \shift_{(m_0, 0)} \mathcal{W}_{-\Delta^2}\mathcal{V}_r g_{(0,0), \tau}, \comb_1\rangle
	\\ &= \langle \sqrt{r}\shift_{(m_0, 0)} g_{(0,0), \zeta}, \comb_1\rangle
	\\ &= \sqrt{r}\langle \mathcal{F}\shift_{(m_0, 0)} g_{(0,0), \zeta}, \mathcal{F} \comb_1\rangle
	\\ &= \sqrt{r}\langle \shift_{(0, -m_0)}\mathcal{F}g_{(0,0), \zeta}, \ee^{-\frac{\pi\ii}{4}}\comb_1\rangle.
\end{align*}

The Fourier transform of $g_{(0,0), \zeta}$ for $\Im \zeta > 0$ is $\mathcal{F}g_{(0,0), \zeta} = \zeta^{-1/2}g_{(0,0), -1/\zeta}$; since we choose $\Re \zeta^{-1/2} > 0$, we have $\arg \zeta^{-1/2} \in (-\pi/2, 0)$. We compute
\[
	-\frac{1}{\zeta} = -\frac{1}{r^2(\tau_1 + \ii \tau_2) - \Delta^2} = \frac{1/r}{\Delta^2/r - r\tau_1 - \ii r \tau_2} = \frac{(\Delta/r)^2 - \tau_1 + \ii \tau_2}{(\Delta^2/r - r\tau_1)^2 + (r\tau_2)^2}.
\]
Under our hypotheses, when $\zeta = r^2 \tau - \Delta^2$, then $\Im(-1/\zeta) \to \infty$. Furthermore, $\zeta/r = r\tau - \Delta^2/r \to -1$ as $\Delta \to 0^+$, so because $\arg \zeta^{-1/2} \in (-\pi/2, 0)$, $\sqrt{r}\zeta^{-1/2} \to -\ii$ as $\Delta \to 0^+$. Therefore
\begin{align*}
	\langle g_{(0, 0), \tau}(x), f(x; r, \Delta, m_0)\rangle &= \ee^{\frac{\pi\ii}{4}}\sqrt{r}\zeta^{-1/2}\langle \shift_{(0, -m_0)}g_{(0,0), -1/\zeta}, \comb_1\rangle
	\\ &= \ee^{\frac{\pi\ii}{4}}\sqrt{r}\zeta^{-1/2}\sum_{m \in\Bbb{Z}} \ee^{-2\pi\ii m_0 m + \pi\ii (-\frac{1}{\zeta})m^2}
	\\ & \to \ee^{\frac{\pi\ii}{4}}(-\ii) = \ee^{-\frac{\pi\ii}{4}},\quad \textnormal{as }\Delta \to 0^+,
\end{align*}
because, as $\Im(-1/\zeta) \to \infty$, the terms in the sum where $m \neq 0$ are superexponentially small.
\end{proof}

This allows us to prove Theorem \ref{thm_Gauss_Sum}.

\begin{proof}[Proof of Theorem \ref{thm_Gauss_Sum}] Let $r$ and $\Delta$ be as in \eqref{eq_def_rD1} and let $\tau = \tau(j)$ be such that $\Im \tau \to \infty$ yet $|\tau| \ll \Delta^{-2} = j(1+\BigO(j^{-1}))$. (If one wishes to be concrete, let $\tau = \ii \sqrt{j}$, so that $g_{(0,0), \tau}$ ``lives on'' $|x| \lesssim j^{-1/4}$.) By Propositions \ref{prop_asmall_decomp} and \ref{prop_fres_riemm_on_Gaussian},
\begin{align*}
	\langle g_{(0,0), \tau}, \mathcal{F}^{\alpha_j}\comb_1\rangle &= \bar{\mu}\frac{1}{\sqrt{b}}\sum_{n=0}^{b-1}\ee^{-2\pi\ii E(n)}\langle g_{(0,0), \tau}, f(x; r, \Delta, m_0)
	\\ &\to \bar{\mu}\frac{\ee^{-\frac{\pi\ii}{4}}}{\sqrt{b}}\sum_{n=0}^{b-1}\ee^{-2\pi\ii E(n)}, \quad \textnormal{as }j \to \infty.
\end{align*}
On the other hand, since 
\[
	\ee^{\frac{\pi\ii}{2}\alpha_j} = \frac{1}{s_j}(a+jb + \ii b),
\]
one obtains using Lemma \ref{lem_M_on_gaussian} that
\[
	\mathcal{F}^{-\alpha_j}g_{(0,0), \tau} = \left(\frac{a+jb}{s_j} - \frac{b \tau}{s_j}\right)^{-1/2}g_{(0,0), \tau'}, \quad \tau' = \frac{b+(a+jb)\tau}{a+jb + b\tau}.
\]
Since $|\tau| \ll j$ and $s_j \sim jb$, $\frac{a+jb}{s_j} - \frac{b\tau}{s_j} = 1+o(1)$ and $\tau' = \tau(1+o(1))$ as $j \to \infty$. Having also assumed that $\Im \tau \to \infty$,
\begin{align*}
	\langle g_{(0,0), \tau}, \mathcal{F}^{\alpha_j} \comb_1\rangle &= \langle \mathcal{F}^{-\alpha_j}g_{(0,0),\tau}, \comb_1\rangle
	\\ &= \left(\frac{a+jb}{s_j} - \frac{b \tau}{s_j}\right)^{-1/2}\sum_{k \in \Bbb{Z}}\ee^{\pi\ii \tau' k^2}
	\\ & \to 1, \quad \textnormal{as }j \to \infty.
\end{align*}

This shows that
\[
	\bar{\mu}\frac{\ee^{-\frac{\pi\ii}{4}}}{\sqrt{b}}\sum_{n=0}^{b-1}\ee^{-2\pi\ii E(n)} = 1.
\]
Since $\bar{\mu} = \mu^{-1}$, this completes the proof of the theorem.
\end{proof}

\subsubsection{Riemann sums for the Fresnel integral}

The antiderivative of $f$ defined in \ref{eq_def_Fres_Riemm} is
\begin{equation}\label{eq_Fresnel_Riemann_antiderivative}
	F(X; \Delta, r, m_0) = \sqrt{r}\sum_{-m_0 \leq m \leq \frac{1}{r}X-m_0} \ee^{\pi\ii(\Delta(m+m_0))^2}.
\end{equation}
We may take $m_0 \in [0,1)$ without loss of generality. (Strictly speaking, to match the antiderivative $\Pi_{\alpha_j}(X)$ defined in Definition \ref{def_antiderivative} we should take the mean of the sum over $-m_0 \leq m \leq \frac{1}{r}X-m_0$ and the sum over $-m_0 < m < \frac{1}{r}X - m_0$, but the error is bounded by $\Delta$.)

Fix an order $N \in \Bbb{N}$ as well as $\Delta > 0$, $m_0 \in [0, 1)$, and a number of intervals $M \in \Bbb{N}$. Let $I_m = [\Delta m, \Delta (m+1)]$ and $I = [-\Delta N, \Delta(M+N)]$. A straightforward consequence of the Newton-Cotes formulas is that there exists a $C = C(N)$ such that, for any $g:I \to \Bbb{R}$ which is $N$ times continuously differentiable,
\begin{multline*}
	\left|\sum_{m=0}^{M-1} \Delta g(\Delta(m+m_0)) - \int_0^{M\Delta}g(y)\,\dd y\right| 
	\\ \leq C_N\left(\Delta \|g\|_{L^\infty(I)} + \Delta^{N+1}\sum_{m=0}^{M-1}\|g^{(N)}\|_{L^\infty(I_m)}\right).
\end{multline*}

We apply this to $g(y) = \ee^{\pi\ii y^2}$, where $\|g\|_{L^\infty} = 1$ and the derivative is bounded by $\|g^{(N)}\|_{L^\infty(I_m)} \leq C_N(1+(\Delta m)^N)$. The rounding error where $\frac{1}{r}X - m_0$ is absorbed in $C_N\Delta$, and we obtain for $\Delta \in (0,1]$, $m_0\in [0, 1)$
\begin{multline*}
	\left|\frac{\Delta}{\sqrt{r}} F(X; \Delta, r, m_0) - \int_0^{\frac{\Delta}{r}X}\ee^{\pi\ii y^2}\,\dd y\right|
	\leq C_N\left(\Delta + \Delta^{N+1}\sum_{m=0}^{\frac{1}{r}X-m_0} (1 + (\Delta m)^N)\right)
	\\  \leq C_N\left(\Delta + \Delta^{N+1}\frac{X}{r} + \Delta^{2N+1}\left(\frac{X}{r}\right)^{N+1}\right)
	\\ \leq C_N\left(\Delta + \Delta^{N+1}\frac{X}{r} +  \left(\frac{\Delta^2}{r}\right)^{N+1}\frac{1}{\Delta}X^{N+1}\right).
\end{multline*}

Our application will be to $\Delta = j^{-1/2}(1+\BigO(j^{-1}))$ and $r = j^{-1}(1+\BigO(j^{-1}))$. In this case, a simple consequence of the estimate above is the convergence of $\frac{\Delta}{\sqrt{r}}F(X; \Delta, r, m_0)$ to a Fresnel integral $\int_0^Y \ee^{\pi \ii y^2}\,\dd y$ when $Y = \frac{\Delta}{r}X$ and $X \in [0, \Delta^{\frac{1}{N}}]$, uniformly in the following sense.

\begin{proposition}\label{prop_Fres_Riemann}
Let $F(X; \Delta, r, m_0)$ be as in \eqref{eq_Fresnel_Riemann_antiderivative} and fix any $N \in \Bbb{N}^*$. If $r = \Delta^2(1+o(1))$ as $\Delta \to 0^+$ then
\[
	\lim_{\Delta \to 0^+}\sup_{X \in [0, \Delta^{1/N}], m_0 \in [0, 1)}\left|\frac{\Delta}{\sqrt{r}}F(X; \Delta, r, m_0) - \int_0^{\frac{\Delta}{r}X} \ee^{\pi\ii y^2}\,\dd y\right| = 0.
\]
\end{proposition}

\begin{proof}
Under these hypotheses,
\begin{multline*}
	\left|\frac{\Delta}{\sqrt{r}}F(X; \Delta, r, m_0) - \int_0^{\frac{\Delta}{r}X} \ee^{\pi\ii y^2}\,\dd y\right| 
	\\ \leq C_N\left(\Delta + \frac{\Delta^2}{r}\Delta^{N-1}\Delta^\frac{1}{N} + \left(\frac{\Delta^2}{r}\right)^{N+1}\Delta^{\frac{N+1}{N}-1}\right)
\end{multline*}
which tends to zero as $\Delta\to 0$ and $\Delta^2/r\to 1$.
\end{proof}

\begin{corollary}
Fix $a, b$ relatively prime integers with $b > 0$ and let $\alpha_j \in (0, \pi)$ be such that $\cot\frac{\pi\alpha_j}{2} = \frac{a+jb}{b}$. For any $N \in \Bbb{N}$, the antiderivative $\Pi_\alpha(X)$ (Definition \ref{def_antiderivative}) converges to the rescaled Fresnel integral 
\[
	\ee^{-\frac{\pi\ii}{4}}S(j^{1/2}X), \quad S(X) = \int_0^X \ee^{\pi\ii x^2}\,\dd x
\]
uniformly on $[0, j^{-\frac{1}{2N}}]$ in the sense that
\[
	\lim_{j\to\infty} \sup_{X \in [0, j^{-1/(2N)}]} \left|\Pi_{\alpha_j}(X) - \ee^{-\frac{\pi\ii}{4}}S(j^{1/2}X)\right| = 0.
\]
\end{corollary}

\begin{proof}
By Propositions \ref{prop_asmall_decomp} and \ref{prop_Fres_Riemann} (recall that $\Delta^2 = j^{-1}(1+\BigO(j^{-1})$ and $r = j^{-1}(1+\BigO(j^{-1})$ as well), as $j \to \infty$
\[
	\Pi_{\alpha_j}(X) = \mu\frac{1}{\sqrt{b}}\sum_{n=0}^{b-1} \ee^{2\pi\ii E(n)}\left((1+o(1))S(j^{1/2}X) + o(1)\right).
\]
Recall \cite[7.3.20]{Abramowitz_Stegun_1964} that $S(X)$ is bounded on $\Bbb{R}$ because $\lim_{X \to \infty}S(X) = \ee^{\frac{\pi\ii}{4}}$. By Theorem \ref{thm_Gauss_Sum},
\[
	\ee^{\frac{\pi\ii}{4}} = \frac{1}{\mu\sqrt{b}}\sum_{n=0}^{b-1}\ee^{-2\pi\ii E(n)}.
\]
Taking complex conjugates and recalling that $\mu = \bar{\mu}^{-1}$ proves the corollary.
\end{proof}

\begin{remark}
In particular, because
\[
	\Re S(\frac{\sqrt{3}}{2}) \approx 0.670, \quad \Im S(\frac{1}{2}) \approx -0.244,
\]
we do not have uniform convergence of $\Pi_\alpha(X)$ to $\Pi_0(X)$ in any neighborhood of zero. We do, however, have ``uniform convergence'' to $\Pi_0(X) = \frac{1}{2}$ in the obvious sense on intervals $[j^{\frac{1}{2} + \eps}, j^{-\eps}]$ for $\eps > 0$ fixed.
\end{remark}

\begin{remark}
Uniform estimates as $\alpha \to 0$ (instead of for $\alpha_j$ as $j \to \infty$), estimates for $X$ in larger sets, or statements of pointwise convergence would certainly be interesting, and seem numerically to be plausible. The author does not currently have any results in these directions.
\end{remark}

\section{Numerics around approximations to irrational cotangents}\label{s_irrational}

The list of things that the author does not know about $\{\mathcal{F}^\alpha \comb_1\}_{\alpha\in\Bbb{R}}$ is enormous. Here, we focus on the question of what happens when $\cot\frac{\pi\alpha}{2}$ is irrational. Indeed, this work fails to answer the second most natural question about $\mathcal{F}^\alpha \comb_1$: having identified $\mathcal{F}^{1/2}\comb_1$, what is $\mathcal{F}^{1/3}\comb_1$?

We have already seen (Figures \ref{fig_intro} and \ref{fig_intro_cty}, Section \ref{s_continuity}) that the antiderivative $\Pi_\alpha(X)$ appears to better show continuity in $\alpha$. It is straightforward to compute $\Pi_{\alpha_j}(X)$ when $\cot \frac{\pi\alpha_j}{2}$ is a (continued fraction) approximation to $\cot \frac{\pi}{6} = \sqrt{3}$. A very striking way to consider these approximations, presented in Figure \ref{fig_traces_sqrt3}, is via the values of the antiderivative $\Pi_{\alpha_j}$ in the complex plane. This flattens the depth $X$, so we use varying colors (oscillating between yellow and dark blue) to indicate this change. In particular, this also helps to see where the path re-crosses itself (frequently), giving information which would be lost with a monochrome curve. To try to clarify this approach, we include in Figure \ref{fig_traces_ReIm} the real and imaginary parts of $\Pi_\alpha(X)$ for an approximation $\cot \frac{\pi\alpha}{2} \approx \sqrt{3}$.

\begin{figure}
\centering
\includegraphics[width = .3\textwidth]{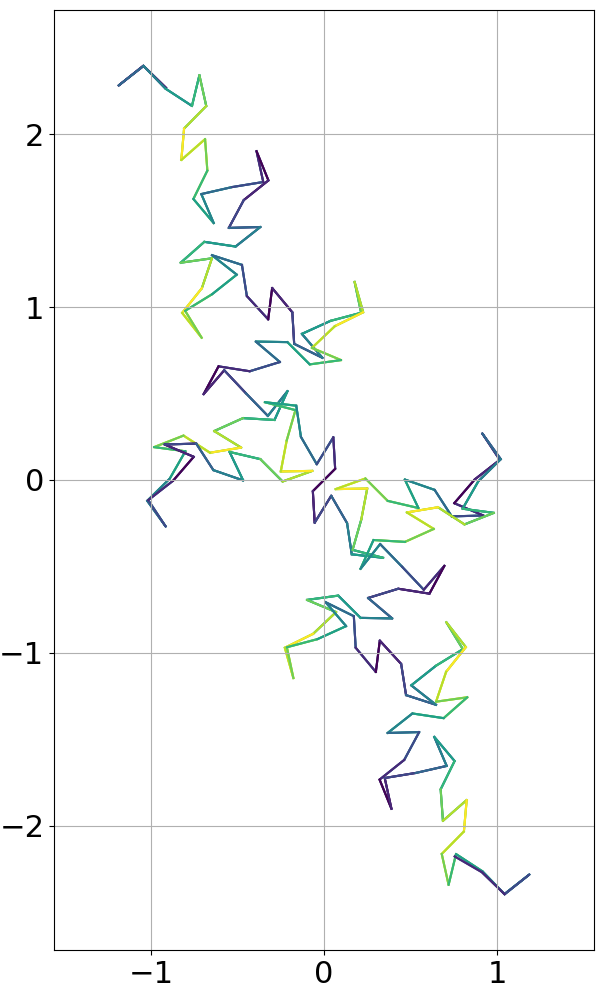}
\includegraphics[width = .3\textwidth]{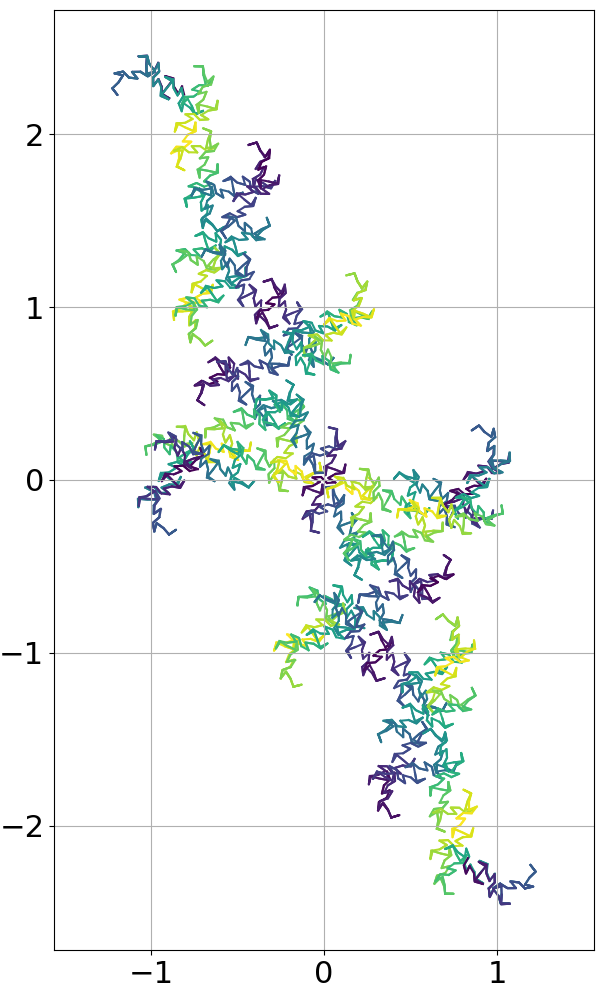}
\includegraphics[width = .3\textwidth]{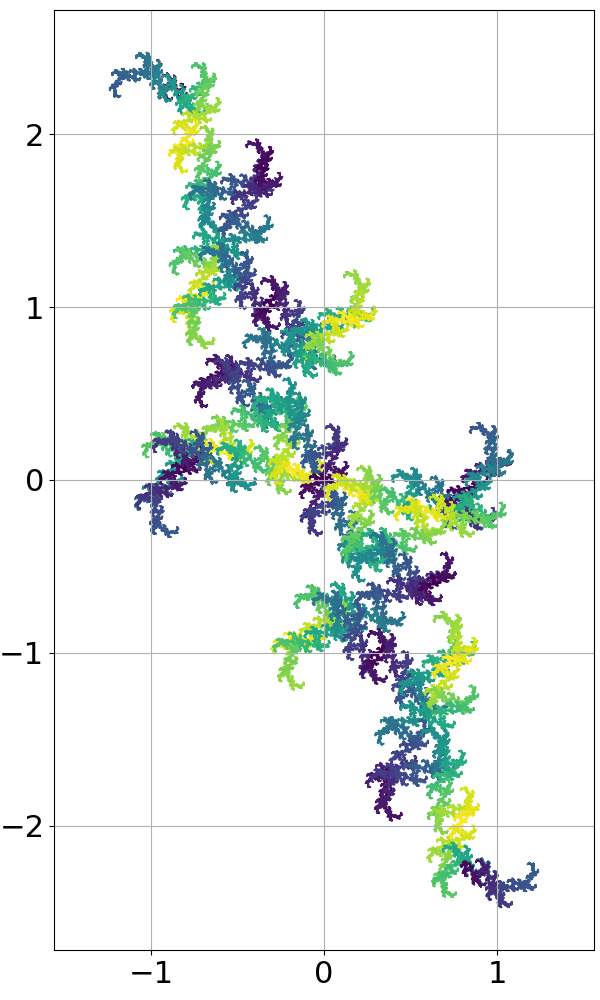}
\caption{Values of $\Pi_\alpha(X)$ in the complex plane for $X \in [-3, 3]$ and, from left to right, $\cot \frac{\pi\alpha}{2} = \frac{26}{15}, \frac{362}{209}, \frac{5042}{2911} \approx \sqrt{3}.$ Colors indicate varying $X$.\label{fig_traces_sqrt3}}
\end{figure}

\begin{figure}
\centering
\includegraphics[width = \textwidth]{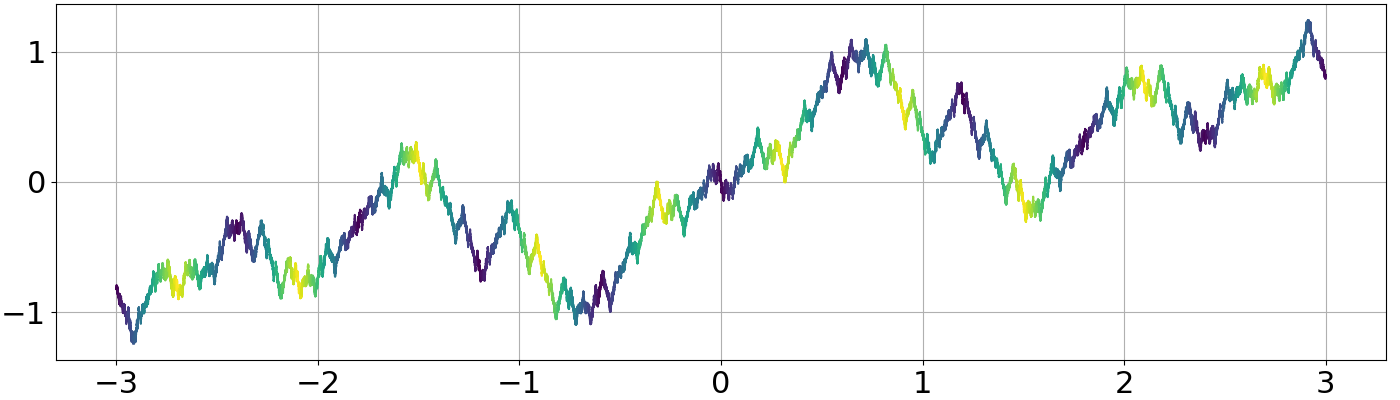}
\includegraphics[width = \textwidth]{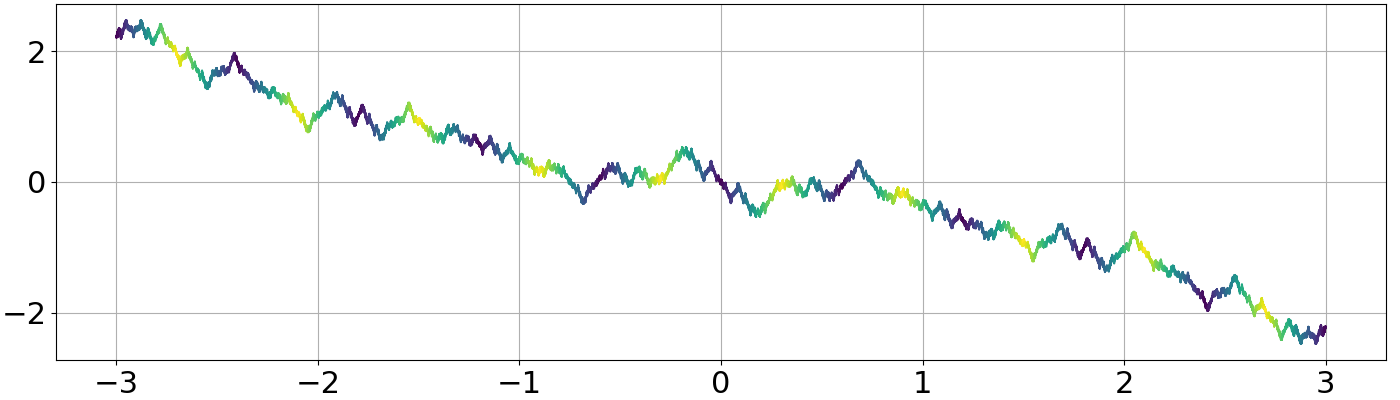}
\caption{Values of $\Re \Pi_\alpha(X)$ (top) and $\Im \Pi_\alpha(X)$ (bottom) as functions of $X$ for $\cot \frac{\pi\alpha}{2} = \frac{18817}{10864} \approx \sqrt{3}$, with colors matching those in Figure \ref{fig_traces_sqrt3}.\label{fig_traces_ReIm}}
\end{figure}

\begin{figure}
\centering
% Most important: fig0.set_size_inches(8, 10)
\includegraphics[width = .3\textwidth]{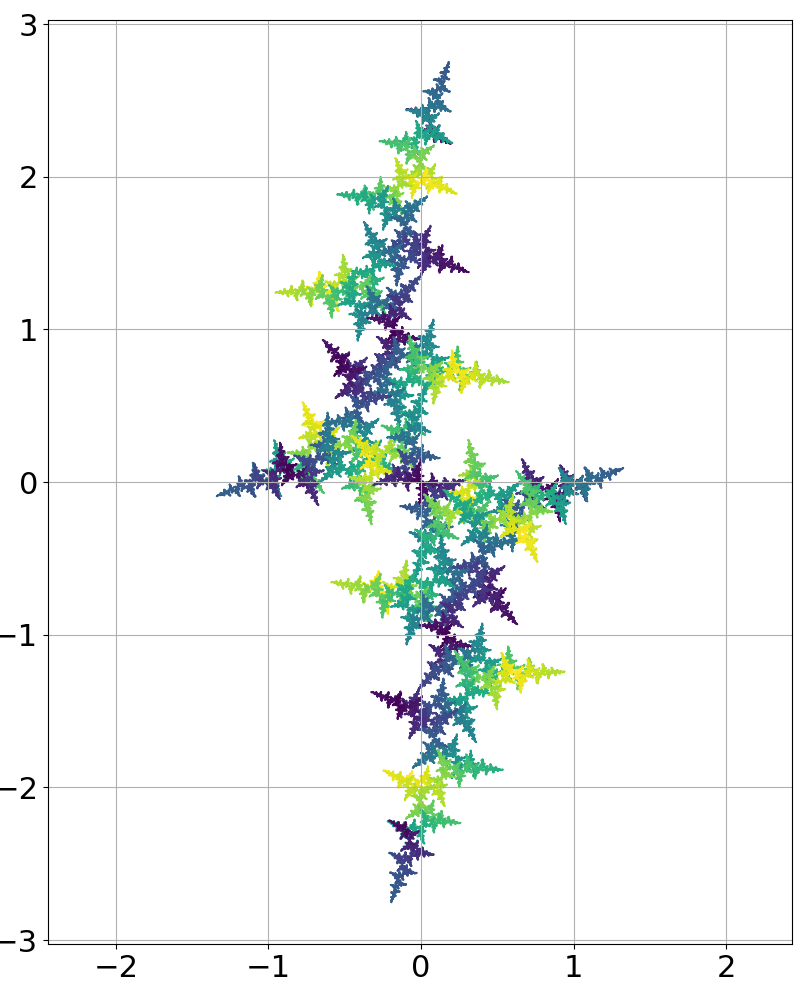}
\includegraphics[width = .3\textwidth]{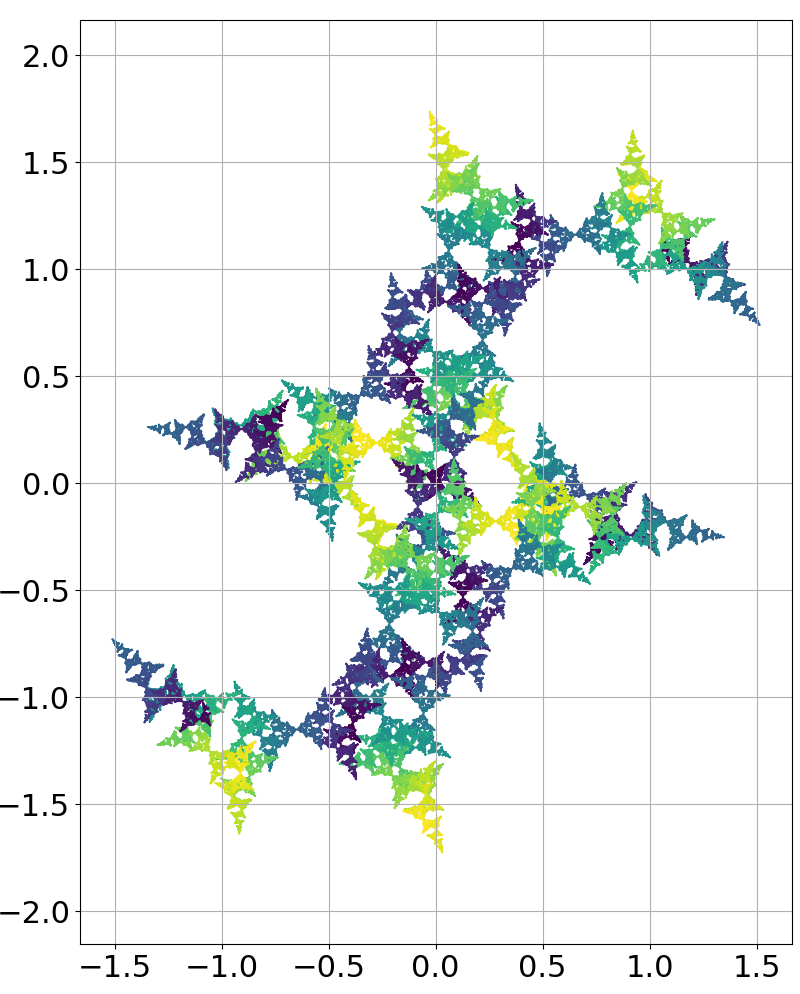}
\includegraphics[width = .3\textwidth]{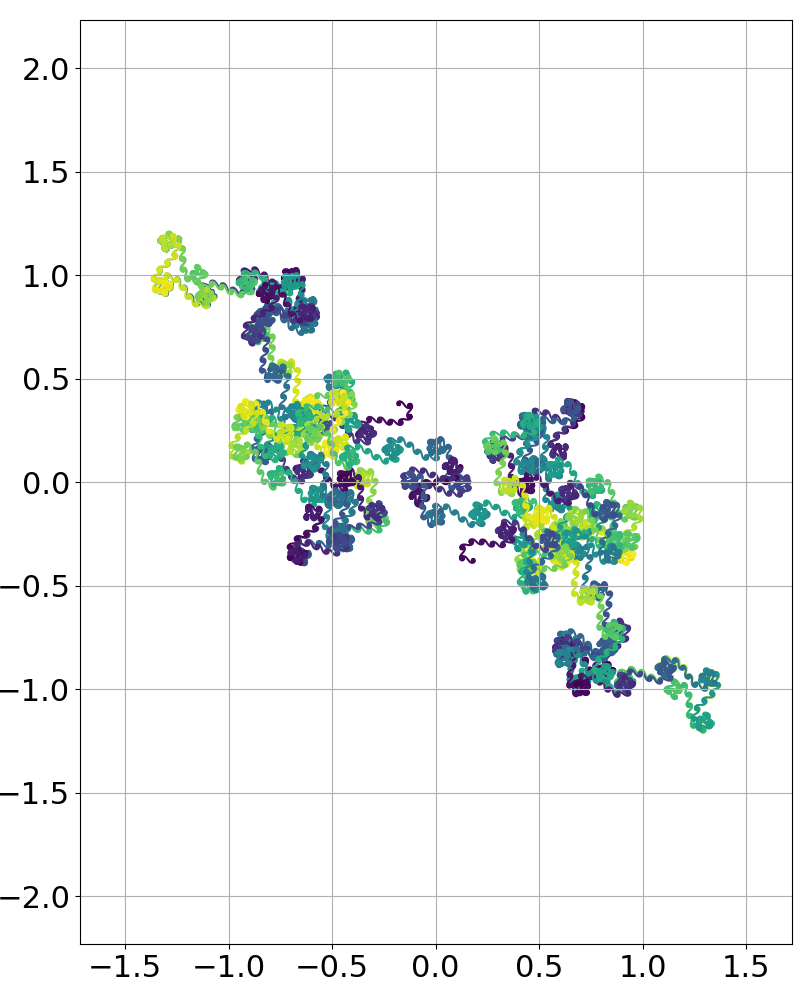}
\caption{Values of $\Pi_\alpha(X)$ in the complex plane for $X \in [-3, 3]$ and, from left to right, $\cot \frac{\pi\alpha}{2} = \frac{4181}{2584}\approx\frac{1+\sqrt{5}}{2}, \frac{8119}{5741} \approx \sqrt{2}$, and $\frac{103993}{33102} \approx \pi$. Colors indicate varying $X$.\label{fig_traces_multiple}}
\end{figure}

To the naked eye, it certainly seems that there is a limit $\Pi_{1/3}(X)$ which is a continuous but not differentiable function whose weak derivative gives $\mathcal{F}^{1/3}\comb_1(x)$.

In Figure \ref{fig_traces_multiple} we draw the values of $\Pi_\alpha(X)$ for similar approximations to cotangents $\frac{1}{2}(1+\sqrt{5})$, $\sqrt{2}$, and $\pi$. The square roots (with their repeating continued fractions) appear to have striking self-similar polygonal structure, and the graph of the values corresponding to $\pi$ seems to contain Euler spirals on several scales (which seem to correspond to large denominators in the continued fraction expansion).

Though we know that these functions are neither even nor odd around any point (Proposition \ref{prop_parity}), it certainly appears that there are many $x$ around which these approximations are nearly and locally even or odd.

One could also wonder whether the $\Pi_\alpha$ is Hausdorff continuous (if it is a function), whether the image has Hausdorff dimension larger than one as a subset of $\Bbb{C}$, whether these functions are nearly periodic and in what sense, whether parts of these curves are self-similar under scaling, what the behavior is as $x \to \infty$, and on and on.

At present, the author has very few answers. The questions seem quite natural, nontrivial, and, in the author's opinion, beautiful. The author hopes that, through further study or through exchange with other areas of research, more information will come to light.

\bibliographystyle{plain}
\bibliography{Brushes}

\begin{thebibliography}{10}

\bibitem{Abramowitz_Stegun_1964}
Milton Abramowitz and Irene~A. Stegun.
\newblock {\em Handbook of mathematical functions with formulas, graphs, and
  mathematical tables}, volume~55 of {\em National Bureau of Standards Applied
  Mathematics Series}.
\newblock For sale by the Superintendent of Documents, U.S. Government Printing
  Office, Washington, D.C., 1964.

\bibitem{Aleman_Viola_2014a}
Alexandru Aleman and Joe Viola.
\newblock Singular-value decomposition of solution operators to model evolution
  equations.
\newblock {\em Int. Math. Res. Not. IMRN}, 2014.

\bibitem{Aleman_Viola_2018}
Alexandru Aleman and Joe Viola.
\newblock On weak and strong solution operators for evolution equations coming
  from quadratic operators.
\newblock {\em J. Spectr. Theory}, 8(1):33--121, 2018.

\bibitem{Bargmann_1961}
Valentine Bargmann.
\newblock On a {H}ilbert space of analytic functions and an associated integral
  transform.
\newblock {\em Comm. Pure Appl. Math.}, 14:187--214, 1961.

\bibitem{Folland_1989}
Gerald~B. Folland.
\newblock {\em Harmonic analysis in phase space}, volume 122 of {\em Annals of
  Mathematics Studies}.
\newblock Princeton University Press, Princeton, NJ, 1989.

\bibitem{Hormander_1995}
Lars H\"{o}rmander.
\newblock Symplectic classification of quadratic forms, and general {M}ehler
  formulas.
\newblock {\em Math. Z.}, 219(3):413--449, 1995.

\bibitem{Hormander_ALPDO1}
Lars H{\"o}rmander.
\newblock {\em The analysis of linear partial differential operators. {I}}.
\newblock Classics in Mathematics. Springer-Verlag, Berlin, 2003.
\newblock Distribution theory and Fourier analysis, Reprint of the second
  (1990) edition [Springer, Berlin; MR1065993 (91m:35001a)].

\bibitem{Hormander_ALPDO3}
Lars H{\"o}rmander.
\newblock {\em The analysis of linear partial differential operators. {III}}.
\newblock Classics in Mathematics. Springer, Berlin, 2007.
\newblock Pseudo-differential operators, Reprint of the 1994 edition.

\bibitem{Leray_1981}
Jean Leray.
\newblock {\em Lagrangian analysis and quantum mechanics}.
\newblock MIT Press, Cambridge, Mass.-London, 1981.
\newblock A mathematical structure related to asymptotic expansions and the
  Maslov index, Translated from the French by Carolyn Schroeder.

\bibitem{Lion_Vergne_1980}
G\'{e}rard Lion and Mich\`ele Vergne.
\newblock {\em The {W}eil representation, {M}aslov index and theta series},
  volume~6 of {\em Progress in Mathematics}.
\newblock Birkh\"{a}user, Boston, Mass., 1980.

\bibitem{Meyer_1972}
Yves Meyer.
\newblock {\em Algebraic numbers and harmonic analysis}.
\newblock North-Holland Publishing Co., Amsterdam-London; American Elsevier
  Publishing Co., Inc., New York, 1972.
\newblock North-Holland Mathematical Library, Vol. 2.

\bibitem{Mumford_1983}
David Mumford.
\newblock {\em Tata lectures on theta. {I}}, volume~28 of {\em Progress in
  Mathematics}.
\newblock Birkh\"auser Boston, Inc., Boston, MA, 1983.
\newblock With the assistance of C. Musili, M. Nori, E. Previato and M.
  Stillman.

\bibitem{Viola_2012a}
Joe Viola.
\newblock Resolvent estimates for non-selfadjoint operators with double
  characteristics.
\newblock {\em J. Lond. Math. Soc. (2)}, 85(1):41--78, 2012.

\bibitem{Viola_2017}
Joe Viola.
\newblock The elliptic evolution of non-self-adjoint degree-2 {H}amiltonians.
\newblock arXiv:1701.00801, 2017.

\end{thebibliography}

\end{document}